\documentclass[A4paper,reqno]{amsart}

\usepackage{amsmath, amssymb, latexsym, graphicx,tikz, stmaryrd,ifsym,enumerate,bbm}
\usepackage[latin1]{inputenc}
\usepackage{amssymb}

\usepackage{amsfonts}
\usepackage{amscd}
\usepackage{mathrsfs}
\usepackage[T1]{fontenc}
\usepackage[english]{babel}
\usepackage{marvosym}
\usepackage{amsthm}
\usepackage{multicol}
\usepackage[all]{xy}
\usepackage{todonotes}
\usepackage{graphicx,tikz}

\usepackage{pb-diagram}

\usepackage{bm}

\usepackage{color}

\newcounter{nameOfYourChoice}
\theoremstyle{plain}
\newtheorem{thm}{Theorem}[section]
\newtheorem{lem}[thm]{Lemma}

\newtheorem{prop}[thm]{Proposition}

\newtheorem{coro}[thm]{Corollary}

\theoremstyle{definition}
\theoremstyle{remark}
\newtheorem{rk}[thm]{Remark}
\newtheorem{df}[thm]{Definition}
\newtheorem{ex}[thm]{Example}
\newtheorem{setup}[thm]{Setup}

\numberwithin{equation}{section}

\usepackage{tikz-cd}
\newcommand\TikZ[1]{\begin{matrix}\begin{tikzpicture}#1\end{tikzpicture}\end{matrix}}

\def\ucp{{\mathbf{b}}}
\def\uc{{\mathbf{c}}}
\def\ub{{\mathbf{b}}}
\def\ui{{\mathbf{i}}}
\def\uj{{\mathbf{j}}}

\def\ut{{\mathbf{t}}}
\def\bfnu{\mathbf{\nu}}

\def\bbZ{\mathbb{Z}}

\def\frakS{\mathfrak{S}}

\def\calB{\mathcal{B}}
\def\calC{\mathcal{C}}

\def\calF{\mathcal{F}}

\def\calQ{\mathcal{Q}}

\def\calX{\mathcal{X}}

\def\frakm{\mathfrak{m}}

\def\bfa{\mathbf{a}}
\def\bfb{\mathbf{b}}
\def\bfc{\mathbf{c}}

\def\bfk{\mathbf{k}}

\def\bfQ{\mathbf{Q}}

\def\Icol{I_{\operatorname{col}}}
\def\Icolnu{I_{\operatorname{col}}(\nu,\bfQ)}

\def\Xint{\mathcal{X}}
\def\Xdomla{\mathcal{X}_\lambda^+}
\def\Xdom{\mathcal{X}^+}
\def\Drep{\op{D}}
\def\WW{\frakS}

\newcommand\restr[2]{{% we make the whole thing an ordinary symbol
  \left.\kern-\nulldelimiterspace % automatically resize the bar with \right
  #1 % the function
  \vphantom{\big|} % pretend it's a little taller at normal size
  \right|_{#2} % this is the delimiter
  }}

\def\homo{\operatorname{\it \mathscr{H}\kern-.25em om}}
\def\ext{\operatorname{\it \mathscr{E}\kern-.25em xt}}
\def\edo{\operatorname{\it \mathscr{E}\kern-.25em nd}}
\def\der{\operatorname{\it \mathscr{D}\kern-.25em er}}

\def\Hom{\mathrm{Hom}}

\def\End{\mathrm{End}}

\def\Id{\mathrm{Id}}

\def\Hfin{\op{H}^{\op{fin}}}
\def\opH{\op{H}}

\def\Heck{\op{H}_{d}(q)}

\def\cHeck{\widehat{\op{H}}_{\bfa}(q)}
\def\S{\op{S}_{d}(q)}

\def\SPol{\op{sP}_{d}}

\def\lHeck{\op{H}_{d,\bfQ}(q)}

\def\lHeckop{\op{H}_{d,\bfQ^{-1}}(q)}
\def\clHeck{\widehat{\op{H}}_{\bfa,\bfQ}(q)}
\def\cylHeck{\op{H}^\bfQ_{d,\bfQ}(q)}
\def\cyHeck{\op{H}_{d}^{\bfQ}(q)}

\def\PollHeck{\op{P}_{d,\bfQ}}
\def\cPolHeck{\widehat{\op{P}}_{\bfa}}
\def\cPollHeck{\widehat{\op{P}}_{\bfa,\bfQ}}

\def\lS{\op{S}_{d,\bfQ}(q)}

\def\lSbar{\op{\overline{S}}_{d,\bfQ}(q)}
\def\lSop{\op{S}_{d,\bfQ^{-1}}(q)}

\def\clS{\widehat{\op{S}}_{\bfa,\bfQ}(q)}
\def\lSPol{\op{sP}_{d,\bfQ}}
\def\clSPol{\widehat{\op{sP}}_{\bfa,\bfQ}}

\def\R{{R}_{\nu}}
\def\lR{{R}_{\nu,\bfQ}}
\def\cyR{{R}_{\nu}^{\bfQ}}
\def\cylR{{R}_{\nu,\bfQ}^{\bfQ}}
\def\cylRd{{R}_{d,\bfQ}^{\bfQ}}
\def\cR{\widehat{{R}}_{\nu}}
\def\clR{\widehat{{R}}_{\nu,\bfQ}}
\def\PolR{\it{Pol}_{\nu}}
\def\PollR{\it{Pol}_{\nu,\bfQ}}
\def\cPolR{\widehat{\it{Pol}}_{\nu}}
\def\cPollR{\widehat{\it{Pol}}_{\nu,\bfQ}}

\def\lA{{A}_{\nu,\bfQ}}
\def\clA{\widehat{{A}}_{\nu,\bfQ}}

\def\PolA{\it{sPol}_{\nu,\bfQ}}
\def\cPolA{\widehat{\it{sPol}}_{\nu,\bfQ}}

\def\cyHecka{\op{H}^\bfQ_{\bfa}(q)}
\def\cyHecknu{\op{H}^\bfQ_{\nu}(q)}
\def\cylHecka{\op{H}^\bfQ_{\bfa,\bfQ}(q)}

\def\cylS{\op{S}_{d,\bfQ}^{\bfQ}(q)}
\def\cylSnu{\op{S}_{\nu,\bfQ}^{\bfQ}(q)}

\def\cylSa{\op{S}_{\bfa,\bfQ}^{\bfQ}(q)}
\def\cylA{A_{\nu,\bfQ}^{\bfQ}}

\def\DJM{\op{S}_{d,\bfQ}^{\op{DJM}}(q)}
\def\DJMnu{\op{S}_{\nu,\bfQ}^{\op{DJM}}(q)}
\def\DJMop{\op{S}_{d,\bfQ^{-1}}^{\op{DJM}}(q)}

\def\p{\overrightarrow{p}}
\def\q{\overleftarrow{p}}

\def\op{\operatorname}

\usetikzlibrary{decorations.pathreplacing,backgrounds,decorations.markings,arrows}
\tikzset{wei/.style={draw=red,double=red!40!white,double distance=1.5pt,thin}}
\tikzset{bdot/.style={fill,circle,color=blue,inner sep=3pt,outer sep=0}}
\tikzset{dir/.style={postaction={decorate,decoration={markings,mark=at position .8 with {\arrow[scale=1.3]{>}}}}}}

\def\cPollHeck{\widehat{\op{P}}_{\bfa,\bfQ}}

\newcommand{\ba}{\mathbf{a}}

\newcommand{\la}{\lambda}

\begin{document}
\title[Higher level affine Schur and Hecke algebras]{Higher level affine Schur and Hecke algebras}%\newline\newline Alg\`ebres de Schur et alg\`ebres de Hecke affines de niveau sup\'erieur}
\author[R. Maksimau]{Ruslan Maksimau}
\address{R. M.: Institut Montpelli\'erain Alexander Grothendieck, Universit\'e de Montpellier, 34095 Montpellier, France}
\email{ruslan.maksimau@umontpellier.fr}

\author[C. Stroppel]{Catharina Stroppel}
\address{C. S.: Department of Mathematics, University of Bonn, 53115 Bonn, Germany}
\email{stroppel@math.uni-bonn.de}

%\maketitle

\begin{abstract}
We define a higher level version of the affine Hecke algebra and prove that, after completion, this algebra is isomorphic to a completion of Webster's tensor product algebra of type A. We then introduce a higher level version of the affine Schur algebra and establish, again after completion, an isomorphism with the quiver Schur algebra. An important observation is that the higher level affine Schur algebra surjects to the Dipper-James-Mathas cyclotomic $q$-Schur algebra. Moreover, we give nice diagrammatic presentations for all the algebras introduced in this paper. 

%\medskip
%\hspace{-1.5em}{\scshape{R\'esum\'e}}.
%On d\'efinit une version de niveau sup\'erieur $\ell$ de l'alg\`ebre de Hecke affine. On d\'emontre que le compl\'et\'e de cette alg\`ebre est isomorphe au compl\'et\'e de l'alg\`ebre produit tensoriel de Webster de type A. Ensuite, on introduit une version de niveau $\ell$ de l'alg\`ebre de Schur affine. On construit un isomorphisme entre le compl\'et\'e de cette alg\`ebre et le compl\'et\'e de l'alg\`ebre de carquois-Schur. Il est remarquable qu'il existe une surjection de l'alg\`ebre de Schur affine de niveau $\ell$ dans l'alg\`ebre de $q$-Schur cyclotomique de  Dipper-James-Mathas. On donne aussi une pr\'esentation diagrammatique agr\'eable de chaque alg\`ebre consid\'er\'ee dans l'article.
\end{abstract}
\maketitle
 
\vspace{-5mm}
\tableofcontents
\vspace{-1cm}

\section*{Introduction}

Let $\bfk$ be an algebraically closed field. Fix $q\in \bfk$,  $q\ne 0,1$ and $d\in\bbZ_{\geqslant 0}$. We study in this paper several versions of Hecke and Schur algebras of type $A$ including in particular a new higher level affine Schur algebra. 
\subsection*{Hecke algebras and their Schur versions}
To introduce the players,  let $\Hfin_d(q)$ be the ordinary Hecke algebra of rank $d$ over the field $\bfk$ (i.e.,\ $\Hfin_d(q)$ is a $q$-deformation of the group algebra $\bfk\frakS_d$ arising from the convolution algebra of complex valued functions on the finite group $\op{GL}_d(\mathbb{F}_q)$ which are constant on double cosets for a chosen Borel subalgebra). Let $\Heck$ be its (extended) affine version, that means it equals $\Hfin_d(q)\otimes \bfk[X_1^{\pm 1},\ldots,X_d^{\pm 1}]$ as a vector space and with a certain multiplication such that both tensor factors are subalgebras. It naturally arises from the convolution algebra of compactly supported functions defined on the p-adic group $\op{GL}_d(\mathbb{Q}_q)$ which are constant on double cosets for an Iwahori subalgebra. These algebras play a crucial role in p-adic representation theory, see e.g. \cite{Bushnelletal}, \cite{KLHecke}.

%\begin{figure}
%\hspace{-4mm}
%\begin{tabular}{|c|c|c|}
%  \hline
%    {Hecke family} & {KLR family}  \\
%   \hline
%  \hline
%   affine Hecke algebra $\Heck$ & KLR algebra $\R$ \\
%  (affine, no level, not Schur) & (affine, no level, not Schur) \\
%  \hline
%  cyclotomic Hecke algebra $\cyHeck$ & cyclotomic KLR algebra $\cyR$\\
%  (cyclotomic,  no level, not Schur) & (cyclotomic, no level, not Schur) \\
%  \hline
%  affine Schur algebra $\S$ & quiver Schur algebra $A_\nu$\\
%  (affine, no level, Schur) & (affine, no level, Schur) \\
%  \hline
%  an idempotent truncation of $\DJM$ & an idempotent truncation of $\cylA$\\
%  (cyclotomic, no level, Schur) & (cyclotomic, no level, Schur) \\
%  \hline
%  ??? & tensor product algebra $\lR$ \\
%  (affine, not Schur, higher level) & (affine, not Schur, higher level) \\
%  \hline
%  ???& cyclotomic tensor product algebra $\cylR$\\
% (cyclotomic, higher level, not Schur) & (cyclotomic, higher level, not Schur) \\
%  \hline
%  ???& higher level quiver Schur algebra $\lA$\\
%  (affine, higher level, Schur) & (affine, higher level, Schur) \\
%  \hline
%  Dipper-James-Mathas cyclotomic  & cyclotomic higher level \\
%  $q$-Schur algebra $\DJM$ &quiver Schur algebra $\lA$\\
%  (cyclotomic, higher level, Schur) & (cyclotomic, higher level, Schur) \\
%  \hline
%\end{tabular}
%\caption{The players}
%\label{theplayers}
%\vspace{-0.5cm}
%\end{figure}

The algebra $\Heck$ has a family of remarkable finite dimensional quotients $\cyHeck$, called \emph{cyclotomic Hecke algebras} or {\it Ariki-Koike algebras} which are deformations of the group algebra $\bfk (\frakS_d\ltimes (\bbZ/\ell\bbZ)^d)$. These algebras are well-studied objects in representation theory. For an excellent overview we refer to \cite{Mathascycl}.

The Dipper-James-Mathas cyclotomic $q$-Schur algebra $\DJM$ was defined in \cite{DJM} in the following way. For each $\ell$-composition $\lambda$ of $d$ they construct some element $m_\lambda$ in $\cyHeck$. Then they define the algebra $\DJM$ as an endomorphism algebra of the right $\cyHeck$-module $\bigoplus_\lambda m_\lambda \cyHeck$. We would like to define an affine version $\lS$ of the algebra $\DJM$ such that 
\begin{itemize}
\item $\lS$ has a nice faithful polynomial representation, and
\item $\lS$ surjects to $\DJM$.
\end{itemize}

So, we ask the following question.

\centerline{\it What should be the correct definition of the algebra $\lS$?}

One might expect that the affine version $\lS$ can be defined similarly as an endomorphism algebra of the $\Heck$-module $\bigoplus_\lambda m_\lambda \Heck$. However, this approach does not work. The reason is that in the cyclotomic case, some polynomials appear in the definition of the element $m_\lambda$. These polynomials play an important role for the structure of the $\cyHeck$-module $m_\lambda\cyHeck$. But in the affine case, these polynomials do nothing with the $\Heck$-module $m_\lambda \Heck$. So, the $\Heck$-modules $m_\lambda \Heck$ becomes quite boring.

However, this approach is known to work in the "no level" case: the (no level) affine $q$-Schur algebra is defined in \cite{vigneras} as the endomorphism algebra of an $\Heck$-module, very much in parallel to \cite{DJM}. At the same time, the cyclotomic $q$-Schur algebra $\DJM$ is defined for higher levels.  Our goal, is to give a higher level version $\lS$ of the affine $q$-Schur algebra $\S$. The existence of such an algebra seems to be natural from the analogy with the KLR algebras. Indeed, the affine higher level Schur version of the KLR algebra (the higher level quiver Schur algebra) is defined in \cite{SW}. However, the definition in \cite{SW} is purely geometric (as usually happens for KLR-like algebras), while the definitions of the Hecke-like algebras are algebraic. So the definition of the higher level quiver Schur algebra in \cite{SW} does not tell us what the definition of the higher level affine $q$-Schur algebra should be.

Finally, we define the higher level affine $q$-Schur algebra $\lS$ in the following way: the definition is in two steps. First, we define the higher level version $\lHeck$ of $\Heck$ by generators and relations. After that, we define $\lS$ as the endomorphism algebra of some $\lHeck$-module. As we explained before, in the cyclotomic case, the $q$-Schur algebra $\DJM$ is defined in \cite{DJM} in one step from $\cyHeck$. However, in the affine case, there is no known direct way to define $\lS$ from $\Heck$. This is probably the reason why the algebra $\lS$ was not known before. 

One more important point is to define the polynomial representation of $\lS$. This is easy for the KLR-like algebras because the polynomial representations appears naturally from geometry. On the other hand, the construction of the polynomial representation of $\S$ is via long and difficult computations. We don't want to follow this approach, but instead give a more conceptual argument. We construct a polynomial representation of $\lS$ as a subrepresentation of the defining representation of $\lS$. 

We believe that our methods can be transferred to the construction and study of other types of Schur algebras. Although we stick to a very special class of algebras in this paper, our approach seems to work in much more generality (including the case of Clifford-Hecke algebras, \cite{NazarovCliffHecke} or affine zigzag algebras, \cite{KlMuth}).

\subsection*{KLR algebras and their Schur versions}
Around 10 years ago, Khovanov-Lauda \cite{KL} and Rouquier \cite{Rou2KM} introduced the \emph{quiver Hecke algebra} (also called \emph{KLR algebra}) $\R$. Again it arises from a convolution algebra structure, but now on the Borel-Moore homology of a Steinberg type variety defined using the moduli space of isomorphism classes of flagged representations of a fixed quiver with dimension vector $\nu$, \cite{VV}, \cite{Rou2KM}. 
The major interest in these algebras is due to the fact that they are naturally graded and are used to categorify the negative part of a quantum group. This holds in particular for the finite or affine type $A$ versions; the algebras arise in several categorification results on the level of 2-morphisms. They were recently also used to approach modular representation theory of general linear groups, \cite{RW}. These KLR  algebras again have a family of interesting (finite dimensional) quotients $\cyR$ (called \emph{cyclotomic KLR algebras}). Apart from being interesting on their own, these quotients $\cyR$ categorify simple modules over the before-mentioned quantum group, \cite{LV}, but also give concrete descriptions of categories arising in geometric and super representation theory.  

A higher level version $\lR$ of the KLR algebra (called \emph{tensor product algebra}) was introduced by Webster \cite{Webster}. The cyclotomic quotient $\cylR$ of the algebra $\lR$ categorify tensor products of simple modules over a quantum group.

Let us give an overview on connections between these algebras. The cyclotomic Hecke algebra $\cyHeck$ has a block decomposition $\cyHeck=\bigoplus_\nu \cyHecknu$. Brundan and Kleshchev  constructed in \cite{BKKL} an isomorphism between the block $\cyHecknu$ and the cyclotomic KLR algebra $\cyR$ of type $A$. A different proof of this isomorphism was given by Rouquier in \cite{Rou2KM}  as a consequence of an isomorphism between (an idempotent version of) a localization of $\Heck$ and a localization of $\R$. It is also possible to give a similar proof, using completions instead of localizations, see \cite{Webstergraded}, \cite{MS}.  (The completion/localization of $\Heck$ depends on $\nu$.) 

To understand the relation between the parameters of the Hecke and KLR algebras note that the Hecke algebra $\Heck$ depends on $q\in\bfk\backslash\{0,1\}$, and the cyclotomic quotient $\cyHeck$ of $\Heck$ furthermore on an $\ell$-tuple $\bfQ=(Q_1,\ldots,Q_\ell)\in(\bfk^*)^\ell$. On the other hand, the KLR algebra $\R$ depends on a quiver $\Gamma$ and on a dimension vector $\nu$ for $\Gamma$. The cyclotomic quotient $\cyR$ of $\R$ depends also on an $\ell$-tuple $\bfQ=(Q_1,\ldots,Q_\ell)$ of vertices of $\Gamma$. To describe the blocks $\cyHecknu$ of $\cyHeck$ in terms of KLR algebras, we have to take the quiver $\Gamma=\Gamma_\calF$ as in Section~\ref{subs-isom_lHeck-tens-compl}. In particular, this choice of $\Gamma$ allows us to consider $\bfQ\in(\bfk^*)^\ell$ as an $\ell$-tuple of vertices of the quiver., see \eqref{Corona}. For this choice of $\Gamma$ we have then the isomorphism $\cyHecknu\simeq \cyR$ from \cite{BKKL},  \cite{Rou2KM}.

The second author and Webster defined in \cite{SW} the \emph{quiver Schur algebra} $A_\nu$ (that is a Schur version of the KLR algebra $\R$) and its generalizations, the higher level quiver Schur algebras $\lA$ together with a family of cyclotomic quotients $\cylA$. Moreover, in \cite{SW}, the isomorphism $\cyHecknu\simeq \cyR$ was extended to an isomorphism $\DJMnu \simeq \cylA$, where $\DJM=\bigoplus_{\nu}\DJMnu$ is the Dipper-James-Mathas cyclotomic $q$-Schur algebra (that is the Schur version of $\cyHeck$). 

On the other hand, an affine (no level) version of the isomorphism $\cylSnu\simeq \cylA$ was constructed by Miemietz and the second author, \cite{MS}. It was proved in \cite[Thm.~9.7]{MS} that a completion of the affine Schur algebra $\S$ (the completion depends on $\nu$) is isomorphic to a completion of the quiver Schur algebra $A_\nu$. 

\subsection*{The zoology}
The zoology of the algebras discussed above can be grouped into two big families: 
\begin{center}
{\it the Hecke family} and {\it the KLR family}. 
\end{center}
An algebra in either family can be
\begin{center}
affine or cyclotomic, \quad higher level or no level,\quad Schur or not Schur.
\end{center}

We briefly describe all the possible cases. 

\begin{enumerate}[(i)]
\item {\bf No level, not Schur, cyclotomic.}
\newline
The algebra in the Hecke family is the cyclotomic Hecke algebra $\cyHeck$, its analogue in the KLR family is the cyclotomic KLR algebras $\cyR$. The isomorphism between a block of the algebra $\cyHeck$ and the algebra $\cyR$ is due to Brundan-Kleshchev \cite{BKKL} and Rouquier \cite{Rou2KM}.

\item{\bf No level, not Schur, affine. }
\newline
The algebra in the Hecke family is the affine Hecke algebra $\Heck$, its analogue in the KLR family is the (affine) KLR algebra $\R$. We have surjections $\Heck\to \cyHeck$ and $\R\to\cyR$. However the isomorphism between a block of $\cyHeck$ and $\cyR$ does in general not lift to the affine level. There are however isomorphisms after suitable completions of $\Heck$ and of $\R$ (where the completion of $\Heck$ depends on $\nu$), \cite{Webstergraded}, \cite{MS}. A similar construction using localizations instead of completions was already given in \cite{Rou2KM}. 
\setcounter{nameOfYourChoice}{\value{enumi}}
\end{enumerate}
We observe a difference between the cyclotomic case and the affine case:\\For the cyclotomic case, a block of the algebra in the Hecke family is isomorphic to the algebra in the KLR family. In the affine case, the completion of the algebra in the KLR family is isomorphic to the completion of the algebra in the KLR family. We will see that exactly the same thing happens in all the remaining cases below. 
\begin{enumerate}[(i)]
\setcounter{enumi}{\value{nameOfYourChoice}}
\item {\bf Higher level, Schur, cyclotomic. } \label{3}
\newline
The algebra in the Hecke family is the cyclotomic Dipper-James-Mathas $q$-Schur algebra  $\DJM$ from \cite{DJM}.  Its analogue in the KLR family is the cyclotomic quiver Schur algebra $\cylA$ defined in \cite{SW}. It is proved in \cite{SW} that each block of the algebra $\DJM$ is isomorphic to the algebra $\cylA$ for some $\nu$.

\item {\bf No level, Schur, cyclotomic.}
\newline
These algebras have no special names. They (and the corresponding isomorphisms) can be obtained as idempotent truncations of the algebras in \eqref{3}.

\item {\bf No level, Schur, affine.}
\newline
The algebra in the Hecke family is the affine $q$-Schur algebra $\S$ from e.g. \cite{Greenaff},  \cite{vigneras}. Its analogue in the KLR family is the (no level, affine) quiver Schur algebra defined in \cite{SW}. It is proved in \cite{MS} that the algebras $\S$ and $A_\nu$ are isomorphic after completion. (The completion of $\S$ depends on $\nu$.)  
\setcounter{nameOfYourChoice}{\value{enumi}}
\end{enumerate}
 A construction of the algebras in the Hecke families of the remaining cases is done in our paper.  All our constructions do not make any assumptions on the characteristic of the underlying field. 
 \begin{enumerate}[(i)]
\setcounter{enumi}{\value{nameOfYourChoice}}
\item {\bf Higher level, not Schur, affine.}\label{six}
\newline
The algebra in the KLR family is Webster's tensor product algebra $\lR$, \cite{Webster}. 
\setcounter{nameOfYourChoice}{\value{enumi}}
\end{enumerate}
We define the Hecke analogue $\lHeck$ of $\lR$, called the \emph{higher level affine Hecke algebra},  by generators and relations (algebraically and diagrammatically) in Section~\ref{sec-lHeck} and then prove in Theorem~\ref{thm-isom-lHeck-tens-comp} that the algebras $\lHeck$ and $\lR$ are isomorphic after completions. On the Hecke side this is with respect to maximal ideals of the centre which we describe in Proposition~\ref{lem-cen-Hdl}.

After having finished writing this paper, we were informed that Webster had defined already a similar algebra with an analogous isomorphism result in \cite[Sec.~4]{Webstergraded}.

\begin{enumerate}[(i)]
\setcounter{enumi}{\value{nameOfYourChoice}}
\item {\bf Higher level, not Schur, cyclotomic.}\label{seven}
\newline
The algebra in the KLR family is here the cyclotomic quotient $\cylR$ of the tensor product algebra $\lR$ defined in \cite{Webster}. 
\setcounter{nameOfYourChoice}{\value{enumi}}
\end{enumerate}
The Hecke analogue of $\cylR$ is a similar quotient $\cylHeck$, see Definition~\ref{defcylHeck}, of the algebra $\lHeck$. We prove that each block of the algebra $\cylHeck$ is isomorphic to the algebra $\cylR$ for some $\nu$, see Theorem~\ref{prop-isom-lHeck-tens-cycl}. As a byproduct we can determine in Corollary~\ref{coro-eigenv_lHeck} the possible eigenvalues of the Laurent polynomial algebra acting on the regular representation $\cylHeck$, which is a well-known fact for ordinary cyclotomic Hecke algebras, \cite[Prop.~3.7]{JamesMathas}. It is important to know the possible eigenvalues to be able to say that the corresponding algebra $\cylR$ is defined with respect to the quiver $\Gamma_\calF$ as in Section~\ref{subs-isom_lHeck-tens-compl}. In particular, this is a quiver of type $A$. 

\begin{enumerate}[(i)]
\setcounter{enumi}{\value{nameOfYourChoice}}
\item {\bf Higher level, Schur, affine.}\label{eight}
\newline
The algebra in the KLR family is the (affine higher level) quiver Schur algebra $\lA$, defined in \cite{SW}.
\end{enumerate}
We define a Hecke analogue $\lS$ of this algebra in Section~\ref{sec-Schur}, and prove (Theorem \ref{thm-isom-qS-QS-comp}) that the algebras $\lS$ and $\lA$ are isomorphic after completion (the completion of $\lS$ depends on $\nu$). As an important tool, which we feel is important on its own, we construct in Corollary~\ref{prop-polrep-lS} a nice polynomial representation involving partially symmetric polynomials. 

 Altogether, this completes the construction of the  Hecke families in all cases together with the corresponding isomorphism theorems. 
 
All these algebras arise as algebras (or quotient algebras in the cyclotomic case) of morphisms in some monoidal category. It is the  {\it universal higher level category}, Definition~\ref{univtensor}, in the not Schur cases and the  {\it universal thickened higher level category}, Definition~\ref{Corona2},  in the Schur cases.\footnote{The  former category is in fact a subcategory of the latter, but it is not a full subcategory and to make the embedding compatible with the relations imposed later on, one needs to chose it in a non-obvious way relying on Lemma~\ref{lem-black_cross}.}  Both categories are generated on the level of objects by sets $I_b$ and $I_r$, but for the definition of the algebras the set $I_r$ is only involved in the higher level cases. These monoidal categories allow a diagrammatic approach for all the involved algebras and in the non-Schur case diagrammatic presentations.

\subsection*{The structure of the paper} 
The definition of the higher level affine Hecke algebra $\lHeck$ by generators and relations (algebraically and diagrammatically) can be found in Section~\ref{sec-lHeck}. Next, in Section~\ref{sec-KLR_tens} we construct an isomorphism between a completion of $\lHeck$ (this completion depends on $\nu$) and a completion of $\lR$. To do this, we use the same strategy as in \cite{MS} (namely the identification of faithful polynomial representations). This is also very much analogous to \cite{Webstergraded}, where similar algebras were introduced, but from a different point of view. Webster's approach is via weighted KLR algebras, whereas our focus is on Schur algebras. In particular our approach is guided by giving a method how to {\it Schurify} different types of Hecke algebras.  A cyclotomic version of the isomorphism between the completions of $\lHeck$ and $\lR$ follows easily from the affine version. This is done in Section~\ref{sec-cycl-quot} and completes case \eqref{seven}. Section~\ref{sec-Schur} contains the definition of the {\it higher level affine Schur algebra} $\S$, the Hecke analogue $\lS$ of the higher level quiver Schur algebra $\lA$.  The isomorphism between a completion of $\lS$ (this completion depends on $\nu$) and a completion of $\lA$ can be found in Section~\ref{sec-QSchur} - completing Case \eqref{eight}.

%\subsection*{Focus on Schur algebras} The Schur algebras are in the focus and guiding principle of all constructions. We believe that the methods can be transferred to the construction and study of other types of Schur algebras. Although we stick to a very special class of algebras in this paper, our approach seems to work in much more generality (including the case of Clifford-Hecke algebras, \cite{NazarovCliffHecke} or affine zigzag algebras, \cite{KlMuth}). 

%{\color{red}
%We also like to stress the connection to the classical Dipper-James-Mathas cyclotomic $q$-Schur algebra. It was defined in \cite{DJM} as an endomorphism algebra of some $\cyHeck$-module. One might expect that its affine version $\lS$ can be defined similarly as an endomorphism algebra of some $\Heck$-module. However, it seems that this does not work. Instead, we define the algebra $\lS$ in two steps. First, we define the higher lever version $\lHeck$ of $\Heck$, and then $\lS$ as the endomorphism algebra of some $\lHeck$-module. The Schur algebras $\S$ and $\lS$ considered in the paper are in fact $q$-Schur algebras. Moreover, the Dipper-James-Mathas cyclotomic $q$-Schur algebra is a quotient of our algebra $\lS$.
%}

\subsection*{\it Acknowledgments}
We thank Alexander Kleshchev for sharing ideas that simplified the construction of polynomial representations of the affine Schur algebras, and also the referee for useful comments. R. M.  is grateful for the support and hospitality of the MPI for Mathematics in Bonn, where a big part of this work is done.

\subsubsection*{Conventions} We fix as ground field an algebraically closed field $\bfk$ and denote $\bfk^*=\bfk-\{0\}$. All  vector spaces, linear maps,  tensor products etc. are taken over $\bfk$ if not otherwise specified.  For $a,b\in\bbZ$ with  $a\leqslant b$ we abbreviate $[a;b]=\{a,a+1,\ldots,b-1,b\}$. For $d\in\bbZ_{\geqslant 0}$ we denote by $\frakS_d$ the symmetric group of order $d!$ with length function ${l}$.

\section{Higher level affine Hecke algebras ${\lHeck}$}
\label{sec-lHeck}
\begin{setup}
\label{set-up}
We fix $q\in \bfk^*$, $q\ne 1$ and integers $d\geqslant 0$, $\ell\geqslant 0$ called {\it rank} and {\it level}, and {\it parameters}  $\bfQ=(Q_1,\ldots,Q_\ell)\in(\bfk^*)^\ell$. We denote $J=\{0,1\}$ and call its elements {\it colours} with $0$ viewed as {\it black} and $1$ viewed as {\it red}. 
\end{setup}

\subsection{The algebraic version}
In this section we introduce the main new player, a higher level version of the affine Hecke algebra.
\begin{df}
Let $J^{\ell,d}\subset J^{\ell+d}$ be the set of $(\ell+d)$-tuples $\uc=(c_1,\hdots,c_{\ell+d})$ such that $\sum_{i=1}^{\ell+d}c_i=\ell$  (i.e., the tuples containing $d$ black and $\ell$ red elements). Let $\frakS_{\ell+d}$ act on $J^{\ell,d}$ by permuting the entries of the tuple in the way such that $\pi(\uc)_m=\uc_{\pi^{-1}(m)}$ for $\pi\in \frakS_{\ell+d}$.
\end{df}

\begin{df}
\label{def-extaffHecke_alg}
The  $\ell$-\emph{affine Hecke algebra} ${\lHeck}$ is the $\bfk$-algebra generated by $e(\uc)$ for $\uc=(c_1,\ldots,c_{\ell+d})\in J^{\ell,d}$, $T_r$ for $r\in[1;\ell+d-1]$ and $X_j$, $X_j'$  for $j\in[1;\ell+d]$, subject to  the following defining relations
\begin{eqnarray}
\sum_{\uc\in J^{\ell,d}}e(\uc)=1,&{and}&e(\uc)e(\uc)=e(\uc),\label{lHecke1}
\end{eqnarray}
\vspace{-0.5cm}
\begin{eqnarray}
X_ie(\uc)=X'_ie(\uc)=0 &&\mbox{ if }c_i=1,\label{rel_Xe=0}\label{lHecke2new}\\
X_iX'_ie(\uc)=X'_iX_ie(\uc)=e(\uc)&&\mbox{ if }c_i=0,\label{lHecke7}
\end{eqnarray}
\vspace{-0.5cm}
\begin{eqnarray}
X_ie(\uc)=e(\uc)X_i,&{and}&
X_i'e(\uc)=e(\uc)X_i',\label{lHecke4new}\\
X_iX_j=X_jX_i,&{and}&
X'_iX'_j=X'_jX'_i,\label{lHecke5new}\\
T_{r}T_{s}=T_{s}T_{r} \mbox{ if }|r-s|>1, &{and}&
T_rX_i=X_iT_r~\mbox{ if }|r-i|>1,\\
T_re(\uc)=0 \mbox{ if } c_r=c_{r+1}=1,&and&T_re(\uc)=e(s_r(\uc))T_r,\label{rel_Te=0}
\end{eqnarray}
\vspace{-0.5cm}
\begin{eqnarray}
(T_rX_{r+1}-X_rT_r)e(\uc)&=&
\begin{cases}
(q-1)X_{r+1} &\mbox{ if }c_r=c_{r+1}=0,\\
0 &\mbox{ else},
\end{cases}\\
(T_rX_{r}-X_{r+1}T_r)e(\uc)&=&
\begin{cases}
-(q-1)X_{r+1}& \mbox{ if }c_r=c_{r+1}=0,\\
0& \mbox{ else},
\end{cases}
\end{eqnarray}
\begin{eqnarray}
T_r^2e(\uc)&=&
\begin{cases}
(q-1)T_re(\uc)+qe(\uc)& \mbox{ if }c_r=c_{r+1}=0,\\
\left(X_r-Q_{\sum_{j=1}^{r+1}c_j}\right)e(\uc) &\mbox{ if }c_r=0, c_{r+1}=1,\\
\left(X_{r+1}-Q_{\sum_{j=1}^{r}c_j}\right)e(\uc) &\mbox{ if }c_{r+1}=0, c_{r}=1,\\
\end{cases}
\quad\quad
\end{eqnarray}
\begin{eqnarray}
&&(T_{r}T_{r+1}T_{r}-T_{r+1}T_{r}T_{r+1})e(\uc)\nonumber\\
&=&
\begin{cases}
0 &\text{if $c_{r+1}=0$ and $r<\ell+d-1\quad$}\\
(1-q)X_{r+2}e(\uc) &\text{if $c_{r+1}=1$, $c_r=c_{r+2}=0$ and $r<\ell+d-1$\quad\quad}
\end{cases}
\quad
\end{eqnarray}
\normalsize
where $i,j$ run through $[1;\ell+d]$ and $r,s$ through $[1;\ell+d-1]$.
\end{df}

\begin{rk}
\label{ordinaryHecke}
In case $\ell=0$ (i.e., $\bfQ=0\in \bfk^0$) the set $J^{\ell,d}\subset J^{\ell+d}$ contains a single element ${\bf c}=(0,0,\ldots,0)$. Then $e(\uc)=1$ by \eqref{lHecke1} and $X_r'=X_r^{-1}$ by \eqref{lHecke7}, and the algebra $\lHeck$ is nothing else than the ordinary (extended) affine Hecke algebra $\Heck$, see e.g. \cite{Kirillovlect}, in the normalization from e.g. \cite{MS}. If additionally $q=1$ then we get the smash product algebra $\bfk[S_d]\#\bfk[X_1^{\pm1},\ldots,X_d^{\pm1}]$. 

Moreover it contains the ordinary finite dimensional Hecke algebra $\Hfin_d(q)$ attached to $\frakS_d$ as subalgebra generated by the $T_r$ for $r\in[1;\ell+d-1]$.
\end{rk}

\subsection{The diagrammatic version}
We introduce a diagrammatic calculus generalizing the usual permutation diagrams of the symmetric group.
It provides a convenient way to display elements in the higher level affine Hecke algebra and is done by realizing algebras of homomorphisms in some monoidal category.

\begin{figure}[t]
\begin{eqnarray}
\label{Hecke-diag-1}
  \begin{tikzpicture}[scale=0.6, thick]  
      \draw (-2,0)  +(0,-1) .. controls (-0.4,0) ..  +(0,1);
       \draw (-0.4,0)  +(0,-1) .. controls (-2,0) ..  +(0,1);
           \node at (2.5,0) {\Large $=\;\;(q-1)$};
         %\node at (1,0) {$(q-1)$};  
    \draw(6,0) +(-1,-1) -- +(1,1);
  \draw(6,0) +(1,-1) -- +(-1,1);
           \node at (8.5,0) {\Large $+\;\;q$};
         % \node at (5.2,0) {$\Large{q}$};
    \draw (11.5,0)  +(0,-1) -- +(0,1);
       \draw (10,0)  +(0,-1) -- +(0,1);
   \end{tikzpicture}
\end{eqnarray}
\begin{equation}
\label{Hecke-diag-2}
    \begin{tikzpicture}[scale=0.6, thick]
      \draw (-3,0)  +(1,-1) -- +(-1,1);
      \draw (-3,0)  +(0,-1) .. controls (-4,0) ..  +(0,1);
      \draw (-3,0) +(-1,-1) -- +(1,1);
      \node at (-1,0) {\Large $=$};
      \draw (1,0)  +(1,-1) -- +(-1,1);
  \draw (1,0)  +(0,-1) .. controls (2,0) ..  +(0,1);
      \draw (1,0) +(-1,-1) -- +(1,1);    
\end{tikzpicture}
\end{equation}
\begin{equation}
\label{Hecke-diag-3}
    \begin{tikzpicture}[scale=0.6, thick]
  \draw(-8,6) +(-1,-1) -- +(1,1);
  \draw(-8,6) +(1,-1) -- +(-1,1);
\fill (-7.5,5.5) circle (5pt);
\node at (-6,6) {\Large $=$};
 \draw(-4,6) +(-1,-1) -- +(1,1);
  \draw(-4,6) +(1,-1) -- +(-1,1);
\fill (-4.5,6.5) circle (5pt);
\node at (-0.5,6) {\Large $+\;\;(q-1)$};
%\node at (3.7,6) {$(q-1)$};
\draw (2,6)  +(0,-1) -- +(0,1);
       \draw (3.5,6)  +(0,-1) -- +(0,1);
       \fill (3.5,6) circle (5pt);
    \end{tikzpicture}    
\end{equation}
\begin{equation}
\label{Hecke-diag-4}
     \begin{tikzpicture}[scale=0.6, thick]
  \draw(-8,6) +(-1,-1) -- +(1,1);
  \draw(-8,6) +(1,-1) -- +(-1,1);
\fill (-7.5,6.5) circle (5pt);
\node at (-6,6) {\Large $=$};
 \draw(-4,6) +(-1,-1) -- +(1,1);
  \draw(-4,6) +(1,-1) -- +(-1,1);
\fill (-4.5,5.5) circle (5pt);
\node at (-0.5,6) {\Large $+\;\;(q-1)$};
%\node at (1.7,6) {$(q-1)$};
\draw (2,6)  +(0,-1) -- +(0,1);
       \draw (3.5,6)  +(0,-1) -- +(0,1);
       \fill (3.5,6) circle (5pt);
    \end{tikzpicture}
\end{equation}
\caption{Affine Hecke algebra relations}
\end{figure}

\begin{df} 
\label{univtensor}
 Let $I_b$ and $I_r$ be sets not both empty. The {\it universal higher level category} corresponding to this pair is the $\bfk$-linear strict monoidal category generated as monoidal category by  objects $i\in I_b$, called {\it black labels}, and objects $Q\in I_r$, called {\it red labels}, and by morphisms (for any $i,j\in I_b$, $Q\in I_r$) 
\begin{equation*}
\label{Morphmonoidal}
\TikZ{[thick,scale=.5] \draw (0,0) node{} to (1,1) node{} (1,0) node{} to (0,1)node{}; \node at (0,-0.3) {\tiny i}; \node at (1,-0.3) {\tiny j};}:\;{i}\otimes{j}\longrightarrow{j}\otimes{i},\quad
\TikZ{[thick,scale=.5] \draw  (1,1) node{} to (0,0)node{};\draw[wei] (1,0) node{} to (0,1)node{}; \node at (1.2,-0.3) {\tiny Q};\node at (0,-0.3) {\tiny i}; }
:\;{i}\otimes {Q}\longrightarrow{Q}\otimes {i},\quad
\TikZ{[thick,scale=.5] \draw  (1,0) node{} to (0,1)node{};\draw[wei] (1,1) node{} to (0,0)node{}; \node at (-0.2,-0.3) {\tiny Q};, \node at (1,-0.3) {\tiny i}; }
:\;{Q}\otimes {i}\longrightarrow{i}\otimes {Q}\end{equation*}
called {\it crossings}, and 
$\TikZ{[thick, scale=.5] \draw (0,0) node{} to (0,1) node{} (0,0.5) node[fill,circle,inner sep=1.5pt]{};\node at (0,-0.3) {\tiny i};}: {i}\longrightarrow {i}$ called {\it dot morphisms}. The diagrams  $\TikZ{[thick, scale=.5] \draw (0,0) node{} to (0,1) node{} (0,0.5);\node at (0,-0.3) {\tiny i};}$ and  $\TikZ{[thick, scale=.5] \draw[wei] (0,0) node{} to (0,1) node{} (0,0.5);\node at (0,-0.3) {\tiny Q};}$ depict the identity for the black label $i$ respectively for the red label $Q$. 
%The diagram $\TikZ{[scale=.45] \draw
%		(0,0) node{} to (0,0.5){}
%                  (0,0.5) node[fill,circle,inner sep=1.5pt]{}
%		(-0.75,0.5) node{${k}$}
%		;}$ denotes the $k$-fold (vertical) composition of the dot morphism. 
\end{df}
The tensor product of morphisms is displayed by placing them horizontally next to each other, whereas for the composition of morphisms we place them vertically. We omit indicating labels which can be arbitrary from $I_b$ or $I_r$ except that the colour has to match the  colour of the strand. Note that by definition two red strands never cross.
\begin{df}
\label{diag} 
Consider the universal higher level category attached to a pair $I_b$ and $I_r$.
Then an {\it $(\ell,d)$-diagram}  is a morphism  between two objects, that (both) are the tensor product of exactly $d$ black labels and $\ell$ red labels, which is a finite composition of tensor products of generating morphisms. 
\end{df}
We observe that for the existence of such a diagram the multiset of black labels and the sequence of red labels for the two involved objects must agree.

\begin{df} 
Let $|I_b|=1$ and $I_r=\bfk^*$. We then define 
\label{def-extaffHeck_diag} 
\begin{enumerate}
\item the {\it higher level affine Hecke category} as the universal higher level category  modulo the {\it affine Hecke algebra relations} \eqref{Hecke-diag-1}-\eqref{Hecke-diag-4} and the {\it higher affine Hecke algebra relations} \eqref{l-Hecke-diag-1}-\eqref{l-Hecke-diag-4} on morphisms; and
\item  the $\ell$-\emph{affine Hecke algebra} ${\lHeck}$ as the induced algebra structure on the vector space spanned by all $(\ell,d)$-diagrams with fixed red labels $\bfQ$ read from from left to right.
\end{enumerate}
\end{df}

\begin{figure}
\begin{equation}
\label{l-Hecke-diag-1}
  \begin{tikzpicture}[scale=0.7, thick,baseline=1.6cm]
    \draw (-2.8,0)  +(0,-1) .. controls (-1.2,0) ..  +(0,1);
       \draw[wei] (-1.2,0)  +(0,-1) .. controls (-2.8,0) ..  +(0,1) node[below,at start]{$Q$};
           \node at (-.3,0) {\Large $=$};
    \draw[wei] (2.8,0)  +(0,-1) -- +(0,1) node[below,at start]{$Q$};
       \draw (1.2,0)  +(0,-1) -- +(0,1);
       \fill (1.2,0) circle (5pt);
          \draw[wei] (-2.8,3)  +(0,-1) .. controls (-1.2,3) ..  +(0,1) node[below,at start]{$Q$};
          \node at (3.5,0) {\Large $-$};
          \node at (4.5,0) {\Large $Q$};
          \draw[wei] (6.8,0)  +(0,-1) -- +(0,1) node[below,at start]{$Q$};
          \draw (5.2,0)  +(0,-1) -- +(0,1);
  \draw (-1.2,3)  +(0,-1) .. controls (-2.8,3) ..  +(0,1);
           \node at (-.3,3) {\Large $=$};
    \draw (2.8,3)  +(0,-1) -- +(0,1);
       \draw[wei] (1.2,3)  +(0,-1) -- +(0,1) node[below,at start]{$Q$};
       \fill (2.8,3) circle (5pt);
       \node at (3.7,3) {\Large $-$};
       \node at (4.5,3) {\Large $Q$};
       \draw (6.8,3)  +(0,-1) -- +(0,1);
       \draw[wei] (5.2,3)  +(0,-1) -- +(0,1) node[below,at start]{$Q$};
  \end{tikzpicture}  
\end{equation}

\begin{equation}
\label{l-Hecke-diag-2}
    \begin{tikzpicture}[scale=0.6, thick]
  \draw(-3,6) +(-1,-1) -- +(1,1);
  \draw[wei](-3,6) +(1,-1) -- +(-1,1);
\fill (-3.5,5.5) circle (5pt);
\node at (-1.5,6) {\Large $=$};
 \draw(0,6) +(-1,-1) -- +(1,1);
  \draw[wei](0,6) +(1,-1) -- +(-1,1);
\fill (.5,6.5) circle (5pt);

  \draw[wei](5,6) +(-1,-1) -- +(1,1);
  \draw(5,6) +(1,-1) -- +(-1,1);
\fill (4.5,6.5) circle (5pt);
\node at (6.5,6) {\Large$=$};
 \draw[wei](8,6) +(-1,-1) -- +(1,1);
  \draw(8,6) +(1,-1) -- +(-1,1);
\fill (8.5,5.5) circle (5pt);
    \end{tikzpicture}
\end{equation}

\begin{equation}
\label{l-Hecke-diag-3}
\begin{tikzpicture}[scale=0.6, thick]
      \draw[wei] (-3,3)  +(1,-1) -- +(-1,1);
      \draw (-3,3)  +(0,-1) .. controls (-4,3) ..  +(0,1);
      \draw (-3,3) +(-1,-1) -- +(1,1);
      \node at (-1.5,3) {\Large $=$};
      \draw[wei] (0,3)  +(1,-1) -- +(-1,1);
  \draw (0,3)  +(0,-1) .. controls (1,3) ..  +(0,1);
      \draw (0,3) +(-1,-1) -- +(1,1);    
      
      \draw (5,3)  +(1,-1) -- +(-1,1);
      \draw (5,3)  +(0,-1) .. controls (4,3) ..  +(0,1);
      \draw[wei] (5,3) +(-1,-1) -- +(1,1);
      \node at (6.5,3) {\Large $=$};
      \draw (8,3)  +(1,-1) -- +(-1,1);
  \draw (8,3)  +(0,-1) .. controls (9,3) ..  +(0,1);
      \draw[wei] (8,3) +(-1,-1) -- +(1,1); 
\end{tikzpicture}      
\end{equation}

\begin{equation}
\label{l-Hecke-diag-4}
\begin{tikzpicture}[scale=0.6, thick]
      \draw (-5,3)  +(1,-1) -- +(-1,1);
      \draw[wei] (-5,3)  +(0,-1) .. controls (-6,3) ..  +(0,1);
      \draw (-5,3) +(-1,-1) -- +(1,1);
         \node at (-3,3) {\Large $=$};
      \draw (-1,3)  +(1,-1) -- +(-1,1);
      \draw[wei] (-1,3)  +(0,-1) .. controls (0.5,3) ..  +(0,1);
      \draw (-1,3) +(-1,-1) -- +(1,1);    
         \node at (2.5,3) {\Large $-\;\;(q-1)$};
      \draw (6,3)  +(-1,-1) -- +(-1,1);
      \draw[wei] (6,3)  +(0,-1)  --  +(0,1);
      \draw (6,3) +(1,-1) -- +(1,1);
      \fill (7,3) circle (5pt);
\end{tikzpicture}
\end{equation}
%\begin{center}
\hspace{1.5cm}(Omitted labels can be arbitrary, but of course fixed in each relation.)
%\end{center}
\label{lHeckepictures}
\caption{Additional relations in the $\ell$-affine Hecke algebra.}
\end{figure}

The following easy observation justifies our notation ${\lHeck}$.
\begin{lem}
\label{lem:twodefs}
The algebras in Definitions~\ref{def-extaffHecke_alg} and~\ref{def-extaffHeck_diag} are isomorphic. 
\end{lem}
\begin{proof}
One can easily verify by checking the relations that the following correspondence on generators defines an isomorphism of the two algebras. The idempotent $e(\uc)$ corresponds to the diagram with vertical strands with colours determined by the sequence $\uc$. The element $X_ie(\uc)$ (resp. $X'_je(\uc)$) such that $c_r$ is black corresponds to the diagram with vertical strands with colours determined by the sequence $\uc$ and a dot labelled by $1$ (resp. $-1$) on the strand number $i$ (counted from the left). (Note that $X_ie(\uc)$ and $X'_ie(\uc)$ are zero if $c_i$ is red by \eqref{rel_Xe=0}.)  The element $T_re(\uc)$ such that at least one of the colours $c_r$, $c_{r+1}$ is black corresponds to the diagram with the $r$-th and $(r+1)$th strand intersecting once and all other strands just vertical, with the colours on the bottom of the diagram determined by $\uc$. (By \eqref{rel_Te=0} we have $T_re(\uc)=0$ if $c_r=1=c_{r+1}$.) 
\end{proof}

The usual affine Hecke algebra $\Heck$ has an automorphism $\#$ given by $(X_i)^\#=X_i^{-1}$ and $(T_r)^\#=(q-1)-T_r=-q(T_r)^{-1}$. We would like to extend it to the higher level affine Hecke algebra. However, we don't get an automorphism of $\lHeck$ but we get an isomorphism between $\lHeck$ and $\lHeckop$, where $\bfQ^{-1}=(Q_1^{-1},\ldots,Q_\ell^{-1})$. The following is straightforward.

\begin{lem}
\label{lem-hash_isom}
There is an isomorphism of algebras 
$$
\begin{array}{rcll}
\#\colon \lHeck & \to & \lHeckop,\\
e(\uc)      &\mapsto  & e(\uc), &\\
X_ie(\uc) & \mapsto & X'_ie(\uc), & \mbox{ if }c_i=0,\\
T_re(\uc) & \mapsto & ((q-1)-T_r)e(\uc) & \mbox{ if }c_r=c_{r+1}=0,\\
T_re(\uc) & \mapsto & T_re(\uc) & \mbox{ if }c_r=1,c_{r+1}=0,\\
T_re(\uc) & \mapsto & -Q_rX'_rT_re(\uc) & \mbox{ if }c_r=0,c_{r+1}=1.
\end{array}
$$

\end{lem}

\subsection{The polynomial representation of ${\lHeck}$}
In this section we generalize the polynomial representation of the affine Hecke algebra to our higher level version by extending the action of $\lHeck$ on a Laurent polynomial ring in $d$ generators to an action of $\lHeck$ on a direct sum $\PollHeck$ of Laurent polynomial rings.

\begin{df}
For each $\uc\in J^{\ell,d}$ consider the subring  
\begin{eqnarray*}
\PollHeck(\uc)=\bfk[x^{\pm1}_1, \ldots, x^{\pm1}_d]\subset \bfk[X^{\pm 1}_1,\ldots,X^{\pm 1}_{\ell+d}]
\end{eqnarray*}
generated by the variables $x_t=X^{\pm 1}_{t_\uc}$ where $1_\uc<2_\uc<\ldots <d_\uc$ are precisely the positions of the black strands, that is those indices where $c_{1_\uc}=\cdots=c_{d_\uc}=0$. Set 
 \begin{eqnarray}
 \label{polrep}
 \PollHeck&=&\bigoplus_{\uc\in J^{\ell,d}}\PollHeck(\uc)=\bigoplus_{\uc\in J^{\ell,d}}\bfk[x_1^{\pm 1},\ldots,x_d^{\pm 1}]e(\uc).
 \end{eqnarray}
 Here $e(\uc)$ is a formal symbol distinguishing the different direct summands. 
\end{df}

\begin{prop}
\label{prop-pol_rep_lH}
There is an action of ${\lHeck}$ on $\PollHeck$ defined as follows.
\begin{itemize}
\item The element $e(\uc)$ acts as the projector to the direct summand $\PollHeck(\uc)$.
\item The element $X_ie(\uc)$ acts by multiplication with $X_i$ on $\PollHeck(\uc)$, if $c_i=0$ and by zero otherwise. (Recall that $X_ie(\uc)=0$ if $c_i=1$.) 
\item The element $T_re(\uc)$ acts only non-trivially on the summand $\PollHeck(\uc)$ where it sends $f\in \PollHeck(\uc)$ to 
\small
%\begin{eqnarray*}
%\begin{cases}
%qs_r(f)+(q-1)\frac{X_{r+1}}{(X_{r+1}-X_{r})}(f-s_r(f))\in\PollHeck(\uc)&\mbox{ if } c_r=c_{r+1}=0,\\
%s_r(f)\in \PollHeck e_{s_r(\uc)} &\mbox{ if } c_r=1, c_{r+1}=0,\\
%\left(X_{r+1}-Q_{\sum_{j=1}^{r+1}c_j}\right)s_r(f)\in \PollHeck(s_r(\uc)) &\mbox{ if } c_r=0, c_{r+1}=1,\\
%0 &\mbox{ if } c_r=c_{r+1}=1.
%\end{cases}
%\end{eqnarray*}
\begin{eqnarray*}
\begin{cases}
-s_r(f)+(q-1)\frac{X_{r+1}}{(X_{r}-X_{r+1})}(s_r(f)-f)\in\PollHeck(\uc)&\mbox{ if } c_r=c_{r+1}=0,\\
s_r(f)\in \PollHeck(s_r(\uc)) &\mbox{ if } c_r=1, c_{r+1}=0,\\
\left(X_{r+1}-Q_{\sum_{j=1}^{r+1}c_j}\right)s_r(f)\in \PollHeck(s_r(\uc)) &\mbox{ if } c_r=0, c_{r+1}=1,\\
0 &\mbox{ if } c_r=c_{r+1}=1.
\end{cases}
\end{eqnarray*}
\normalsize
(Recall that $T_re(\uc)=0$ if $c_r=1=c_{r+1}$.) 
\end{itemize}
\end{prop}
\begin{proof}
One directly verifies the relations from Definition~\ref{def-extaffHecke_alg}.
\end{proof}
During the  proof of  Proposition~\ref{prop-basis-Hdl} we will establish a crucial fact: 
\begin{prop}
\label{prop-faithfullHecke}
The representation from Proposition~\ref{prop-pol_rep_lH} is faithful.
\end{prop}

\subsection{A basis of ${\lHeck}$}
\label{subs-basis_lHeck}

The goal of this section is to construct a basis of the algebra ${\lHeck}$. To do this, it is enough to construct a basis of $e(\ucp){\lHeck}e(\uc)$ for each $\ucp,\uc\in J^{\ell,d}$. First we define for each $w\in \frakS_d$, $\ucp,\uc\in J^{\ell,d}$ an element $T_w^{\ucp,\uc}\in e(\ucp){\lHeck}e(\uc)$. We define this element using the diagrammatic calculus as follows: Consider the permutation $w$ and draw a permutation diagram using black strands representing $w$ with a minimal possible number of crossings. Then we create the sequence $\ucp$ (resp. $\uc$) on the top (resp. bottom) of the diagram by adding accordingly $\ell$ red points on the top and $\ell$ red points on the bottom. Finally we join the red points on the top with the red points on the bottom by red strands in such a way that there are no intersections between red strands and such that a red strand intersects each black strand at most once. The resulting element is denoted $T_w^{\ucp,\uc}$. By construction it depends on several choices, but we just fix such a choice for any triple $(\ucp,\uc, w)$.

\begin{ex}
Let $d=3$, $\ell=2$, $\ucp=(1,1,0,0,0)$, $\uc=(0,1,0,0,1)$, and $w=s_1s_2s_1$. Then there are precisely two choices for the permutation diagram of $w$, we displayed one on the left in \eqref{threediags}.  The diagram $T_w^{\ucp,\uc}$ involves again a choice. Two of the possible choices are as follows
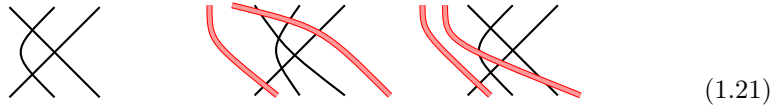
\begin{eqnarray}
\label{threediags}
\begin{tikzpicture}[scale=0.6, thick]
      \draw (-3,3)  +(1,-1) -- +(-1,1);
      \draw (-3,3)  +(0,-1) .. controls (-4,3) ..  +(0,1);
      \draw (-3,3) +(-1,-1) -- +(1,1);
\end{tikzpicture}\quad&&\quad
\begin{tikzpicture}[scale=0.6, thick]
      \draw (-3,3)  +(1,-1) .. controls (-3.3,3) ..  +(-1,1);
      \draw (-3,3)  +(0,-1) .. controls (-3.7,3) ..  +(0,1);
      \draw (-3,3) +(-1,-1) -- +(1,1);
       \draw[wei] (-3,3) +(-2,1) .. controls +(0,-0.8) ..  +(-0.5, -1) ;
       \draw[wei] (-3.5,3)+(-1,1) .. controls +(2,-0.5) ..  +(2.5,-1);
\end{tikzpicture}
\quad
\begin{tikzpicture}[scale=0.6, thick]
      \draw (-3,3)  +(1,-1) -- +(-1,1);
      \draw (-3,3)  +(0,-1) .. controls (-4,3) ..  +(0,1);
      \draw (-3,3) +(-1,-1) -- +(1,1);
        \draw[wei] (-3,3)+(-2,1) .. controls +(0,-0.6) .. +(-0.5,-1);
        \draw[wei] (-3.5,3)+(-1,1) .. controls +(0,-0.8) .. +(2,-1);
\end{tikzpicture}
\end{eqnarray}
%This choice is not unique. It is also possible to take 
%$$
%\begin{tikzpicture}[scale=0.6, thick]
%      \draw (-3,3)  +(1,-1) -- +(-1,1);
%      \draw (-3,3)  +(0,-1) .. controls (-4,3) ..  +(0,1);
%      \draw (-3,3) +(-1,-1) -- +(1,1);
%        \draw[wei] (-3,3)+(-3,1) -- +(-0.5,-1);
%        \draw[wei] (-3,3)+(-2,1) .. controls +(0,-0.3) .. +(2,-1);
%\end{tikzpicture}
%$$
\end{ex} 

Let $\lHeck^{\leqslant w}$ be the span of the elements of the form $T_y^{\ucp,\uc}f$, where $\ucp,\uc\in J^{\ell,d}$, $y\leqslant w$ and $f\in\PollHeck(\uc)$. Define $H_{\ell,d}^{\bfQ,<w}$ similarly. 
\begin{lem}
\begin{enumerate}
\item \label{1} The subspaces $\lHeck^{\leqslant w}$ and $\lHeck^{<w}$ of $\lHeck$ are independent of the choices of the elements $T_x^{\ucp,\uc}$.
\item  \label{2} The different choices of $T_w^{\ucp,\uc}$ attached to $w, \uc, \ucp$ by the construction above are equal modulo $H_{\ell,d}^{\bfQ,<w}$.
\end{enumerate}
\end{lem}

\begin{proof}
We prove both parts simultaneously by induction on the length of $w$. Assume ${l}(w)=0$. In this case the definition of the element $T_w^{\ucp,\uc}$ is independent of any choice and there is noting to show. Assume now that the statements are true for all $w$ such that ${l}(w)<n$ and let us prove them for ${l}(w)=n$.

By definition, the vector space $\lHeck^{<w}$ is spanned by $\lHeck^{\leqslant z}$ for all $z<w$. By the induction hypothesis, the vector spaces $\lHeck^{\leqslant z}$ are independent of the choices of $T^{\ucp,\uc}_x$ such that $x\leqslant z$. Thus the vector space $\lHeck^{<w}$  is independent of the choices of $T^{\ucp,\uc}_y$ where $y < w$. This proves the second part of~\ref{1}.). To prove~\ref{2}.), consider two different choices for the diagram $T_w^{\ucp,\uc}$. Then one of them can be obtained from the other one by applying relations in Definition~\ref{def-extaffHeck_diag}, which might create additional terms, but they are all contained in $\lHeck^{<w}$ hence~\ref{2}.) holds. Now~\ref{1}.) follows from~\ref{2}.) and the part of~\ref{1}.) which we already established.
\end{proof}

To give a basis of ${\lHeck}$, it is convenient to introduce some new elements $x_1,\ldots,x_d\in {\lHeck}$. Set $x_r=\sum_{\uc\in J^{\ell,d}}X_{r_\uc}e(\uc)$, where $r_\uc$ is the number of the position in $\uc$ where the colour black appears for the $r$th time (counted from the left). Then the  following statement is obvious from the relations \eqref{lHecke7}-\eqref{lHecke5new}.
\begin{lem}
The elements $x_1,\ldots,x_d$ pairwise commute and are invertible. 
\end{lem}

The following provides two bases of ${\lHeck}$.
\begin{prop}
\label{prop-basis-Hdl}
For each $\ucp,\uc\in J^{\ell,d}$, the following sets 
\begin{eqnarray*}
\{T_w^{\ucp,\uc}x_1^{m_1}\ldots x_d^{m_d}\mid w\in \frakS_d,m_i\in \bbZ\}, && \{x_1^{m_1}\ldots x_d^{m_d}T_w^{\ucp,\uc}\mid w\in \frakS_d,m_i\in \bbZ\}
\end{eqnarray*}
each form a basis of $e(\ucp){\lHeck} e(\uc)$.
\end{prop}
\begin{proof}
It is clear from the defining relations of ${\lHeck}$ that the asserted basis elements span $e(\ucp){\lHeck} e(\uc)$. Indeed, we can use relations \eqref{Hecke-diag-1} - \eqref{l-Hecke-diag-4} to write each diagram as a linear combination of diagrams where all dots are above (resp. below) all intersections and such that two strands intersect at most twice. To prove the linear independence, it suffices to show that the elements act by linearly independent operators on the polynomial representation \eqref{polrep}. 

The element $T_w^{\ucp,\uc}$ takes $\bfk[x^{\pm 1}_1,\cdots,x^{\pm 1}_d]e(\ucp)$ to $\bfk[x^{\pm 1}_1,\cdots,x^{\pm 1}_d]e(\uc)$
by sending $fe(\ucp)$ to $\sum_{y\in\frakS_d,y\leqslant w}C_yy(f)e(\uc)$,
where the $C_y \in\bfk(x_1,\cdots,x_d)$ are rational functions such that $C_w\ne 0$. Since $y\in\frakS_d$ acts on the polynomial $f$ by the obvious permutation $y(f)$ of variables, an expression of the form $\sum_{w}a_w T_w^{\ucp,\uc}$ or $\sum_{w} T_w^{\ucp,\uc}a_w$, where $a_w\in\bfk[x^{\pm 1}_1,\ldots,x^{\pm 1}_d]$, $w\in\frakS_d$, can only act by zero if each $a_w$ is zero. This implies the linear independence. 
\end{proof}
\begin{rk}
\label{rk-special}
In the special case $\ell=0$ these bases are the standard bases of the affine Hecke algebra from \cite[Prop.~3.7]{Lus89}, see also  \cite[Cor.~3.4]{MS}.
\end{rk}
\subsection{The centre of ${\lHeck}$}
\label{subs-centre-lHecke}
Consider the element $\omega=(1,\ldots,1,0,\ldots,0)\in J^{\ell,d}$. This means that $\omega$ contains the colour red $\ell$ times followed by the colour black $d$ times. The following lemma shows that the affine Hecke algebra $\Heck$ from Remark~\ref{ordinaryHecke} can be realised as an idempotent truncation of the higher level affine Hecke algebra. In particular our diagrams generalize indeed the ordinary permutation diagrams.
\begin{lem}
\label{lem-Hd_in_Hdl}
There is an isomorphism of algebras $\Heck\simeq e(\omega){\lHeck}e(\omega)$. 
\end{lem}
\begin{proof}
There is an obvious algebra homomorphism $\Heck\to e(\omega){\lHeck}e(\omega)$ that adds $\ell$ red strands to the left of the diagram. It is an isomorphism, because it sends the standard basis (see Remark~\ref{rk-special}) of the affine Hecke algebra $\Heck$ to the basis of $e(\omega){\lHeck} e(\omega)$ from Proposition~\ref{prop-basis-Hdl}. 
\end{proof}

The group $\frakS_d$ acts on the Laurent polynomial ring $\PollHeck(\uc)$ for each $\uc\in J^{\ell,d}$. Moreover, the group $\frakS_{\ell+d}$ acts on $\PollHeck$ such that the permutation $w\in\frakS_{\ell+d}$ sends the element $f\in \PollHeck({\uc})$ to $w(f)\in\PollHeck({w(\uc)})$. For each $\uc\in J^{\ell,d}$, the restriction of the projection $\PollHeck\to \PollHeck(\uc)$ to $\PollHeck^{\frakS_{\ell+d}}$ yields an isomorphism $\PollHeck^{\frakS_{\ell+d}}\simeq \PollHeck(\uc)^{\frakS_{d}}$ of vector spaces. By identifying $\PollHeck(\uc)=\bfk[x_1^{\pm 1},\ldots,x_d^{\pm 1}]e(\uc)$, we can view $\PollHeck$ as a subalgebra of ${\lHeck}$ containing the algebra $\bfk[x^{\pm 1}_1,\ldots,x^{\pm 1}_d]$  embedded diagonally. Moreover, the subalgebra $\bfk[x^{\pm 1}_1,\ldots,x^{\pm 1}_d]^{\frakS_d}$ coincides with $\PollHeck^{\frakS_{\ell+d}}$. The centre $Z({\lHeck})$ of ${\lHeck}$ is then given as follows.
\begin{prop}
\label{lem-cen-Hdl}
We have  $Z({\lHeck})=\bfk[x^{\pm 1}_1,\ldots,x^{\pm 1}_d]^{\frakS_d}=\PollHeck^{\frakS_{\ell+d}}$.
\end{prop} 
\begin{proof}
It is clear that $\PollHeck^{\frakS_{\ell+d}}\subset Z({\lHeck})$. It suffices to show that the centre contains not more elements. Let $z\in Z({\lHeck})$. Write $z=\sum_{\uc\in J^{\ell,d}}z_\uc$, where $z_\uc=ze(\uc)$. Then $z_\omega\in Z(e(\omega){\lHeck}e(\omega))$. Since the centre of the affine Hecke algebra is formed by symmetric Laurent polynomials, \cite[Prop.~3.11]{Lus89},  there exists, by Lemma~\ref{lem-Hd_in_Hdl}, some $f\in \PollHeck(\omega)^{\frakS_d}$ such that $z_\omega=f$.
To complete, it is enough to show that $z_{w(\omega)}=w(f)\in \PollHeck({w(\uc)})$ for each $w\in \frakS_{\ell+d}$. Let $T=T_{\Id}^{w(\omega),\omega}$. Since $z$ commutes with $T$, we must have $z_{w(\omega)}T=Tz_{\omega}$. On the other hand we have $Tz_{\omega}=Tf=w(f)T$. This implies $z_{w(\omega)}=w(f)$ because the map 
$e(w(\omega)){\lHeck}e(w(\omega))\longrightarrow e(w(\omega)){\lHeck}e(\omega),$ $y\longmapsto yT
$ is injective by Proposition~\ref{prop-basis-Hdl}.
\end{proof}

\subsection{Completion}
\label{subs-compl-Hecke}
For our main result we have to complete the higher level affine Hecke algebra. We first recall the completion $\cHeck$ of $\Heck$ from \cite[Sec.~3.3]{MS} at a maximal ideal of $Z(\Heck)$. From now on we assume $\bfk$ to be algebraically closed. 

For each $\bfa=(a_1,\ldots,a_d)\in(\bfk^*)^d$ consider the central character 
$\chi_\bfa\colon Z(\Heck)=\bfk[X_1^{\pm 1},\ldots,X_d^{\pm 1}]^{\frakS_d}\to \bfk$ obtained by restriction of the algebra homomorphism which sends $X_1,\ldots,X_d$ to $a_1,a_2,\ldots,a_d$ respectively. Two such central characters $\chi_\bfa$ and $\chi_{\bfa'}$ coincide if and only if $\bfa'$ is a permutation of $\bfa$. Fix now $\bfa$. 

\begin{df}
We denote by  $\cHeck$ the completion of $\Heck$ with respect to the ideal $\mathfrak{m}_{\bfa}$ of $\Heck$ generated by $\ker \chi_\bfa$.
\end{df}

Each finite dimensional $\cHeck$-module decomposes into its generalised eigen\-spaces  $M=\bigoplus_{\ui\in \frakS_d\bfa}M_{\ui}$,  for the $\bfk[X_1^{\pm 1},\ldots,X_d^{\pm 1}]$-action, where 
\begin{eqnarray}
\label{dec}
M_{\ui}&=&\{m\in M\mid\exists N\in\bbZ_{\geqslant 0} \mbox{ such that } (X_r-i_r)^Nm=0~\forall r\}.
\end{eqnarray}
For each $\ui\in\frakS_d\bfa$, there is an idempotent $e(\ui)\in \cHeck$ which projects onto $M_{\ui}$ when applied to $M$. Obviously, $1=\sum_{\ui}e(\ui)$ holds.

\begin{df}
By a {\it topological basis} or {\it Schauder basis} of a topological $\bfk$-vector space $V$ we mean a sequence $v_i$, $i\in\bbZ_{\geqslant 0}$ of vectors in $V$ such that every element of $V$ can be expressed uniquely as a convergent series of the form $\sum_{i\in\bbZ_{\geqslant 0}}a_iv_i$ with $a_i\in\bfk$. 
\end{df}
We consider now $\cHeck$ with its $\mathfrak{m}_{\bfa}$-adic topology.  It comes with the usual  $\mathfrak{m}_{\bfa}$-adic-order function, namely the order of an element is the minimal number $j$ such that $f$ is not in $\mathfrak{m}_{\bfa}^j$. This defines a norm on $\cHeck$ and hence we can talk about topological bases, see \cite[VII]{ZS} for more details. 

\begin{prop}[{\cite[Lemma~3.8]{MS}}]
%Each of the following sets
The following set (viewed as a sequence by picking any total ordering)
$$
\left\{T_w(X_1-i_1)^{m_1}\ldots (X_d-i_d)^{m_d}e(\ui)\mid w\in \frakS_d,m_i\in \bbZ_{\geqslant 0},\ui\in \frakS_d\bfa\right\}
$$
%$$
%\{(X_1-a'_1)^{c_1}\ldots (X_d-a'_d)^{c_d}T_we(\bfb)\mid  w\in \frakS_d,m_i\in \bbZ_{\geqslant 0},\bfb\in \frakS_d\bfa\}
% \qquad \mbox{and} \qquad \{x_1^{c_1}\ldots x_d^{c_d}T_w^{\ui,\uj}; w\in \frakS_d,m_i\in \bbZ\}
%$$
forms a topological basis of $\cHeck$.
\end{prop}
 Informally speaking this means that every element in $\cHeck e(\ui)$ can be written uniquely as a power series in 
the $(X_r-i_r)$ with coefficients in $\Hfin_d(q)$, see \cite[VII (8)]{ZS} for a precise statement. In particular  $\Heck$ is everywhere dense in $\cHeck$ in the sense of  \cite[VII, Lemma 1]{ZS}. 

\begin{prop}[{\cite[Cor.~3.13]{MS}}]
\label{MS}
The algebra $\cHeck$ acts faithfully on 
\begin{eqnarray*}
\cPolHeck&=&\bigoplus_{\ui\in\frakS_d\bfa}\bfk[[x_1-i_1,\ldots,x_d-i_d]]e(\ui\def\cPollHeck{\widehat{\op{P}}_{\bfa,\bfQ}}).  %\simeq\widehat H^\bfQ_{\bfa,\ell}(q)\otimes_{ {\lHeck}}\Pol.
\end{eqnarray*}
\end{prop}
By Proposition~\ref{lem-cen-Hdl}, the algebra $Z(\lHeck)$ is independent of the level $\ell$ and so we can consider $\chi_\bfa$ as a central character of $\lHeck$ as well. Let  $\clHeck$ be the completion of ${\lHeck}$ with respect to the ideal $\frakm_\bfa$ generated by the kernel of $\chi_\bfa$ in $\lHeck$. We have again the  decomposition \eqref{dec} for each finite dimensional $\clHeck$-module $M$ and an idempotent  $e(\ui)\in \clHeck$ projecting onto $M_{\ui}$. The idempotents $e(\ui)$ for $\ui\in \frakS_d\bfa$ commute with the idempotents $e(\uc)$ for $\uc\in J^{\ell,d}$. Thus we may define idempotents $e(\uc,\ui)=e(\uc)e(\ui)$ in $\clHeck$. We have $1=\sum_{\uc,\ui}e(\uc,\ui)$. 
\begin{prop}
\label{prop_basis-lHecke-comp}
%Each of the following sets\
\begin{enumerate}
\item The following set (viewed as a sequence by picking any total ordering)
\begin{equation*}
\left\{T_w^{\ucp,\uc}(x_1-i_1)^{m_1}\ldots (x_d-i_d)^{m_d}e(\uc,\ui)\left|
\begin{array}[c]{ll}
  w\in \frakS_d,&m_i\in \bbZ_{\geqslant 0},\\
  \ucp,\uc\in J^{\ell,d},&\ui\in \frakS_d\bfa
  \end{array}
  \right.\right\}
% \qquad \mbox{and} \qquad \{(x_1-i_i)^{m_1}\ldots (x_d-i_d)^{m_d}T_w^{\ucp,\uc}; w\in \frakS_d,m_i\in \bbZ\}
\end{equation*}
%$$
%\{(x_1-a'_1)^{m_1}\ldots (x_d-a'_d)^{m_d}T_w^{\ucp,\uc}e(\ucp,\bfa')\mid  w\in \frakS_d,m_i\in \bbZ_{\geqslant 0},\ucp\in J^{\ell,d},\bfa'\in \frakS_d\bfa\}% \qquad \mbox{and} \qquad \{x_1^{m_1}\ldots x_d^{m_d}T_w^{\ucp,\uc}; w\in \frakS_d,m_i\in \bbZ\}
%$$
forms a topological basis of $\clHeck$. 
\item \label{prop_rep-lHecke-comp}
The algebra $\clHeck$ acts (extending the actions from Propositions~\ref{prop-pol_rep_lH} and~\ref{MS}) faithfully on
\begin{eqnarray*}
\cPollHeck=\bigoplus_{\uc\in J^{\ell,d},\ui\in \frakS_d\bfa}\bfk[[x_1-i_1,\ldots,x_d-i_d]]e(\uc,\ui). %\simeq\widehat H^\bfQ_{\bfa,\ell}(q)\otimes_{ {\lHeck}}\Pol.
\end{eqnarray*}
where $e(\uc,\ui)$ is just a formal symbol on which $e(\uc,\ui)$ acts by the identity and all other $e(\uc',\uj)$ as zero.
\end{enumerate}
\end{prop}
\begin{proof} All statements follow directly from the definitions except the faithfulness. The action is such that  $e(\uc,\ui)=e(\uc)e(\ui)$ acts as the projector to the direct summand $\bfk[[x_1-i_1,\ldots,x_d-i_d]]e(\uc,\ui)$. We then write $\cPollHeck=\bigoplus_{\uc\in J^{\ell,d}} P({\uc)}$, where $P({\uc})=\bigoplus_{\ui\in \frakS_d\bfa }\bfk[[x_1-i_1,\ldots,x_d-i_d]]e(\uc,\ui)$. Then the completion $\bfk[[x_1-i_1,\ldots,x_d-i_d]]e(\ui)\subset\clHeck$ acts just by the obvious multiplication on $\bfk[[x_1-i_1,\ldots,x_d-i_d]]e(\ui)$ and by zero on the other summands. There is an action of $\frakS_d$ on $P({\uc})$ such that $w\in \frakS_d$ sends 
$f(x_1-i_1,\hdots,x_d-i_d)e(\uc,\ui)$ to $f(x_{w(1)}-i_1,\hdots,x_{w(d)}-i_d)e(\uc,w(\ui))$
where $w(\ui)=(i_{w^{-1}(1)},\hdots, i_{w^{-1}(d)})$. 
The action of $\frakS_d$ on $P(\uc)$ can therefore be extended to an action on
$$
\bigoplus_{\ui\in \frakS_d\bfa}\bfk((x_1-i_1,\ldots,x_d-i_d))e(\uc,\ui).
$$
Then the element $T_w^{\ucp,\uc}$ takes $ P({\uc)}$ to $ P({\ucp)}$ and sends an element $fe(\uc)$,
$$
f\in \bigoplus_{\ui\in \frakS_d\bfa}\bfk[[x_1-i_1,\ldots,x_d-i_d]]e(\uc,\ui),
$$
to an element of the form $\sum_{y\in\frakS_d,y\leqslant w}y(\varphi_yf)e(\ucp)$,
where we have 
$$\varphi_y\in \bigoplus_{\ui\in\frakS_d\bfa}\bfk((x_1-i_1,\hdots,x_d-i_d))e(\uc,\ui)
$$
and $\varphi_w\ne 0$. This implies that an expression of the form $\sum_{w\in \frakS_d}T^{\ucp,\uc}_wa_w$ with $a_w\in\bfk[[x_1-i_1,\ldots,x_d-i_d]]e(\ui)$ acts on $\cPollHeck$ by zero only if each $a_w$ is zero. This means exactly that the set from the statement of Proposition~\ref{prop_basis-lHecke-comp} acts on $\cPollHeck$ by linearly independent operators. It is clear that this set spans the algebra $\clHeck$ in the topological sense. Hence it forms a topological basis of $\clHeck$, and that the representation $\cPollHeck$ is faithful.
\end{proof}

\section{Affine KLR and tensor product algebras $\lR(\Gamma)$}
\label{sec-KLR_tens}
The next goal is to identify our higher level Hecke algebras, after completion, with Webster's tensor product algebras, \cite{Webster}, attached to a type $A$ quiver depending on $q$ and $\bfQ$. Let $J$ be as in Setup~\ref{set-up}.

\subsection{Tensor product algebras}
Let $\Gamma=(I,A)$ be a quiver without loops with set of vertices $I$ and set of arrows $A$. We call elements in $I$ {\it labels} since they will be used later as black and red labels.

Consider the set $\Icol=J\times I$ with the two obvious projections $c\colon \Icol\to J$ and $\gamma\colon \Icol\to I$ that forget the labels respectively the colours. Obviously, elements $z\in\Icol$ are determined by their colour $c(z)$ and their label $\gamma(z)$, thus we call them {\it called labels}. We call $z$ {\it black} if $c(z)=0$ and {\it red} otherwise. One  can also think of $\Icol$  as two copies of $I$, one copy coloured in black and the other copy coloured in red. 

We fix an $\ell$-tuple $\bfQ=(Q_1,\ldots,Q_\ell)\in I^\ell$.
\begin{df}
Let $\bfnu\in I^d$. Then $\Icolnu$ denotes the set of $(\ell+d)$-tuples $\ut=(t_1,\cdots,t_{\ell+d})\in \Icol^{\ell+d}$ such that 
\begin{itemize}
\item $\sum_{i=1}^{\ell+d}c(t_i)=d$ (i.e., $c(\ut)$ contains $d$ black elements and $\ell$ red elements),
\item the labels of black elements in $\ut$ form a permutation of $\bfnu$,
\item the labels of the red elements of $\ut$ are $Q_1,\ldots, Q_\ell$ (in this order).  
\end{itemize}
\end{df}  

\begin{df} 
 A $\Gamma$-$(\ell,d)$-diagram is an $(\ell,d)$-diagram in the sense of Definition~\ref{diag}  for the set $I_b=I_r=I$ of vertices of $\Gamma$. It is of type $(\bfnu,\bfQ)$, if the sequence of coloured labels is in $\Icolnu$.
\end{df}  
As before, the labels are read from left to right at the bottom of the diagram. Since reds strands never cross, we could read off the type (although possibly realized via a different sequence in the same orbit) at any horizontal slice of the diagram  instead of at the bottom.

\begin{ex}
Take $\bfnu=(i,i,j)\in I^3$, $\bfQ=(i,k)\in I^2$ (in particular, we have $d=3$ and $\ell=2$).  Then the tuple $\ut=((i,1),(j,0),(i,0),(i,0),(k,1))$ is an element of $\Icol$. The labels of black elements in $\ut$ are $(j,i,i)$, which is a permutation of $\bfnu$. The labels of red elements in $\ut$ are $(i,k)$, this coincides with $\bfQ$. If we forget the labels in $\ut$, we get the tuple of colours $c(\ut)=(1,0,0,0,1)\in J^{2,3}$.
\end{ex}
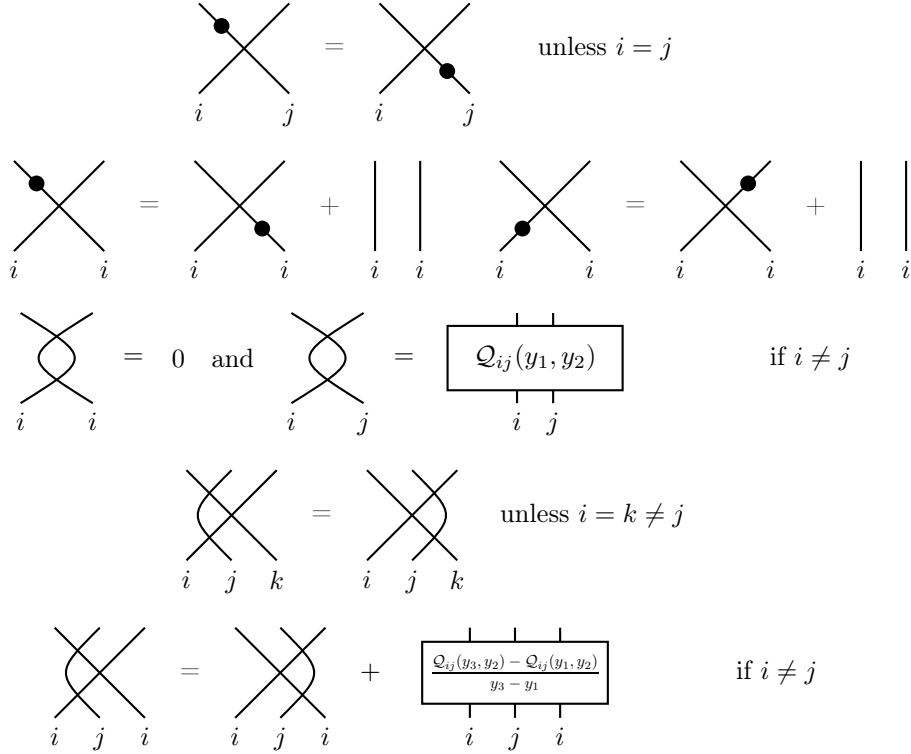
\begin{figure}
\begin{eqnarray*}
    \begin{tikzpicture}[scale=0.6]
      \draw[thick](-4,0) +(-1,-1) -- +(1,1) node[below,at start]
      {$i$}; \draw[thick](-4,0) +(1,-1) -- +(-1,1) node[below,at
      start] {$j$}; \fill (-4.5,.5) circle (5pt);
      % \draw[thick] (0,0) +(0,-1) -- +(0,1) node[below, at
      % start]{$i$}; \fill (0,0) circle (5pt);
      \node at (-2,0){=}; \draw[thick](0,0) +(-1,-1) -- +(1,1)
      node[below,at start] {$i$}; \draw[thick](0,0) +(1,-1) --
      +(-1,1) node[below,at start] {$j$}; \fill (.5,-.5) circle (5pt);
      \node at (4,0){unless $i=j$};
    \end{tikzpicture}
    \end{eqnarray*}
\begin{eqnarray*}
    \begin{tikzpicture}[scale=0.6,thick]
      \draw[thick](-4,0) +(-1,-1) -- +(1,1) node[below,at start]
      {$i$}; \draw[thick](-4,0) +(1,-1) -- +(-1,1) node[below,at
      start] {$i$}; \fill (-4.5,.5) circle (5pt);
      % \draw[thick] (0,0) +(0,-1) -- +(0,1) node[below, at
      % start]{$i$}; \fill (0,0) circle (5pt);
      \node at (-2,0){=}; \draw[thick](0,0) +(-1,-1) -- +(1,1)
      node[below,at start] {$i$}; \draw[thick](0,0) +(1,-1) --
      +(-1,1) node[below,at start] {$i$}; \fill (.5,-.5) circle (5pt);
      \node at (2,0){+}; \draw[thick](4,0) +(-1,-1) -- +(-1,1)
      node[below,at start] {$i$}; \draw[thick](4,0) +(0,-1) --
      +(0,1) node[below,at start] {$i$};
    \end{tikzpicture}
&&
    \begin{tikzpicture}[scale=0.6,thick]
      \draw[thick](-4,0) +(-1,-1) -- +(1,1) node[below,at start]
      {$i$}; \draw[thick](-4,0) +(1,-1) -- +(-1,1) node[below,at
      start] {$i$}; \fill (-4.5,-.5) circle (5pt);
      % \draw[thick] (0,0) +(0,-1) -- +(0,1) node[below, at
      % start]{$i$}; \fill (0,0) circle (5pt);
      \node at (-2,0){=}; \draw[thick](0,0) +(-1,-1) -- +(1,1)
      node[below,at start] {$i$}; \draw[thick](0,0) +(1,-1) --
      +(-1,1) node[below,at start] {$i$}; \fill (.5,.5) circle (5pt);
      \node at (2,0){+}; \draw[thick](4,0) +(-1,-1) -- +(-1,1)
      node[below,at start] {$i$}; \draw[thick](4,0) +(0,-1) --
      +(0,1) node[below,at start] {$i$};
    \end{tikzpicture}
    \end{eqnarray*}
\begin{eqnarray*}
      \begin{tikzpicture}[scale=0.6, thick]  
      \draw (-4,0)  +(0,-1) .. controls (-2.4,0) ..  +(0,1) node[below,at start]{$i$}; 
       \draw (-2.4,0)  +(0,-1) .. controls (-4,0) ..  +(0,1)  node[below,at start]{$i$}; 
       \node at (-1.5,0){$=$}; \node at (-0.5,0){$0$}; \node at (0.7,0){and};
    \draw (2,0)  +(0,-1) .. controls (3.6,0) ..  +(0,1)  node[below,at start]{$i$};
       \draw (3.6,0)  +(0,-1) .. controls (2,0) ..  +(0,1)  node[below,at start]{$j$} ; 
       \node at (4.5,0){$=$}; 
      \draw (7.8,0) +(0,-1) -- +(0,1) node[below,at start]{$j$};
      \draw (7,0) +(0,-1) -- +(0,1) node[below,at start]{$i$};
       \node[inner xsep=10pt,fill=white,draw,inner ysep=7pt] at (7.4,0) {$\calQ_{ij}(y_1,y_2)$};
\node at (13.5,0) {if $i\ne j$};
        \end{tikzpicture} 
\end{eqnarray*}
\begin{eqnarray*}
    \begin{tikzpicture}[thick,scale=0.6]
      \draw (-3,0) +(1,-1) -- +(-1,1) node[below,at start]{$k$}; 
      \draw (-3,0) +(-1,-1) -- +(1,1) node[below,at start]{$i$};
       \draw (-3,0) +(0,-1) .. controls (-4,0) ..  +(0,1) node[below,at
      start]{$j$}; \node at (-1,0) {=}; 
      \draw (1,0) +(1,-1) -- +(-1,1) node[below,at start]{$k$}; 
      \draw (1,0) +(-1,-1) -- +(1,1) node[below,at start]{$i$}; 
      \draw (1,0) +(0,-1) .. controls (2,0) ..  +(0,1) node[below,at start]{$j$}; \node at (5,0)
      {unless $i=k\neq j$};
    \end{tikzpicture}
  \end{eqnarray*}
\begin{eqnarray*}
    \begin{tikzpicture}[thick,scale=0.6]
      \draw (-3,0) +(1,-1) -- +(-1,1) node[below,at start]{$i$}; 
      \draw (-3,0) +(-1,-1) -- +(1,1) node[below,at start]{$i$}; 
      \draw (-3,0) +(0,-1) .. controls (-4,0) ..  +(0,1) node[below,at start]{$j$}; \node at (-1,0) {=}; 
      \draw (1,0) +(1,-1) -- +(-1,1) node[below,at start]{$i$}; 
      \draw (1,0) +(-1,-1) -- +(1,1) node[below,at start]{$i$}; 
      \draw (1,0) +(0,-1) .. controls (2,0) ..  +(0,1) node[below,at start]{$j$}; \node at (3,0){$+$};        
      \draw (6.2,0)+(1,-1) -- +(1,1) node[below,at start]{$i$}; 
      \draw (6.2,0)+(-1,-1) -- +(-1,1) node[below,at start]{$i$}; 
      \draw (6.2,0)+(0,-1) -- +(0,1) node[below,at start]{$j$};
\node[inner ysep=8pt,inner xsep=5pt,fill=white,draw,scale=.6] at (6.2,0){$\displaystyle \frac{\calQ_{ij}(y_3,y_2)-\calQ_{ij}(y_1,y_2)}{y_3-y_1}$};
\node at (12,0) {if $i\ne j$};
    \end{tikzpicture}
 \end{eqnarray*}
 \caption{Tensor product algebra relations I: The KLR relations}
  \label{defKLR}
\end{figure}

To define the tensor product algebras we need one more definition. For each $i,j\in I$ we denote by $h_{i,j}$ the number of arrows in the quiver $\Gamma$ going from $i$ to $j$, and define for $i\ne j$ the polynomials
$$
\calQ_{ij}(u,v)=(u-v)^{h_{i,j}}(v-u)^{h_{j,i}}.
$$

\begin{df}
\label{tpalg}
Fix a $d$-tuple $\bfnu\in I^d$.  The \emph{tensor product algebra} $\lR(\Gamma)$ (or simply $\lR$) is the induced algebra structure on the vector space spanned by all $\Gamma$-$(\ell,d)$-diagrams of type $(\bfnu,\bfQ)$ 
modulo the {\it tensor product algebra relations of KLR type} from Figure~\ref{defKLR} and the {\it  tensor product algebra relations of the second type} from Figure~\ref{reltensor2}.
\end{df}

\begin{rk}
The special case where we only allow black strands (that is $\ell=0$),  is the KLR algebra $\cR$ originally introduced in \cite{KL} and \cite{Rou2KM}. The following elements (defined for $\ui=(i_1,\cdots,i_d)\in I^\bfnu$, $i\in[1;d]$ and $r\in[1;d-1]$)
$$
\tikz[thick,xscale=2.5,yscale=1.5]{
\node at (-.7,.25) {$e(\ui)=$};
\draw (0,0) --(0,.5) node[below,at start]{$i_1$};
\draw (.4,0) --(.4,.5) node[below,at start]{$i_2$};
\node at (.7,.25) {$\cdots$};
\draw (1,0) --(1,.5) node[below,at start]{$i_i$};
\node at (1.3,.25) {$\cdots$};
\draw (1.6,0) --(1.6,.5) node[below,at start]{$i_{d-1}$};
\draw (2,0) --(2,.5) node[below,at start]{$i_d$};
}
$$
and
$$
\tikz[thick,xscale=2.5,yscale=1.5]{
\node at (-.9,.25) {$y_ie(\ui)=$};
\draw (0,0) --(0,.5) node[below,at start]{$i_1$};
\draw (.4,0) --(.4,.5) node[below,at start]{$i_2$};
\node at (.7,.25) {$\cdots$};
\draw (1,0) --(1,.5) node[below,at start]{$i_i$};
\fill (1,.25) circle (1pt);
\node at (1.3,.25) {$\cdots$};
\draw (1.6,0) --(1.6,.5) node[below,at start]{$i_{d-1}$};
\draw (2,0) --(2,.5) node[below,at start]{$i_d$};
}
$$
and 
$$
\tikz[thick,xscale=2.5,yscale=1.5]{
\node at (-.67,.25) {$\psi_re(\ui)=$};
\draw (0,0) --(0,.5) node[below,at start]{$i_1$};
\node at (.3,.25) {$\cdots$};
\draw (.6,0) --(.6,.5) node[below,at start]{$i_{r-1}$};
\draw (.9,0) --(1.3,.5) node[below,at start]{$i_r$};
\draw (1.3,0) --(.9,.5) node[below,at start]{$i_{r+1}$};
\draw (1.6,0) --(1.6,.5) node[below,at start]{$i_{r+2}$};
\node at (1.9,.25) {$\cdots$};
\draw (2.2,0) --(2.2,.5) node[below,at start]{$i_d$};
}
$$
generate the algebra, see \cite{KL}, \cite{Rou2KM}.
\end{rk}

Now, for $\ui\in \Icolnu$, $r\in[1;\ell+d-1]$, $j\in [1;\ell+d]$ we define more generally elements $e(\ui)$,  $\psi_re(\ui)$, $Y_ie(\ui)$ that will generate the algebra $\lR$. 

\begin{df}
Let $e(\ui)\in \lR$ be the idempotent given by the diagram with only vertical strands with colours and labels determined by the sequence $\ui$. Let $Y_je(\ui)$ be the same diagram with additionally  a dot on the strand number $j$ (counting from the left) in case  $i_i$ is black, and set $Y_je(\ui)=0$ if $i_j$ is red. Finally let $\psi_re(\ui)$ be the same diagram as $e(\ui)$ except that the $r$-th and $(r+1)$th strand intersect once in case not both $i_r$ and $i_{r+1}$ are red, and set $\psi_re(\ui)=0$ otherwise. 
\end{df}

\begin{ex}
For example, for $\ui=((i,1),(j,0),(i,0),(i,0),(k,1))$, we have 
$$
    \begin{tikzpicture}[thick, scale=0.4]
\node at (2.8,-.2){$ \Large{e(\ui)\;\;= }$}  ;
      \draw[wei] (6.5,0)  +(-2,-1) -- +(-2,1) node[at start,below]{$i$};
      \draw (6.5,0) +(-1,-1) -- +(-1,1) node [at start,below]{$j$};
      \draw (6.5,0) +(0,-1) -- +(0,1)node [at start,below]{$i$};
      \draw (6.5,0)  +(1,-1) -- +(1,1) node[at start,below]{$i$};
      \draw[wei] (6.5,0)  +(2,-1) -- +(2,1) node[at start,below]{$k$};
 \end{tikzpicture}
$$
with $Y_re(\ui)=0$ for $r=1$ and $r=5$.
\end{ex}

We preferred here to define the algebras diagrammatically instead of giving a cumbersome definition similar to Definition~\ref{def-extaffHecke_alg}.  Analogously to the situation for the algebra ${\lHeck}$, it is convenient to introduce the elements $y_1,\ldots,y_d\in \lR$ defined as $y_r=\sum_{\ui\in I^d}Y_{r_\ui}e(\ui)$, with $r_\ui$ being the number of the position in $\ui$ where the colour black appears for the $r$th time (counted from the left).

\subsection{Polynomial representation}
Let $\PollR$ be the direct sum 
\begin{eqnarray*}
\PollR&=&\bigoplus_{\ui\in \Icolnu}\bfk[y_1,\ldots,y_d]e(\ui),
\end{eqnarray*}
of polynomial rings,  where again $e(\ui)$ is just a formal symbol. We can also view $e(\ui)$ as a projector in $\PollR$ to the summand  $\bfk[y_1,\ldots,y_d]e(\ui)$. 

For $r\in [1;d-1]$ denote by $\partial_r$ the {\it Demazure operator} 
\begin{eqnarray}
\label{defDemazure}
\partial_r\colon \bfk[y_1,\ldots,y_d]\to \bfk[y_1,\ldots,y_d], \qquad
f\mapsto (f-s_r(f))/(y_r-y_{r+1}).
\end{eqnarray}
For each $i,j\in I$ such that $i\ne j$, consider the following polynomial $P_{i,j}(u,v)=(u-v)^{h_{i,j}}$. In the case $\ell=0$ we write $\R$ instead of $\lR$ and $\PolR$ instead of $\PollR$. (The algebra $\R$ is the usual KLR algebra.) Then we have the following faithful representation, see \cite[Sec.~2.3]{KL}. 
\begin{lem}
\label{lem-polrep_KLR}
The algebra $\R$ has a faithful representation on $\PolR$ such that

\begin{itemize}
\item the element $e(\ui)$ acts as the projector onto $\bfk[y_1,\ldots,y_d]e(\ui)$,

\item the element $y_re(\ui)$ acts by multiplication with $y_r$ on $\bfk[y_1,\cdots,y_d] e(\ui)$ and by zero on all other direct summands of $\PolR$,

\item the element $\psi_re(\ui)$ acts nontrivially only on $\bfk[y_1,\ldots,y_d]e(\ui)$ and there as 
\begin{eqnarray*}
fe(\ui)&\mapsto&
\begin{cases}
\partial_r(f)e(\ui) &\mbox{ if } j_r=j_{r+1},\\
P_{i_r,i_{r+1}}(y_r,y_{r+1})s_r(f)e(s_r(\ui)) &\mbox{ else.}
\end{cases}
\end{eqnarray*}
\end{itemize}
\end{lem}

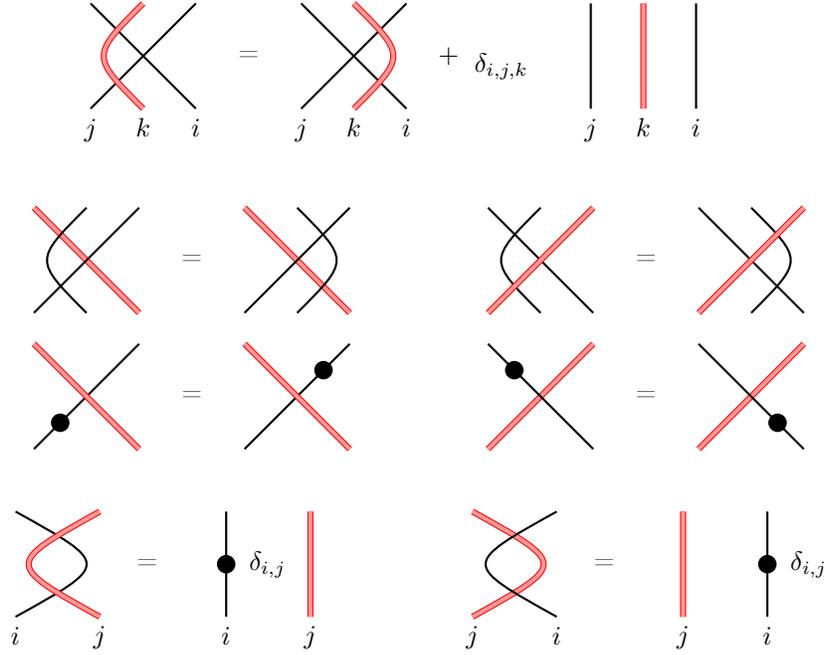
\begin{figure}
\begin{eqnarray*}
    \begin{tikzpicture}[scale=0.7, thick]
      \draw (-3,0)  +(1,-1) -- +(-1,1) node[at start,below]{$i$};
      \draw (-3,0) +(-1,-1) -- +(1,1)node [at start,below]{$j$};
      \draw[wei] (-3,0)  +(0,-1) .. controls (-4,0) ..  +(0,1)node [at start,below]{$k$};
      \node at (-1,0) {=};
      \draw (1,0)  +(1,-1) -- +(-1,1) node[at start,below]{$i$};
      \draw (1,0) +(-1,-1) -- +(1,1) node [at start,below]{$j$};
      \draw[wei] (1,0) +(0,-1) .. controls (2,0) ..  +(0,1)node [at start,below]{$k$};
\node at (2.8,0) {$+ $};
      \draw (6.5,0)  +(1,-1) -- +(1,1) node[at start,below]{$i$};
      \draw (6.5,0) +(-1,-1) -- +(-1,1) node [at start,below]{$j$};
      \draw[wei] (6.5,0) +(0,-1) -- +(0,1)node [at start,below]{$k$};
\node at (3.8,-.2){$ \delta_{i,j,k} $}  ;
 \end{tikzpicture}
\end{eqnarray*}
\begin{eqnarray*}
    \begin{tikzpicture}[scale=0.7,thick,baseline=2.85cm]
      \draw[wei] (-3,3)  +(1,-1) -- +(-1,1);
      \draw (-3,3)  +(0,-1) .. controls (-4,3) ..  +(0,1);
      \draw (-3,3) +(-1,-1) -- +(1,1);
      \node at (-1,3) {=};
      \draw[wei] (1,3)  +(1,-1) -- +(-1,1);
  \draw (1,3)  +(0,-1) .. controls (2,3) ..  +(0,1);
      \draw (1,3) +(-1,-1) -- +(1,1);    \end{tikzpicture}
&\quad\quad\quad&
    \begin{tikzpicture}[scale=0.7,thick,baseline=2.85cm]
      \draw (-3,3)  +(1,-1) -- +(-1,1);
      \draw (-3,3)  +(0,-1) .. controls (-4,3) ..  +(0,1);
      \draw[wei] (-3,3) +(-1,-1) -- +(1,1);
      \node at (-1,3) {=};
      \draw (1,3)  +(1,-1) -- +(-1,1);
  \draw (1,3)  +(0,-1) .. controls (2,3) ..  +(0,1);
      \draw[wei] (1,3) +(-1,-1) -- +(1,1);    \end{tikzpicture}
\end{eqnarray*}
\begin{eqnarray*}
    \begin{tikzpicture}[scale=0.7,thick]
  \draw(-3,6) +(-1,-1) -- +(1,1);
  \draw[wei](-3,6) +(1,-1) -- +(-1,1);
\fill (-3.5,5.5) circle (5pt);
\node at (-1,6) {=};
 \draw(1,6) +(-1,-1) -- +(1,1);
  \draw[wei](1,6) +(1,-1) -- +(-1,1);
\fill (1.5,6.5) circle (5pt);
    \end{tikzpicture}
&\quad\quad\quad&
    \begin{tikzpicture}[scale=0.7,thick]
  \draw[wei](-3,6) +(-1,-1) -- +(1,1);
  \draw(-3,6) +(1,-1) -- +(-1,1);
\fill (-3.5,6.5) circle (5pt);
\node at (-1,6) {=};
 \draw[wei](1,6) +(-1,-1) -- +(1,1);
  \draw(1,6) +(1,-1) -- +(-1,1);
\fill (1.5,5.5) circle (5pt);
    \end{tikzpicture}
    \end{eqnarray*}
   
\begin{eqnarray*}
  \begin{tikzpicture}[scale=0.7,thick]
    \draw (-2.8,0)  +(0,-1) .. controls (-1,0) ..  +(0,1) node[below,at start]{$i$};
       \draw[wei] (-1.2,0)  +(0,-1) .. controls (-3,0) ..  +(0,1) node[below,at start]{$j$};
           \node at (-.3,0) {=};
    \draw[wei] (2.8,0)  +(0,-1) -- +(0,1) node[below,at start]{$j$};
    \draw (1.2,0)  +(0,-1) -- +(0,1) node[below,at start]{$i$};
       \fill (1.2,0) circle (5pt) node[right=5pt]{$\delta_{i,j}$};
      \end{tikzpicture}   
  &\quad\quad\quad&
  \begin{tikzpicture} [scale=0.7,thick]
             \draw[wei] (-2.8,0)  +(0,-1) .. controls (-1,0) ..  +(0,1) node[below,at start]{$j$};
  \draw (-1.2,0)  +(0,-1) .. controls (-3,0) ..  +(0,1) node[below,at start]{$i$};
           \node at (-.3,0) {=};
    \draw (2.8,0)  +(0,-1) -- +(0,1) node[below,at start]{$i$};
       \draw[wei] (1.2,0)  +(0,-1) -- +(0,1) node[below,at start]{$j$};
       \fill (2.8,0) circle (5pt) node[right=5pt]{$\delta_{i,j}$};        
  \end{tikzpicture}
  \end{eqnarray*}
   \caption{Tensor product algebra relations II involving red strands}
   \label{reltensor2}
\end{figure}

%\end{df}

%Let $\rmPol_{\ell}$ be the direct sum of polynomial rings $\rmPol_{\ell}=\bigoplus_{\ui\in I^{\bfa,\bfQ}_{\col}}\bfk[y_1,\ldots,y_d]e(\ui)$, where $e(\ui)$ is just a formal symbol. We can also see $e(\ui)$ as a projector in $\rmPol$ to $\bfk[y_1,\ldots,y_d]e(\ui)$.

The following may be deduced from \cite[Prop.~4.7,~Prop.~4.9]{SW} (see also \cite[Fig.~3]{SW}). Hereby  $\PollR$ is realized as a subring of $\bigoplus_{\ui\in \Icolnu}\bfk[Y_1,\ldots,Y_{\ell+d}]e(\ui)$ via
$
P(y_1,\ldots,y_r)e(\ui)\mapsto P(Y_{1_\ui},\ldots,Y_{d_\ui})e(\ui).
$

\begin{lem}
\label{lem-polrep_tenspr}
The algebra $\lR$ has a faithful representation on $\PollR$ such that 
\begin{itemize}
\item the element $e(\ui)$ acts as the projector onto $\bfk[y_1,\ldots,y_d]e(\ui)$,

\item the element $y_re(\ui)$ acts by multiplication with $y_r$ on $\bfk[y_1,\cdots,y_d] e(\ui)$ and by zero on other direct summand of $\PollR$,

\item the element $\psi_re(\ui)$ acts only nontrivially on $\bfk[y_1,\ldots,y_d]e(\ui)$, where it sends $fe(\ui)$ to 
\end{itemize}
\begin{eqnarray*}
&&
%\left\{\begin{array}{lll}
\begin{cases}
\partial_r(f)e(\ui) &\mbox{\rm if } c(j_r)=c(j_{r+1})=0,~ j_r=j_{r+1},\\
P_{\gamma(j_r),\gamma(j_{r+1})}(Y_r,Y_{r+1})s_r(f)e(s_r(\ui)) &\mbox{\rm if } c(j_r)=c(j_{r+1})=0,~ j_r\ne j_{r+1},\\
0 &\mbox{\rm if } c(j_r)=c(j_{r+1})=1,\\
Y_{r+1}s_r(f)e(s_r(\ui)) &\mbox{\rm if } c(j_r)=0,~ c(j_{r+1})=1,~ \gamma(j_r)=\gamma(j_{r+1}),\\
s_r(f)e(s_r(\ui)) &\mbox{\rm for all other cases}.\\
%\end{array}\right.
\end{cases}
\end{eqnarray*}
\end{lem}

\subsection{Completion}

Let $\mathfrak{m}$ be the ideal in $\bfk[y_1,\ldots,y_d]$ generated by all $y_r$, $1\leqslant r\leqslant d$. 

\begin{df}
Denote by $\cR$ the completion of the algebras $\R$ at the sequence of ideals $\R \mathfrak{m}^j \R$. Denote by $\clR$ the completion of the algebra $\lR$ at the sequence of ideals $\lR \mathfrak{m}^j \lR$.
\end{df}

\begin{rk}
%\begin{enumerate}
%\item 
The faithful polynomial representation of $\R$ on $\PolR$ (see Lemma~\ref{lem-polrep_KLR}) yields a faithful representation of $\cR$ on 
\begin{eqnarray}
\label{faithful1}
\cPolR=\bigoplus_{\ui\in I^\nu}\bfk[[y_1,\cdots,y_d]]e(\ui).
\end{eqnarray}
%\item 
The faithful polynomial representation of $\lR$ on $\PollR$ (see Lemma~\ref{lem-polrep_tenspr}) yields a faithful representation of $\clR$ on 
\begin{eqnarray}
\label{faithful2}
\cPollR&=&\bigoplus_{\ui\in \Icolnu}\bfk[[y_1,\cdots,y_d]]e(\ui).
\end{eqnarray}
%\end{enumerate}
\end{rk}

\subsection{The isomorphisms $ \cR\simeq \cHeck$ and $ \clR\simeq\clHeck$}
\label{subs-isom_lHeck-tens-compl}
Fix $q\in\bfk$ such that $q\not\in\{0,1\}$. Fix an $\ell$-tuple $\bfQ=(Q_1,\ldots,Q_\ell)\subset (\bfk^*)^\ell$.

Consider the following set
\begin{equation}
\label{Corona}
\calF=\{q^nQ_m\mid n\in\bbZ, m\in[1;\ell]\}\subset\bfk^*.
\end{equation}
We can consider $\calF$ as a vertex set of a quiver $\Gamma_\calF$ such that for $i,j\in\calF$ we have an arrow $i\to j$ if and only if we have $j=qi$. If $q$ is an $e$th root of unity, then the quiver $\Gamma_\calF$ is a disjoint union of at most $\ell$ oriented cycles of length $e$. If $q$ is not a root of unity, then the quiver $\Gamma_\calF$ is a disjoint union of at most $\ell$ (two-sided) infinite oriented linear quivers. Then $\bfQ$ can be considered as an $\ell$-tuple of vertices of the quiver $\Gamma_\calF$. In this section we assume that the KLR algebra and the tensor product algebra are defined with respect to the quiver $\Gamma_\calF$. In particular we have $I=\calF$. We also assume $\nu=\bfa$. Then we have $I^\nu=\frakS_d\bfa$.

First, we recall the isomorphism $\cR\simeq \cHeck$ from \cite[Thm.~7.3]{MS}. For this we identify the vector spaces $\cPolR$ and $\cPolHeck$ via
\begin{eqnarray}
\label{identifyPol}
\cPolR\to\cPolHeck, && -i_ry_re(\ui)\mapsto (X_r-i_r)e(\ui).
\end{eqnarray}
\begin{prop}[{\cite[Thm.~7.3]{MS}}]
\label{prop-isom-Heck-KLR-comp}
There is an isomorphism $\cR\simeq \cHeck$ of algebras  sending $e(\ui)$ to $e(\ui)$, $y_re(\ui)$ to $-\gamma(i_r)^{-1}(X_r-\gamma(i_r))e(\ui)$ and  $\psi_re(\ui)$ to the expression in \eqref{imagepsi} below.
\end{prop}
\begin{proof}
It is enough to check that the induced actions of the generators and their images agree on the (faithful) polynomial representations \eqref{identifyPol}. This is straightforward noting that the element $\psi_re(\ui)\in \cR$ acts as
\begin{eqnarray}
\label{imagepsi}
%\left\{
%\begin{array}{lll}
\begin{cases}
-\frac{i_r}{X_r-qX_{r+1}}(T_r+1) e(\ui)&\mbox{if }i_r=i_{r+1},\\
i_r^{-1}q^{-1}((X_r-X_{r+1})T_r+(q-1)X_{r+1}) e(\ui)&\mbox{if } qi_r=i_{r+1},\\
\left(1-\frac{X_r-X_{r+1}}{X_r-qX_{r+1}}(T_r+1)\right)e(\ui) &\mbox{else},\\
\end{cases}
%\end{array}
%\right.
\end{eqnarray}
and so the claim follows.
\end{proof}

We extend this now to an isomorphism $\clHeck\simeq \clR$. First, note that we have an obvious bijection $\Icolnu\simeq J^{\ell,d}\times \frakS_d\bfa$. This is important because the algebra $\clR$ has idempotents parametrised by $\Icolnu$ and the algebra $\clHeck$ has idempotents parametrised by $J^{\ell,d}\times \frakS_d\bfa$.

We identify the vector spaces underlying the polynomial representations, $\cPollR$ for $\clR$ and $\cPollHeck$ for  $\clHeck$, via 
\begin{eqnarray}
\label{identifyPol2}
\cPollR\to\cPollHeck, \qquad -\gamma(i_r)Y_re(\ui)\to (X_r-\gamma(i_r))e(\ui) \quad \mbox{ if $c(i_r)=0$}.
\end{eqnarray}
(Recall that both $Y_re(\ui)$ and $X_re(\ui)$ are zero if $c(i_r)=1$.)

\begin{thm}
\label{thm-isom-lHeck-tens-comp}
There is an isomorphism of algebras $ \clR\simeq\clHeck$ extending the isomorphism from Proposition~\ref{prop-isom-Heck-KLR-comp}.
\end{thm}
\begin{proof}
Abbreviate $(\dagger)=q^{-1}\frac{1}{\gamma(i_r)}((X_r-X_{r+1})T_r+(q-1)X_{r+1}) $. We claim that sending $\psi_re(\ui)\in \clR$ to the element 
\begin{eqnarray*}
\begin{cases}
\frac{-\gamma(i_r)}{X_r-qX_{r+1}}(T_r+1) &\mbox{if }i_r=i_{r+1},~ c(i_r)=c(i_{r+1})=0,\\
(\dagger)&\mbox{if } q\gamma(i_r)=\gamma(i_{r+1}),~c(i_r)=c(i_{r+1})=0,\\
(1-\frac{X_r-X_{r+1}}{X_r-qX_{r+1}}(T_r+1))e(\ui) &\mbox{for all other cases with } c(i_r)=c(i_{r+1})=0,\\
T_re(\ui) &\mbox{if } c(i_r)=1,~c(i_{r+1})=0,\\
\frac{-1}{\gamma(i_r)}T_re(\ui) &\mbox{if } c(i_r)=0,~c(i_{r+1})=1,~\gamma(i_r)=\gamma(i_{r+1}),\\
\frac{1}{(X_{r+1}-\gamma(i_{r+1}))}T_re(\ui), &\mbox{if } c(i_{r})=0,~c(i_{r+1})=1,~\gamma(i_r)\ne \gamma(i_{r+1}),
\end{cases}
\end{eqnarray*}
defines an isomorphism as claimed.  Clearly this makes the map unique, since we specified the image of on a set of generators and moreover surjective, since the generators of $\clHeck$ are in the image. To show well-definedness and that it is an isomorphism it suffices to show that the action of the generators agrees with that  of their images on the (faithful) polynomial representations \eqref{identifyPol2}. For the idempotents $e(\ui)\in \clR$ this is clear, and the element $Y_re(\ui)\in \clR$ acts as $-\gamma(i_r)^{-1}(X_r-\gamma(i_r))e(\ui)\in\clHeck$ if $c(i_r)=0$. (Recall that if $c(i_r)=1$ then both $X_re(\ui)$ and $Y_re(\ui)$ are zero.)  Since $\psi_re(\ui)\in \clR$ acts exactly as its proposed image (recalling that if $c(i_r)=c(i_{r+1})=1$ then both $T_re(\ui)$ and $\psi_re(\ui)$ are zero), the claim follows. 
\end{proof}

\begin{rk}
It is useful to give an explicit inverse of the isomorphism from Theorem~\ref{thm-isom-lHeck-tens-comp}.
The element $T_re(\ui)$ acts on the polynomial representation by the same operator as
\begin{eqnarray*}
\begin{cases}
\left(-1+{(q-1+Y_r-qY_{r+1})}\psi_r\right) e(\ui) &\mbox{if }i_r=i_{r+1},~ c(i_r)=c(i_{r+1})=0,\\
\left(\frac{q(q-1)(Y_{r+1}-1)}{1-q-Y_r+qY_{r+1}}+\frac{q\psi_r}{q-1-qY_r+Y_{r+1}}\right)e(\ui) &\mbox{if }q\gamma(i_r)=\gamma(i_{r+1}),~ c(i_r)=c(i_{r+1})=0,\\
\left(\frac{(1-q)\gamma(i_{r+1})(1-Y_{r+1})}{\gamma(i_r)(1-Y_r)-\gamma(i_{r+1})(1-Y_{r+1})}-\right.&\\
\left.\frac{\gamma(i_{r+1})(1-Y_r)-q\gamma(i_{r})(1-Y_{r+1})}{\gamma(i_{r+1})(1-Y_r)-\gamma(i_{r})(1-Y_{r+1})}\psi_r\right)e(\ui), &\mbox{otherwise, with } c(i_r)=c(i_{r+1})=0,\\
\psi_re(\ui) &\mbox{ if } c(i_r)=1, c(i_{r+1})=0,\\
(\gamma(i_r)(1-Y_{r+1})-\gamma(i_{r+1}))\psi_re(\ui) &\mbox{ if }\gamma(i_{r})\ne \gamma(i_{r+1}), c(i_r)=0, c(i_{r+1})=1,\\
-\gamma(i_r)\psi_re(\ui) &\mbox{ if }\gamma(i_{r})=\gamma(i_{r+1}), c(i_r)=0, c(i_{r+1})=1.\\
\end{cases}
\end{eqnarray*}
\end{rk}

\section{Higher level affine Schur algebras $\lS$}
\label{sec-Schur}
We recall the definition of the (ordinary) affine Schur algebra as it appears for instance in \cite{Greenaff}, \cite{MS}, \cite{vigneras} and then generalize it to a higher level version.
\subsection{Affine Schur algebras}
\label{subs-affSchur}

For each non-negative integer $d$, a \emph{composition} of $d$ is a tuple $\lambda=(\lambda_1,\hdots,\lambda_r)$ (the number $r$, called the {\it length} ${l}(\lambda)$ of $\lambda$, is not fixed) such that $\sum_{i=1}^{r}\lambda_i=d$ and $\lambda_i>0$. If $\lambda$ is a composition of $d$, we write $|\lambda|=d$. Denote by $\calC_d$ the set of compositions of $d$. We use the convention that $\calC_0$ contains a unique composition which is empty. For each $\lambda=(\lambda_1,\hdots,\lambda_r)\in \calC_d$ denote by $\frakS_\lambda$ the parabolic  (or Young) subgroup 
\begin{eqnarray}
\label{Young}
\frakS_\lambda&=&\frakS_{\lambda_1}\times\hdots\times \frakS_{\lambda_r}\subset \frakS_d.
\end{eqnarray}
Its unique longest element is denoted by $w_\lambda$. Moreover, let $\Drep_{\lambda,\mu}$ be the set of shortest length representatives for the double cosets  $\frakS_\lambda\backslash\frakS_d/\frakS_\mu$. We also write $\Drep_{\emptyset,\mu}$ and $\Drep_{\lambda,\emptyset}$ for the sets of shortest length representatives of the cosets $\frak{S}_d/\frak{S}_\mu$ and $\frak{S}_\lambda\backslash \frak{S}_d$  respectively. Attached to this subgroup, let $m_\lambda\in \Heck$ be defined by 
\begin{eqnarray}
\label{defm}
m_\lambda&=&\sum_{w\in \frakS_\lambda}(-q)^{l(w_\lambda)-l(w)}T_w.
\end{eqnarray}
We consider $m_\lambda \Heck$ as a right $\Heck$-module.

\begin{df}
The {\it affine Schur algebra} is the algebra
\begin{eqnarray}
\label{affSchurDef}
\S&=&\End_{\Heck}\left(\bigoplus_{\lambda\in \calC_d}m_\lambda \Heck\right).
\end{eqnarray}
The algebra $\S$ has idempotents $e(\lambda)$, $\lambda\in\calC_d$ given by the projection to $m_\lambda \Heck$.
\end{df}

\subsection{Generators of $\S$ and thick calculus}
Next we  introduce the {\it thick calculus} for the algebra $\S$.

\label{subs-gen_S}
Let $\lambda, \mu\in \calC_d$ and assume that $\mu$ is obtained from $\lambda$ by splitting one component of $\lambda$. In other words, there is an index $t$ such that $\mu$ is of the form $(\lambda_1,\hdots,\lambda_{t-1},\lambda'_t,\lambda''_t,\lambda_{t+1},\hdots,\lambda_{l(\lambda)})$,  
where $\lambda'_t$ and $\lambda''_t$ are positive integers such that $\lambda'_t+\lambda''_t=\lambda_t$. In this case we say that $\mu$ is a \emph{split} of $\lambda$ and that $\lambda$ is a \emph{merge} of $\mu$ (at position $t$).

\begin{df}
Assume  $\mu$ is a split of $\lambda$. We define the special elements in $\S$:
\begin{eqnarray*}
\text{the {\it split morphism}} &&m_\lambda x\mapsto m_\mu x\in \Hom_{\Heck}(m_\lambda\Heck,m_\mu\Heck),\\
\text{the {\it merge morphism}} &&m_\mu x\mapsto m_\lambda x\in \Hom_{\Heck}(m_\mu\Heck,m_\lambda\Heck).
\end{eqnarray*}
\end{df}
More generally, if  $\mu$ is a refinement of the composition $\lambda$ we have the corresponding split morphism, denoted $(\lambda\rightarrow\mu)$, and the corresponding merge morphism, denoted  $(\mu\rightarrow\lambda)$, defined in the obvious way. They are the compositions of the splits (respectively merges) describing the refinement. Note that the order in the composition does not matter because of the associativity property of splits and merges, \cite[Lemma~6.5 (twisted with the automorphism $\sharp$)]{MS}.  The idempotents $e(\lambda)$, splits, merges and multiplication with (invariant) polynomials generate the algebra $\S$ see \cite[Prop.~6.19]{MS}.

We draw the generators as diagrams that are similar to the diagrams for $\Heck$ from Definition~\ref{def-extaffHeck_diag}. The differences are that the black strands are now allowed to have a higher thickness (corresponding to multiplicities of the labels given by a nonnegative integer),  the diagrams representing the generators are now of the form
\begin{equation}
\label{diag-split-merge}
\tikz[thick,xscale=2.5,yscale=-1.5]{
\draw (0,0) node[above] {$a$} to [out=90,in=-90](.3,.5)
(.6,0) node[above] {$b$} to [out=90,in=-90] (.3,.5)
(.3,.5) -- (.3,.8) node[below] {$a+b$};
}
\qquad
\tikz[thick,xscale=2.5,yscale=1.5]{
\draw (0,0) node[below] {$a$} to [out=90,in=-90](.3,.5)
(.6,0) node[below] {$b$} to [out=90,in=-90] (.3,.5)
(.3,.5) -- (.3,.8) node[above] {$a+b$};
}
\end{equation}
and each strand with thickness $b$ is allowed to carry now any symmetric Laurent polynomial in $b$ variables instead of dots.

\begin{df}
Let $\lambda,\mu\in\calC_d$. We draw the idempotent $e(\lambda)\in \S$ given by the identity endomorphism of the right $\Heck$-module $m_\lambda\Heck$ as a diagram with ${l}(\lambda)$ vertical strands labelled by the parts of $\lambda$,

$$
\tikz[thick,xscale=2.5,yscale=1.5]{
\node at (-.9,.25) {$e(\lambda)\quad\mapsto\quad$};
\draw (-0.3,0) --(-0.3,.5) node[below,at start]{$\lambda_1$};
\draw (-0,0) --(-0,.5) node[below,at start]{$\lambda_2$};
\node at (0.25,.25) {$\cdots$};
\draw (0.5,0) --(0.5,.5) node[below,at start]{$\lambda_{{l}(\lambda)}$};
}
$$

Let $f\in \bfk[x^{\pm 1}_1,\hdots,x^{\pm 1}_d]^{\frakS_\lambda}$ be of the form $f=f_1\cdots f_{l(\lambda)}$, where $f_j$ is a symmetric Laurent polynomial containing only variables with indices in $[\lambda_1+\ldots+\lambda_{j-1}+1;\;\lambda_1+\ldots+\lambda_{j}]$. Then we associate to $fe(\lambda)\in \S$ the diagram

$$
\tikz[thick,xscale=2.5,yscale=1.5]{
\node at (-1,.35) {$fe(\lambda)\quad\mapsto\quad$};
\draw (-0.3,0) --(-0.3,.2) node[below,at start]{$\lambda_1$};
\draw (-0.3,0.5) --(-0.3,.7);
\draw (-0.15,0.2) rectangle (-0.45,0.5);
\node at (-0.3,0.35) {$f_1$};
\draw (0.1,0) --(0.1,.2) node[below,at start]{$\lambda_2$};
\draw (0.1,0.5) --(0.1,.7);
\draw (-0.05,0.2) rectangle (0.25,0.5);
\node at (0.1,0.35) {$f_2$};
\node at (0.45,.35) {$\cdots$};
\draw (0.8,0) --(0.8,.2) node[below,at start]{$\lambda_{{l}(\lambda)}$};
\draw (0.8,0.5) --(0.8,.7);
\draw (0.6,0.2) rectangle (1.0,0.5);
\node at (0.8,0.35) {$f_{{l}(\lambda)}$};
}
$$

Since any $f\in \bfk[x^{\pm 1}_1,\hdots,x^{\pm 1}_d]^{\frakS_\lambda}$ can be written as a sum of polynomials of the form $f_1\ldots f_{{l}(\lambda)}$, the notation $fe(\lambda)$ makes sense for any such $f$. In the special case where $\la_i=1$ the $i$th strand is allowed to carry any Laurent polynomial in the variable $x_i$, in particular it can carry dots as in our notation before.  

We assign to a split $\lambda\rightarrow \mu$ of the form $(a+b)\to (a,b)$ (respectively  a merge $(a,b)\to(a+b)$) the first (resp. second) diagram in \eqref{diag-split-merge}, and if the compositions have more parts we add additionally vertical strands to the left and to the right labelled by the remaining components. 
\end{df}

\begin{rk}
\label{rkcat}
The algebra $\S$ can be realized more conceptually as a quotient of the algebra structure on the direct sum of  homomorphism spaces of certain objects in the following category with $I_b=(\mathbb{Z}_{>0},+)$. The objects we take are all tensor products of black labels such that the sum of the labels is $d$, and $R_a$ is the ring of symmetric Laurent polynomials in $a$ variables. The strands with labels $\boxed{f}$ is then the image the dot morphism for $f$. Note that hereby $I_r$ does not play any role.
\end{rk}

\begin{df}
\label{Corona2} 
 Let $I_r$ be a set and $I_b$ an additive monoid. The {\it universal thickened higher level category} corresponding to this pair is the $\bfk$-linear strict monoidal category generated as monoidal category by objects  ${a}\in I_b$, called {\it black labels} and objects $Q\in I_r$, called {\it red labels}, and by the following morphisms: 
\begin{itemize}
\item the {\it split morphisms}  $a+b\longrightarrow a\otimes b$ and the {\it merge morphisms} $a\otimes b\longrightarrow a+b$ for any $a,b\in I_b$ given abstractly by diagrams \eqref{diag-split-merge},
\item the {\it crossings} 
$\TikZ{[thick,scale=.5] \draw  (1,1) node{} to (0,0)node{};\draw[wei] (1,0) node{} to (0,1)node{}; \node at (1.2,-0.3) {\tiny Q};\node at (0,-0.3) {\tiny a};}
:\;{a}\otimes{Q}\longrightarrow{Q}\otimes {a}$ and $\TikZ{[thick,scale=.5] \draw  (1,0) node{} to (0,1)node{};\draw[wei] (1,1) node{} to (0,0)node{}; \node at (-0.2,-0.3) {\tiny Q};, \node at (1,-0.3) {\tiny i}; }
:\;{Q}\otimes {i}\longrightarrow{i}\otimes {Q},$
\item the {\it dot morphisms} $\TikZ{[thick, scale=.5] \draw (0,0) node{} to (0,1) node{} (0,0.5) node[fill,circle,inner sep=1.5pt]{};\node at (0,-0.3) {\tiny a};\node at (0.4,0.5) {\tiny f};}:{a}\longrightarrow  {a}$ for $a\in I_b$ and $f\in R_a$ for some fixed commutative ring $R_a$ depending on $a$. 
%%The diagram $\TikZ{[scale=.45] \draw
%		(0,0) node{} to (0,0.5){}
%                  (0,0.5) node[fill,circle,inner sep=1.5pt]{}
%		(-0.75,0.5) node{${k}$}
%		;}$ denotes the $k$-fold (vertical) composition of the dot morphism. 
\end{itemize}
The monoidal structure is again horizontal placement, the composition of morphisms vertical placement. We impose the relation that the composition of two dot morphisms with the same  thickness, say for $f_1$ and $f_2\in R_a$,  equals the dot morphism for $f_1f_2\in R_a$. 
\end{df}

\begin{rk}
The affine Hecke algebra $\Heck$ is an idempotent truncation of $\S$. The thick calculus in $\S$ generalizes the diagrammatic calculus in $\Heck$ in the sense that  each usual Hecke strand can be viewed as a strand of thickness $1$ and $R_1=\bfk[X]$, the usual polynomial ring in one variable. The dot morphism using $\bullet$ is just the abbreviation for the dot morphism for $X\in R_1$.
\end{rk}

\begin{rk}
One would like to have an analogue of Lemma~\ref{lem:twodefs} for the Schur algebras, that is an explicit presentation in terms of the  generators in the universal thickened higher level category modulo explicit relations. This is so far not known. In \cite{Greenaff} solved the analogous problem in case of generic $q$ for a Morita equivalent algebra using similar generators. It would be nice to be able to generalize this result to our setting.
\end{rk}

It is also convenient to explicitly specify a few more elements of $\S$. For this let $\lambda,\mu\in\calC_d$ and assume that $\mu$ is obtained from $\lambda$  by swapping $\lambda_t$ and $\lambda_{t+1}$ for some $t$. Let $\nu$ be the merge of $\lambda$ at position $t$. Denote by $w(\lambda/\mu)$ the shortest coset representative in $\frakS_\nu/ \frakS_{\lambda}$ of $w_{\nu}$. (As a permutation diagram one might draw a cross as displayed in \eqref{crossdiag} indicating that $\lambda_t$ elements get swapped with $\lambda_{t+1}$ elements keeping the order inside the groups.) Then with $T=T_{w(\lambda/\mu)}\in \Heck$ it holds $T m_\lambda=m_\mu T$ in $\Heck$.

\begin{df} The corresponding  \emph{black crossing} is the element of $\S$ which is only nonzero on the summand $m_\lambda \Heck$ and there given by $m_\lambda \Heck\to m_\mu \Heck$, $m_\lambda h\mapsto T m_\lambda h=m_\mu T h$. We draw this element in the following way.

\begin{equation}
\label{crossdiag}
\tikz[thick,xscale=2.5,yscale=1.5]{
\draw (-1,0) --(-1,.5) node[below,at start]{$\lambda_1$};
\node at (-0.75,0.25) {$\cdots$};
\draw (-0.6,0) --(-0.3,.5) node[below,at start]{$\lambda_t$};
\draw (-0.3,0) --(-0.6,.5) node[below,at start]{$\lambda_{t+1}$};
\node at (-0.15,0.25) {$\cdots$};
\draw (0.1,0) --(0.1,.5) node[below,at start]{$\lambda_{{l}(\lambda)}$};
}
\end{equation}
\end{df}

\begin{lem}
\label{lem-black_cross}
A black crossing can be written as a product of splits, merges and Laurent polynomials.
\begin{proof}
 \cite[Prop.~6.19]{MS} using  \cite[(3.6)]{MS} and the definition \cite[(4.3)]{MS}. 
\end{proof}
\end{lem}

\subsection{Demazure operators}
\label{subs_Demazure}
For each $w\in\frakS_d$, fix a reduced expression $w=s_{k_1}\ldots s_{k_r}$ and define  $\partial_w=\partial_{k_1}\ldots\partial_{k_r}$ using the Demazure operators from \eqref{defDemazure}. This definition is independent of the choice of a reduced expression, see \cite[Thm.~1]{Demazure}. 

\begin{df}
Set $D_d=\partial_{w_d}$, where $w_d$ is the longest element in $\frakS_d$.
For  positive integers $a$ and $b$ such that $a+b=d$ let $D_{a,b}=\partial_{w_{a,b}}$ with 
$w_{a,b} \in\frakS_d$ the permutation
$$
w_{a,b}(i)=
\begin{cases}
i+b & \mbox{ if } 1\leqslant i\leqslant a,\\
i-a & \mbox{ if } a< i\leqslant a+b.
\end{cases}
$$
\end{df}

We need the following well-known symmetrizing properties of these operators:
\begin{lem}
\begin{enumerate}
\item For each polynomial $f$, the polynomial $D_d(f)$ is symmetric.
\item In case $f$ is $\frakS_a\times \frakS_b$-symmetric, then $D_{a,b}(f)$ is symmetric.
\end{enumerate}
\end{lem}
\begin{proof}
The first property follows directly from the definition. Moreover, it is easy to see that each symmetric polynomial is in the image of the operator $D_d$.

Then the second statement follows because for each $\frakS_a\times\frakS_b$-symmetric polynomial $f$ we can find a polynomial $g$ such that $D_{a,b}(f)=D_{a+b}(g)$.
\end{proof}

\subsection{Polynomial representation of $\S$}
\label{subs-polrep_S}

By definition, the algebra $\S$ has a faithful representation on the vector space $\bigoplus_{\lambda\in \calC_d}m_\lambda \Heck$. We will construct a faithful polynomial representation of $\S$ on 
\begin{eqnarray*}
\SPol&=&\bigoplus_{\lambda\in \calC_d}\bfk[x^{\pm 1}_1,\hdots,x^{\pm 1}_d]^{\frakS_\lambda}e(\lambda)
\end{eqnarray*}
realized as a subrepresentation of the defining representation.

Fix $\lambda\in \calC_d$. We will say that the indices $i,j\in[1;d]$ {\it are in the same block for $\lambda$} if there exists some $t$ such that 
$$
\sum_{a=1}^{t-1}\lambda_a<i,j\leqslant \sum_{b=1}^{t}\lambda_a.
$$  

\begin{df}
\label{defpolys}
We consider the following polynomials depending on $\lambda$: 
$$
\p_\lambda=\prod_{i<j}(x_i-qx_j),\quad \q_\lambda=\prod_{i<j}(x_j-qx_i)
$$
where the product is taken over all $i,j\in [1;d]$ such that $i$ and $j$ are in the same block with respect to $\lambda$. 
Set also $n_\lambda=\sum_{w\in\frakS_\lambda}T_w$ and $n'_\lambda=\sum_{w\in \Drep_{\lambda,\emptyset}}T_w$.
\end{df}
For instance, if $\lambda=(2,3)$ then 
\begin{eqnarray*}
\p_\lambda&=&(x_1-qx_2)(x_3-qx_4)(x_3-qx_5)(x_4-qx_5),\\
\q_\lambda&=&(x_2-qx_1)(x_4-qx_3)(x_5-qx_3)(x_5-qx_4).
\end{eqnarray*}

\begin{df}
For each $\lambda\in \calC_d$ we define the following linear map
\begin{eqnarray}
\label{inclusiondef}
\Phi_\lambda\colon\;\bfk[x^{\pm 1}_1,\hdots,x^{\pm 1}_d]^{\frakS_\lambda}\to m_\lambda \Heck,&&
f\mapsto m_\lambda \p_\lambda f n'_\lambda.
\end{eqnarray}
which is in fact an inclusion by Corollary~\ref{coro_exch-mu-Pol} below, since $m_\lambda\p_\lambda f n'_\lambda=\q_\lambda f n_d.$
\end{df}

\begin{lem}
\label{lem-polrep-S-inside}
Let $\lambda,\mu\in  \calC_d$ and assume that $\mu$ is a split of $\lambda$. 
\begin{enumerate}
\item The split in $\Hom_{\Heck}(m_\lambda\Heck,m_\mu\Heck)$ applied to the image of $\Phi_\lambda$ is contained in the image of $\Phi_\mu$. 
\item The merge in $\Hom_{\Heck}(m_\mu\Heck,m_\lambda\Heck)$ applied to the image of $\Phi_\mu$ is contained in the image of $\Phi_\lambda$. 
\end{enumerate}
\end{lem}

The proof will be given in Section~\ref{subs_some-rel}. We will also need the following auxiliary polynomials. Assume that $a$ and $b$ are positive integers such that $a+b=d$.
\begin{eqnarray*}
\p'_{a,b}=\prod_{1\leqslant i\leqslant a< j\leqslant b }(x_i-qx_j),
&&
\q'_{a,b}=\prod_{1\leqslant i\leqslant a< j\leqslant b }(x_j-qx_i).
\end{eqnarray*}
\begin{prop}
\label{prop-polrep_S}
The algebra $\S$ has a faithful representation in $\SPol$ such that  the generators act as follows, using the abbreviation  $P=\bfk[x^{\pm 1}_1,\hdots,x^{\pm 1}_d]$. 
\begin{itemize}
\item The idempotent $e(\lambda)$, $\lambda\in \calC_d$, acts on $\SPol$ as the projection to $P^{\frakS_\lambda}e(\lambda)$.
\item For each $g\in  P^{\frakS_\lambda}$, $\lambda\in \calC_d$, the element $ge(\lambda)$ sends $fe(\lambda)\in P^{\frakS_\lambda}e(\lambda)$ to $gfe(\lambda)$.
\item Assume $\mu$ is a split of $\la$ at position $j$.  Then the split map $\la\rightarrow\mu$ acts by sending $fe(\lambda)\in P^{\frakS_\lambda}e(\lambda)$ to $\q'_{a,b}fe(\mu)$ and the merge map acts by sending $fe(\mu)\in P^{\frakS_\mu}$ to $D_{a,b}(f)e(\lambda)$ in case $\lambda=(a+b)$ and $\mu=(a,b)$ with $a+b=d$. In the general case they act by the same formulas with $a=\mu_j$, $b=\mu_{j+1}$, where $\q'_{a,b}$ and $D_{a,b}$ are defined with respect to the variables $x_i$ such that $i\in [\lambda_1+\ldots+\lambda_{j-1}+1;\lambda_1+\ldots+\lambda_{j-1}+\lambda_j]$.
\end{itemize}
\end{prop}
\begin{proof}
The existence of such a representation follows  from \eqref{affSchurDef}, Lemma~\ref{lem-polrep-S-inside}  and from the fact that the algebra $\S$ is generated by the idempotents $e(\lambda)$, splits, merges and multiplications with (invariant) polynomials. 

Assume this representation is not faithful. Then we can find $\lambda,\mu\in \calC_d$ and a nonzero $\phi\in \Hom_{\Heck}(m_\lambda\Heck,m_\mu\Heck)$ such that $\phi$ acts by zero on the polynomial representation. Let us compose $\phi$ with the split $(d)\to\lambda$ on the right. Then we get  (as splits are injective) a nonzero element $\psi\in \Hom_{\Heck}(m_d\Heck,m_\mu\Heck)$ that acts by zero. By construction of the polynomial representation, this implies $\psi(m_d\p_d)=0$ and thus $\psi(m_d)\p_d=0$. Since  $\Heck$ is a free right $P$-module, this implies $\psi(m_d)=0$ and thus $\psi=0$. This is a contradiction.
\end{proof}

\subsection{Some useful relations in the affine Hecke algebra}
\label{subs_some-rel}
To prove Lemma~\ref{lem-polrep-S-inside} we need to establish some explicit formulas which we think are of interest by themselves. In particular we want to understand the action of the special elements 
\begin{equation*}
m_d=\sum_{w\in{\frakS_d}}(-q)^{{l}(w_d)-{l}(w)}T_w,
\quad
n_d=\sum_{w\in{\frakS_d}}T_w,\quad
n'_{a,b}=\sum_{w\in \Drep_{(a,b),\emptyset}}T_w
\end{equation*}
from Sections~\ref{subs-gen_S},~\ref{subs-polrep_S} on the polynomial representation of  $\Heck$. 

\begin{lem} \hfill
\label{lem_m-on-Pol-rep}
\begin{enumerate}
\item The element $m_d$ acts on the polynomial representation as $D_d\q_d$.
\item The element $n_d$ acts on the polynomial representation as $\p_d D_d$.
\end{enumerate}
\end{lem}

\begin{proof}
Let $A=\bfk(X_1,\ldots,X_d)\#\bfk[\frak{S}_d]$ be the subalgebra of linear endomorphisms of $\bfk(X_1,\ldots,X_d)$ generated the multiplications with the $X_i$'s and by the permutations of variables for $w\in\frakS_d$. The algebra $A$ is free as a left $\bfk(X_1,\ldots,X_d)$-module with for instance the bases 
\begin{eqnarray}
\{w\mid w\in \bfk[\frakS_d]\}&\text{respectively}&\{T_w\mid w\in\frakS_d\},
\end{eqnarray} 
where $T_w$ is the endomorphism given as the composition of endomorphisms $T_r=T_{s_r}=-s_r-(q-1)X_{r+1}\partial_r$ according to a reduced expression. Let $N$ be the left $\bfk(X_1,\ldots,X_d)$-submodule of $A$ generated by $\{w\in\frakS_d\mid w\not=w_d\}$. Note that $N$ is equal to the left submodule of $A$ generated by $\{T_w\mid w\not=w_d\}$. Moreover, the submodule $N$ does not change if we replace "left" by "right".

To prove the first statement write $D_d\q_d$ in the form $D_d\q_d=\sum_{w}T_w a_w$, where $a_w\in\bfk(X_1,\ldots,X_d)$. We need to show $a_w=(-q)^{{l}(w_d)-{l}(w)}$. Since the Demazure operator $D_d$ sends rational functions to symmetric rational functions and $T_r$ act by $-1$ on symmetric rational functions, we have $T_rD_d \q_d= -D_d \q_d$ for each $r$, hence
$$
T_r\left(\sum_w T_wa_w\right)=-\left(\sum_w T_w a_w\right
).
$$
This implies $-qa_{s_rw}=a_{w}$ for each $w$ such that ${l}(w)<{l}(s_rw)$, and it suffices to show $a_{w_d}=1$.
Because $T_r$ can be written as $T_r=-\frac{X_r-qX_{r+1}}{X_r-X_{r+1}}s_r-\frac{(q-1)X_{r+1}}{X_{r}-X_{r+1}}$ we have
\begin{eqnarray*}
T_{w_d}&\equiv&(-1)^{{l}(w_d)}\prod_{1\leqslant a<b\leqslant d}\frac{X_a-qX_b}{X_a-X_b}~w_d\equiv (-1)^{{l}(w_d)}w_d\prod_{1\leqslant a<b\leqslant d}\frac{X_b-qX_a}{X_b-X_a},
\end{eqnarray*}
where $\equiv$ means equality modulo the subspace $N$. Thus  
\begin{eqnarray*}
w_d\equiv (-1)^{{l}(w_d)}T_{w_d}\prod_{1\leqslant a<b\leqslant d}\frac{X_b-X_a}{X_b-qX_a}.
\end{eqnarray*}
Finally, we can write 
\begin{eqnarray*}
D_d&\equiv&\prod_{1\leqslant a<b\leqslant d}\frac{1}{X_a-X_b}~w_d\equiv (-1)^{{l}(w_d)}w_d\prod_{1\leqslant a<b\leqslant d}\frac{1}{X_b-X_a}
\end{eqnarray*}
and therefore
%\begin{eqnarray*}
$D_d=T_{w_d}\prod_{1\leqslant a<b\leqslant d}\frac{1}{X_b-qX_a}+n$ for some $n\in N$.
%\end{eqnarray*}
This implies $a_{w_d}=1$ and hence the first statement follows.

To prove the second statement write $D_d$ in the form $D_d=\sum_w b_w T_w$, where $b_w\in \bfk(X_1,\ldots,X_d)$. It then suffices to show $b_w=\frac{1}{\p_d}$. Since $D_d$ sends rational functions to symmetric rational functions and $T_r$ acts by $-1$ on symmetric rational functions, we have $T_rD_d=-D_d$. This yields
\begin{eqnarray*}
T_r\left(\sum_w b_w T_w\right)&=&-\left(\sum_w b_w T_w\right).
\end{eqnarray*}
Using the relation $T_rb_w=s_r(b_w)T_r-(q-1)X_{r+1}\partial_r(b_w)$
we deduce that for each $w$ with ${l}(s_rw)>{l}(w)$ we have
\begin{eqnarray}
\label{eq-relations_for_b}
-b_{s_rw}&=&s_r(b_w)+(q-1)s_r(b_{s_rw})-(q-1)X_{r+1}\partial_r(b_{s_rw}).
\end{eqnarray}

Clearly, the rational functions $b_w$ are determined by  $b_{w_d}$ and \eqref{eq-relations_for_b}. Thus it suffices to show $b_{w_d}=\frac{1}{\p_d}$ and that $b_w=\frac{1}{\p_d}$ satisfy the relations \eqref{eq-relations_for_b}. We have
\begin{eqnarray*}
-\frac{1}{\p_d}&=&s_r\left(\frac{1}{\p_d}\right)+\left(q-1\right)s_r\left(\frac{1}{\p_d}\right)-\left(q-1\right)X_{r+1}\partial_r\left(\frac{1}{\p_d}\right),
\end{eqnarray*}
and since $\p_d$ is a product of $X_r-qX_{r+1}$ by an element that commutes with $s_r$ and $\partial_r$, it is enough to verify that $-\frac{1}{X_r-qX_{r+1}}$ equals
\begin{eqnarray*}
s_r\left(\frac{1}{X_r-qX_{r+1}}\right)+(q-1)s_r\left(\frac{1}{X_r-qX_{r+1}}\right)-(q-1)X_{r+1}\partial_r\left(\frac{1}{X_r-qX_{r+1}}\right).
\end{eqnarray*}
which is straightforward.  The proof of $b_{w_d}=\frac{1}{\p_d}$ is similar to the arguments in the first part, namely we have 
\begin{eqnarray*}
D_d&\equiv&\prod_{1\leqslant a<b\leqslant d}\frac{1}{X_a-X_b}~w_d\equiv
\prod_{1\leqslant a<b\leqslant d}\frac{1}{X_a-qX_b}~T_{w_d},\\
\end{eqnarray*}
which implies the claim.
\end{proof}
We obtain the following generalization of the easy equality  in $\Heck$
\begin{equation}
\label{eq_exch-(T-q)-pol}
(T_r-q)(X_r-qX_{r+1})=(X_{r+1}-qX_r)(T_r+1).
\end{equation}
\begin{coro}
\label{coro_exch-mu-Pol}
We have the equality $m_d\p_d= \q_dn_d$ in $\Heck$.
\end{coro}
\begin{proof}
This follows from Lemma~\ref{lem_m-on-Pol-rep} and from the symmetricity of $\q_d\p_d$.
\end{proof}

We also need to know how $n'_{a,b}$ acts on the polynomial representation. In light of Lemma~\ref{lem_m-on-Pol-rep} it would be natural to expect that $n'_{a,b}$ acts as $\p'_{a,b} D_{b,a}$. Unfortunately, this is not true is general. However, the following lemma shows that this becomes true in the presence of $n_an_b^{\small{+a}}$ on the left of $n'_{a,b}$. Here we mean that $n_b^{\small{+a}}$ is defined with respect to the shifted indices $a+1,\ldots,a+b$, i.e., with the  composition $\nu=(1,1,\ldots,1,b)$ of $d$ we have
\begin{eqnarray*}
n_b^{\small{+a}}&=&\sum_{w\in {\frakS_\nu}}T_w.
\end{eqnarray*}
%where $w_b^{\small +a}$ in the longest element in ${\frakS_\nu}$.
We will use analogously the notations $m_b^{\small{+a}}$, $\p_b^{\small{+a}}$, $\q_b^{\small{+a}}$ and $D_b^{\small{+a}}$.

\begin{lem}
\label{lem_n'-on-Pol-rep}
The element $n_a n_b^{\small{+a}} n'_{a,b}$ acts on the polynomial representation as $(\p_a D_a)(\p_b^{\small{+a}} D_b^{\small{+a}})(\p'_{a,b} D_{b,a})$.
\end{lem}
\begin{proof}
The statement follows directly from Lemma~\ref{lem_m-on-Pol-rep} $(b)$. Indeed, the product $n_a n_b^{\small{+a}} n'_{a,b}$ is equal to $n_{a+b}$ and the product $(\p_a D_a)(\p_b^{\small{+a}} D_b^{\small{+a}})(\p'_{a,b} D_{b,a})$ is equal to $\p_{a+b}D_{a+b}$.
\end{proof}

\begin{proof}[Proof of Lemma~\ref{lem-polrep-S-inside}]
It is enough to prove these statements in the case where $\mu$ has only two components and $\lambda$ has only one component.  Assume therefore $\mu=(a,b)$ and $\lambda=(a+b)$.
We have $m_\mu=m_am_b^{\small{+a}}$ and $m_\lambda=m_{a+b}$. 
%Here we mean that $m_b^{\small{+a}}$ is defined with respect to the shifted indices $a+1,\ldots,a+b$, i.e., with the  composition $\nu=(1,1,\ldots,1,b)$ of $d$ we have
%\begin{eqnarray*}
%m_b^{\small{+a}}&=&\sum_{w\in {\frakS_\nu}}(-q)^{{l}(w_b^{\small +a})-{l}(w)}T_w,
%\end{eqnarray*}
%where $w_b^{\small +a}$ in the longest element in ${\frakS_\nu}$.
%We will use analogously the notations $n_b^{\small{+a}}$, $\p_b^{\small{+a}}$, $\q_b^{\small{+a}}$ and $D_b^{\small{+a}}$.  

%Some relation in the affine Hecke algebra that we need in this proof will be justified in Section~\ref{subs_some-rel}. 

To prove the first part fix $f\in \bfk[x^{\pm 1}_1,\hdots,x^{\pm 1}_d]^{\frakS_\lambda}$. Then $\Phi_\lambda(f)=m_{a+b} \p_{a+b} f$, and the split sends $\Phi_\lambda(f)$ to the element $m_{a+b} \p_{a+b} f\in m_\mu \Heck$. We have to check that it is in the image of $\Phi_\mu$. Now,  we have
\begin{equation*}
\begin{array}{lllll}
m_{a+b} \p_{a+b} f&=&\q_{a+b}n_{a+b} f
&=&\q_{a}\q_{b}^{\small{+a}}\q'_{a,b}n_an_b^{\small{+a}}n'_{a,b}f\\
&=&\q_a\q_b^{\small{+a}}n_an_b^{\small{+a}}\q'_{a,b}fn'_{a,b}
&=&m_am_b^{\small{+a}}\p_a\p_b^{\small{+a}}\q'_{a,b}fn'_{a,b}\\
&=&\Phi_\mu(\q'_{a,b}f).
\end{array}
\end{equation*}

Here the first and the fourth equalities follow from Corollary~\ref{coro_exch-mu-Pol}. The third equality follows since  $\q'_{a,b}$ is symmetric with respect to the first $a$ and the last $b$ variables, and $f$ is symmetric. Hence the split $(a+b)\to (a,b)$ sends $\Phi_\lambda(f)$ to $\Phi_\mu(\q'_{a,b}f)$. This proves the first statement.

To prove the second part  fix $f\in \bfk[x^{\pm 1}_1,\hdots,x^{\pm 1}_d]^{\frakS_\mu}$. We show that  $\Phi_\mu(f)$  is sent by the split to $\Phi_\lambda(D_{a,b}(f))$, in formulas
\begin{eqnarray}
\label{eq_check-for-merge-before}
m_{a+b} \p_{a}\p^{\small{+a}}_{b} f n'_{a,b}&=&m_{a+b}\p_{a+b}D_{a,b}(f).
\end{eqnarray} 
By Lemmas~\ref{lem_m-on-Pol-rep}, and ~\ref{lem_n'-on-Pol-rep} it suffices to verify, for any $g\in \bfk[x_1^{\pm 1},\ldots,x_d^{\pm 1}]$, that
\begin{eqnarray}
D_{a+b}\left(\q_{a+b}\p_a\p^{\small{+a}}_bf\p'_{a,b}D_{b,a}(g)\right)&=&D_{a+b}\left(\q_{a+b}\p_{a+b}D_{a,b}(f)g\right). \label{eq_check-for-merge}
\end{eqnarray}

(Note that it is not obvious that we are allowed to apply Lemma~\ref{lem_n'-on-Pol-rep} here,  because we have no "$n_an_b^{\small{+a}}$" on the left of "$n'_{a,b}$" in the formula on the left hand side of \eqref{eq_check-for-merge-before}. But we can write $m_{a+b}$ in the form $xm_am_b^{\small{+a}}$ and rewrite the left hand side of  \eqref{eq_check-for-merge-before} using Corollary~\ref{coro_exch-mu-Pol} as follows 
\begin{equation*}
m_{a+b} \p_{a}\p_{b}^{\small{+a}} f n'_{a,b} = xm_am_b^{\small{+a}}  \p_{a}\p^{\small{+a}} _{b} f n'_{a,b} = x\q_a\q_b^{\small{+a}}fn_an_b^{\small{+a}} n'_{a,b},
\end{equation*}
which allows us apply the lemma.)
Since, the polynomials $\p_{a+b}\q_{a+b}$ and $D_{a,b}(f)$ are symmetric, the right hand side of \eqref{eq_check-for-merge} is equal to $\p_{a+b}\q_{a+b}D_{a,b}(f)D_{a+b}(g)$. It agrees with the left hand side of \eqref{eq_check-for-merge} by the calculation
\begin{equation*}
\begin{array}{lllll}
&D_{a+b}\left(\q_{a+b}\p_a\p_b^{\small{+a}}f\p'_{a,b}D_{b,a}(g)\right)
=&D_{a+b}\left(\q_{a+b}\p_{a+b}fD_{b,a}(g)\right)\\
=&\q_{a+b}\p_{a+b}D_{a+b}\left(fD_{b,a}(g)\right)
=&\q_{a+b}\p_{a+b}D_{a,b}D_{a}D^{+a}_{b}\left(fD_{b,a}(g)\right)\\
=&\q_{a+b}\p_{a+b}D_{a,b}\left(fD_aD^{+a}_bD_{b,a}(g)\right)
=&\q_{a+b}\p_{a+b}D_{a,b}\left(fD_{a+b}(g)\right)\\
=&\q_{a+b}\p_{a+b}D_{a,b}(f)D_{a+b}(g).
\end{array}
\end{equation*}

Here the second equality follows since $\q_{a+b}\p_{a+b}$ is symmetric. The fourth equality follows because $f$ is symmetric in the first $a$ and last $b$ variables. The sixth equality follows since $D_{a,b}(f)$ is symmetric. This proves \eqref{eq_check-for-merge}. 
\end{proof}

\subsection{Higher level affine Schur algebra}

Now we define the higher level version $\lS$ of the algebra $\S$ depending on $\bfQ=(Q_1,\hdots,Q_\ell)\in \bfk^\ell$. 

\begin{df}
\label{def-multicomp}
An \emph{$(\ell+1)$-composition} of $d$ is an $(\ell+1)$-tuple $\lambda=(\lambda^{(0)},\hdots,\lambda^{(\ell)})$ such that $\lambda^{(0)},\hdots,\lambda^{(\ell)}$ are compositions (of some non-negative integers) such that $\sum_{i=0}^\ell|\lambda^{(i)}|=d$. Denote by $\calC^\ell_d$ the set of $(\ell+1)$-compositions of $d$. For $\lambda=(\lambda^{(0)},\hdots,\lambda^{(\ell)})\in \calC^\ell_d$ let $\frakS_\lambda=\frakS_{\lambda^{(0)}}\times\hdots\times \frakS_{\lambda^{(\ell)}}\subset \frakS_d$ be the corresponding parabolic subgroup of $\frakS_d$. For each $(\ell+1)$-composition $\lambda$ of $d$ we denote by $[\lambda_r^{(k)}]$ the subset of $\{1,2,\ldots,d\}$ that contains the elements 
$$
\mbox{from}\quad 1+\sum_{i=0}^{k-1}|\lambda^{(i)}|+\sum_{j=1}^{r-1}\lambda_j^{(k)} \quad \mbox{to} \quad \sum_{i=0}^{k-1}|\lambda^{(i)}|+\sum_{j=1}^{r}\lambda_j^{(k)}.
$$
Note that the set $[\lambda_r^{(k)}]$ depends on $\lambda$, $r$ and $k$ (not only on the number $\lambda_r^{(k)}$).
\end{df}

To $\lambda\in \calC^\ell_d$ we attach the following element $m_\lambda\in\lHeck$,
\begin{eqnarray*}
\tikz[thick,xscale=2.5,yscale=1.5]{
\node at (-.9,.25) {$m_\lambda\;\;=$};
\node at (-0.25,0.25) {$m_{\lambda^{(0)}}$};
\draw[wei] (0,0) --(0,.5) node[below,at start]{$Q_1$};
\node at (0.25,0.25) {$m_{\lambda^{(1)}}$};
\draw[wei] (.5,0) --(.5,.5) node[below,at start]{$Q_2$};
\node at (0.85,.25) {$\cdots$};x
\draw[wei] (1.2,0) --(1.2,.5) node[below,at start]{$Q_{\ell-1}$};
\node at (1.45,0.25) {$m_{\lambda^{(\ell-1)}}$};
\draw[wei] (1.7,0) --(1.7,.5) node[below,at start]{$Q_\ell$};
\node at (1.95,0.25) {$m_{\lambda^{(\ell)}}$};
}
\end{eqnarray*}

We consider $m_\lambda \lHeck$ as a right $\lHeck$-module.
\begin{df}
The {\it affine Schur algebra (of level $\ell$)} is the algebra
\begin{eqnarray}
\label{defhigherlevSchur}
\lS&=&\End_{\lHeck}\left(\bigoplus_{\lambda\in \calC^\ell_d}m_\lambda \lHeck\right).
\end{eqnarray}
\end{df}

We could define $n_\lambda$ similarly to $m_\lambda$ and consider the following modification of the affine Schur algebra defined in terms of $n_\lambda$ instead of $m_\lambda$:
\begin{eqnarray*}
\lSbar &= &\End_{\lHeck}\left(\bigoplus_{\lambda\in \calC^\ell_d}n_\lambda \lHeck\right).
\end{eqnarray*}

Using the isomorphism $\#\colon \lHeck\to \lHeckop$ in Lemma~\ref{lem-hash_isom} we have $(n_\lambda)^\#=m_\lambda$ (up to a sign). This implies directly the following.
\begin{lem}
\label{lem-isom-Sbar-Sop}
There is an isomorphism of algebras $\lSbar\to \lSop$.
\end{lem}

We introduce now the thick calculus for the algebra $\lS$ extending the diagrammatic calculus for $\lHeck$ and $\S$. We draw special elements of this algebra as diagrams that are similar to the special diagrams for $\lHeck$. The difference is that the black strands are also allowed to have "multiplicities" (that are positive integers). We also allow the diagrams to contain locally elements of the form \eqref{diag-split-merge}. Instead of dots, a segment of a strand of multiplicity $b$ is allowed to carry a symmetric Laurent polynomial of $b$ variables.
\subsection{Generators of $\lS$}
\label{subs-gen_lS}

For each $\lambda\in\calC^\ell_d$ there is an idempotent $e(\lambda)\in \lS$ given by the identity endomorphism of the  summand  $m_\lambda\lHeck$. We draw it as 

$$
\tikz[thick,xscale=2.5,yscale=1.5]{
\node at (-.9,.25) {$e(\lambda)=$};
\draw (-0.6,0) --(-0.6,.5) node[below,at start]{$\lambda^{(0)}_1$};
\draw (-0.4,0) --(-0.4,.5) node[below,at start]{$\lambda^{(0)}_2$};
\node at (-0.2,.25) {$\cdots$};
\draw[wei] (0,0) --(0,.5) node[below,at start]{$Q_1$};
\draw (0.2,0) --(0.2,.5) node[below,at start]{$\lambda^{(1)}_1$};
\draw (0.4,0) --(0.4,.5) node[below,at start]{$\lambda^{(1)}_2$};
\node at (0.6,.25) {$\cdots$};
\draw[wei] (0.8,0) --(0.8,.5) node[below,at start]{$Q_2$};
\draw (1.0,0) --(1.0,.5) node[below,at start]{$\lambda^{(2)}_1$};
\draw (1.2,0) --(1.2,.5) node[below,at start]{$\lambda^{(2)}_2$};
\node at (1.4,.25) {$\cdots$};
\node at (1.6,.25) {$\cdots$};
\node at (1.8,.25) {$\cdots$};
\draw[wei] (2.0,0) --(2.0,.5) node[below,at start]{$Q_\ell$};
\draw (2.2,0) --(2.2,.5) node[below,at start]{$\lambda^{(\ell)}_1$};
\draw (2.4,0) --(2.4,.5) node[below,at start]{$\lambda^{(\ell)}_2$};
\node at (2.6,.25) {$\cdots$};
}
$$

Let $\mu$ be another $(\ell+1)$-composition of $d$. We say that $\mu$ is a \emph{split} of $\lambda$ (and $\lambda$ is a {\it merge} of $\mu$) if there is a $t$ such that the component $\mu^{(t)}$ of $\mu$ is a split of the component $\lambda^{(t)}$ of $\lambda$ (in the sense of Section~\ref{subs-gen_S}) and $\mu^{(i)}=\lambda^{(i)}$ if $i\ne t$.  In this case we can define the split map $\lambda\to \mu$ and the merge map $\mu\to\lambda$ in $\lS$ in the same way as in Section~\ref{subs-gen_S}. We draw the split and merge map for $\la= (a+b)$ and $\mu=(a,b)$ as in \eqref{diag-split-merge} and for arbitrary $\lambda$, $\mu$ by adding the appropriate vertical strands to the left and right.

\begin{df}
Assume $\lambda,\mu \in\calC^\ell_d$ such that  $\mu$ is obtained from $\lambda$ by moving the first component of $\lambda^{(t)}$ to the end of $\lambda^{(t-1)}$ for some $t\in[1;\ell]$. More precisely, we assume $\lambda^{(i)}=\mu^{(i)}$ for $i\ne t-1, t$ and $\mu^{(t-1)}=(\lambda^{(t-1)}_1,\lambda^{(t-1)}_2,\hdots,\lambda^{(t-1)}_{{l}(\lambda^{(t-1)})},\lambda^{(t)}_1)$ and $\mu^{(t)}=(\lambda^{(t)}_2,\lambda^{(t)}_3,\hdots,\lambda^{(t)}_{{l}(\lambda^{(t)})})$. In this case we say that $\mu$ is a \emph{left crossing} of $\lambda$ and that $\lambda$ is a \emph{right crossing} of $\mu$.
\end{df} 

To a left crossing $\mu$ of $\lambda$ we  assign the two special elements in $\lS$ given by left multiplication with 
\begin{eqnarray}
\label{eq_def-lsh}
\begin{tikzpicture}[thick,baseline=9pt]
\node at (-0.5,0.5) {$$};
\draw[wei] (0,0) node[below]{$Q_t$}  -- (1.5,1);
\draw (1.5,0) -- (0.9,1); 
\draw (1.2,0) -- (0.6,1); 
\node at (0.5,0.6) {$\hdots$};
\draw (0.6,0)  -- (0,1);
\node at (1.8,0.5) {$$};
\end{tikzpicture}
&\text{respectively}&
\begin{tikzpicture}[thick,baseline=9pt]
\node at (-0.5,0.5) {$$};
\draw (0,0) -- (0.6,1); 
\draw (0.3,0) -- (0.9,1); 
\node at (1,0.6) {$\hdots$};
\draw (0.9,0)  -- (1.5,1);
\draw[wei] (1.5,0) node[below]{$Q_t$}  -- (0,1);
\node at (1.8,0.5) {$$};
\end{tikzpicture}
\end{eqnarray}
where in either case we have $\lambda^{(t)}_1$ parallel black strands crossing the involved red strand and all other strands (which we did not draw) are just vertical. Such  a multiplication yields an element of $\Hom_{\lHeck}(m_\lambda\lHeck,m_\mu\lHeck)$ respectively of $\Hom_{\lHeck}(m_\mu\lHeck,m_\lambda\lHeck)$
because of the relations \eqref{l-Hecke-diag-3}. Thus by extending by zero to the other summands we obtain indeed an element of $\lS$. We call these elements of $\lS$ \emph{left crossings} respectively \emph{right crossings}, denote them $\lambda\to\mu$ respectively  $\lambda\to\mu$ and usually draw them just as 
\begin{eqnarray}
\label{eq_def-lsh2}
\tikz[thick,xscale=2.5,yscale=1.5, baseline=0.8cm]{
\draw[wei] (0,0) node[below] {$Q_t$} to (.6,1) ;
\draw (.6,0) node[below] {$\lambda^{(t)}_1$} to (0,1) ;
}
&\quad\text{respectively}\quad&
\tikz[thick,xscale=2.5,yscale=1.5, baseline=0.8cm]{
\draw (0,0) node[below] {$\lambda^{(t)}_1$} to (.6,1) ;
\draw[wei] (.6,0) node[below] {$Q_t$} to (0,1) ;
}
\end{eqnarray}
(with possibly vertical strands to the left and right).
Similarly to Section~\ref{subs-gen_S}, for each $\lambda\in\calC^\ell_d$ and $f\in \bfk[x_1^{\pm 1},\ldots, x_d^{\pm 1}]$ we have an element $fe(\lambda)\in \lS$.

\begin{rk}
Similarly, to Section~\ref{subs-gen_S}, we could introduce a black crossing in $\lS$. But this element can be expressed in terms of other generators of $\lS$.
\end{rk}

\begin{rk} One could again realize $\lS$ as a quotient of an algebra structure on the direct sum of homomorphism spaces in the universal thickened higher level category, where we take $I_r=\bfk^*$, $I_b=(\mathbb{Z}_{>0},+)$ and again for $R_a$ the ring of symmetric Laurent polynomials in $a$ variables. The objects to consider are all tensor products of black and red labels, such that the sum of the black labels is $d$ and the sequence of red labels is $\bfQ$.  Since we do not know the defining relations for $\lS$ we do not follow this viewpoint here. In particular, the faithful representation constructed below becomes crucial. Similar remarks also apply to the (higher level) quiver Schur algebras in Section~\ref{sec-QSchur}.
\end{rk}

%%%%%%%%%%%%%%%%%%%%%%%%%%%%%%

\subsection{Polynomial representation of $\lS$}
\label{subs:Pol_rep_lS}
By definition, \eqref{defhigherlevSchur}, the algebra $\lS$ has a faithful representation on the vector space  $\bigoplus_{\lambda\in \calC^\ell_d}m_\lambda\lHeck$. In this section we are going to construct a polynomial representation
$$
\lSPol=\bigoplus_{\lambda\in \calC^\ell_d}\bfk[x^{\pm 1}_1,\hdots,x^{\pm 1}_d]^{\frakS_\lambda}e(\lambda)
$$
of the algebra $\lS$ sitting inside the defining representation.
\begin{df}
For each $\lambda\in \calC^\ell_d$ we denote  

\begin{itemize}
\item by $\overline\lambda$ the elements of $\calC_d$ obtained by concatenation of the $\ell+1$ components of $\lambda$, i.e., we have $\overline\lambda=\lambda^{(0)}\cup\ldots \cup \lambda^{(\ell)}$, where $\cup$ denotes the concatenation of compositions; and 
\item  by $e^0(\lambda)$ the idempotent in $\lHeck$ obtained from $e(\lambda)$ by replacing each vertical black strand of multiplicity $a$ (for each positive integer $a$) by $a$ usual (multiplicity $1$) vertical black strands; and
\item by $\mathfrak{r}_\lambda$ the element of $\lHeck$ represented by the diagram defined by the following three properties. The top part of the diagram corresponds to the idempotent $e^0(\lambda)$. At the bottom of the diagram, each red strand is on the left of each black strand. The diagram may contain left crossings, but neither dots, splits, merges nor right crossings. 
\end{itemize}
\end{df}
\begin{ex}
Take $\ell=2$, $\lambda=((1),(2,1),(1,2))$. In this case we have

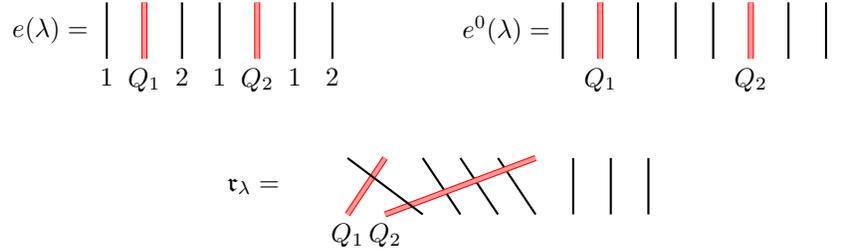
\begin{eqnarray*}
\tikz[thick,xscale=2.5,yscale=1.5]{
\node at (-1.3,.25) {$e(\lambda)=$};
\draw (-1.0,0) --(-1.0,.5) node[below,at start]{$1$};
\draw[wei] (-0.8,0) --(-0.8,.5) node[below,at start]{$Q_1$};
\draw (-0.6,0) --(-0.6,.5) node[below,at start]{$2$};
\draw (-0.4,0) --(-0.4,.5) node[below,at start]{$1$};
\draw[wei] (-0.2,0) --(-0.2,.5) node[below,at start]{$Q_2$};
\draw (0,0) --(0,.5) node[below,at start]{$1$};
\draw (0.2,0) --(0.2,.5) node[below,at start]{$2$};
}
&\quad\text\quad&
\tikz[thick,xscale=2.5,yscale=1.5]{
\node at (-1.3,.25) {$e^0(\lambda)=$};
\draw (-1.0,0) --(-1.0,.5);
\draw[wei] (-0.8,0) --(-0.8,.5) node[below,at start]{$Q_1$};
\draw (-0.6,0) --(-0.6,.5);
\draw (-0.4,0) --(-0.4,.5);
\draw (-0.2,0) --(-0.2,.5);
\draw[wei] (0,0) --(0,.5) node[below,at start]{$Q_2$};
\draw (0.2,0) --(0.2,.5);
\draw (0.4,0) --(0.4,.5);
\draw (0.6,0) --(0.6,.5);
}
\end{eqnarray*}
\end{ex}

$$
\tikz[thick,xscale=2.5,yscale=1.5]{
\node at (-1.5,.25) {$\mathfrak{r}_\lambda=$};
\draw (-0.4,0) --(-0.6,.5);
\draw (-0.2,0) --(-0.4,.5);
\draw (0,0) --(-0.2,.5);
\draw[wei] (-0.8,0) --(0,.5) node[below,at start]{$Q_2$};
\draw (0.2,0) --(0.2,.5);
\draw (0.4,0) --(0.4,.5);
\draw (0.6,0) --(0.6,.5);
\draw[wei] (-1.0,0) --(-0.8,.5) node[below,at start]{$Q_1$};
\draw (-0.6,0) --(-1.0,.5);
}
$$

Denote by $\iota$ the obvious inclusion of $\Heck$ to $\lHeck$ obtained by adding $\ell$ red strands on the left. This defines an inclusion 
\begin{eqnarray}
\label{inclhigherlevel}
\Phi_\lambda\colon \bfk[x_1^{\pm 1},\ldots,x_d^{\pm 1}]^{\frakS_\lambda}\to m_\lambda \lHeck, \qquad f\mapsto \mathfrak{r}_{\lambda}\iota(\Phi_{\overline\lambda}(f)).
\end{eqnarray}

\begin{ex} Let $\lambda=((2,1),(1,2))$. Then the element $\Phi_\lambda(f)$ is displayed on the right hand side in Figure~\ref{diag-ex_Phi}. It equals the left hand side, since relations \eqref{l-Hecke-diag-2}-\eqref{l-Hecke-diag-3} allow dots and black-black crossings to slide through red strands. This argument shows in general that the element $\Phi_\lambda(f)$ is indeed in $m_\lambda \lHeck$. (Although this is obvious for the left hand side of the equality in Figure~\ref{diag-ex_Phi}, this was not completely obvious for the original definition of $\Phi_\lambda(f)$.)
\begin{figure}
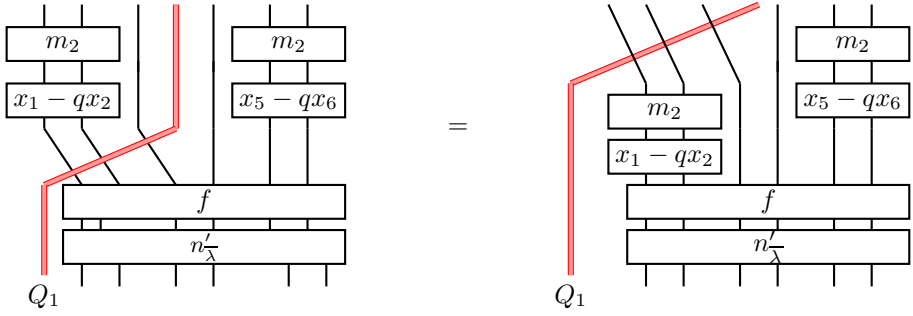

\begin{equation*}
\tikz[thick,xscale=2.5,yscale=1.5]{
%\node at (-.9,.25) {$R_\lambda=$};
%\draw (-0.6,1) -- (-0.4,1.3);
%\draw (-0.4,1) -- (-0.6,1.3);
\draw (-0.6,1.4) --(-0.6,1.6);
\draw (-0.4,1.4) --(-0.4,1.6);
\draw (-0.1,1) --(-0.1,1.6);
\node at (-0.5,1.25) {$m_2$};
\draw (-0.8,1.1) rectangle (-.2,1.4);
\draw (0.3,1) --(0.3,1.6);
\draw (0.6,1.4) --(0.6,1.6);
\draw (0.8,1.4) --(0.8,1.6);
\node at (0.7,1.25) {$m_2$};
\draw (0.4,1.1) rectangle (1,1.4);
\draw (-0.6,0.5) --(-0.6,.6);
\draw (-0.6,.9) --(-0.6,1.1);
\draw (-0.4,0.5) --(-0.4,.6);
\draw (-0.4,.9) --(-0.4,1.1);
\draw (-0.1,0.5) --(-0.1,1.1);
\node at (-0.5,0.75) {$x_1-qx_2$};
\draw (-0.8,.6) rectangle (-.2,0.9);
\draw[wei] (0.1,0.5) --(0.1,1.6);
\draw (0.3,0.5) --(0.3,1.1);
\draw (0.6,0.5) --(0.6,.6);
\draw (0.6,.9) --(0.6,1.1);
\draw (0.8,0.5) --(0.8,.6);
\draw (0.8,.9) --(0.8,1.1);
\node at (0.7,0.75) {$x_5-qx_6$};
\draw (0.4,.6) rectangle (1,0.9);
\draw (-0.4,0) --(-0.6,.5);
\draw (-0.2,0) --(-0.4,.5);
\draw (0.1,0) --(-0.1,.5);
\draw[wei] (-0.6,0) --(0.1,.5) ;
\draw (0.3,0) --(0.3,.5);
\draw (0.6,0) --(0.6,.5);
\draw (0.8,0) --(0.8,.5);
\node at (0.25,-0.15) {$f$};
\draw (-0.5,0) rectangle (1,-0.3);
\draw[wei] (-0.6,-.4) -- (-0.6,0);
\draw (-0.4,-0.4) --(-0.4,-0.3);
\draw (-0.3,-0.4) --(-0.3,-0.3);
\draw (0.1,-0.4) --(0.1,-0.3);
\draw (0.3,-0.4) --(0.3,-0.3);
\draw (0.6,-0.4) --(0.6,-0.3);
\draw (0.8,-0.4) --(0.8,-0.3);
\node at (0.25,-0.55) {\small{$n'_{\overline\lambda}$}};
\draw (-0.5,-0.4) rectangle (1,-0.7);
\draw[wei] (-0.6,-0.8) -- (-0.6,-0.4) node[below,at start]{$Q_1$};
\draw (-0.4,-0.9) --(-0.4,-0.7);
\draw (-0.2,-0.9) --(-0.2,-0.7);
\draw (0.1,-0.9) --(0.1,-0.7);
\draw (0.3,-0.9) --(0.3,-0.7);
\draw (0.7,-0.9) --(0.7,-0.7);
\draw (0.9,-0.9) --(0.9,-0.7);
\node at (1.6,0.5) {$=$};
\draw[wei] (2.2,0.9) --(3.2,1.6);
\draw (2.6,0.9) -- (2.4,1.6);
\draw (2.8,0.9) -- (2.6,1.6);
\draw (3.1,0.9) -- (2.9,1.6);
\draw (3.3,1) --(3.3,1.6);
\draw (3.6,1.4) --(3.6,1.6);
\draw (3.8,1.4) --(3.8,1.6);
\node at (3.7,1.25) {$m_2$};
\draw (3.4,1.1) rectangle (4,1.4);
\draw (3.3,0.5) --(3.3,1.1);
\draw (3.6,0.5) --(3.6,.6);
\draw (3.6,.9) --(3.6,1.1);
\draw (3.8,0.5) --(3.8,.6);
\draw (3.8,.9) --(3.8,1.1);
\node at (3.7,0.75) {$x_5-qx_6$};
\draw (3.4,.6) rectangle (4,0.9);
%
%
%\draw (2.6,0) --(2.4,.5);
%\draw (2.8,0) --(2.6,.5);
%\draw (3.1,0) --(2.9,.5);
%\draw[wei] (2.4,0) --(3.1,.5) ;
%
%
%
\draw (2.6,0.8) --(2.6,0.9);
\draw (2.8,0.8) --(2.8,0.9);
\draw (3.1,0.5) --(3.1,0.9);
\node at (2.7,0.65) {$m_2$};
\draw (2.4,.5) rectangle (3,0.8);
\draw (2.6,0) --(2.6,0.1);
\draw (2.6,0.4) --(2.6,0.5);
\draw (2.8,0) --(2.8,0.1);
\draw (2.8,0.4) --(2.8,0.5);
\draw (3.1,0) --(3.1,0.5);
\node at (2.7,0.25) {$x_1-qx_2$};
\draw (2.4,.1) rectangle (3,0.4);
\draw (3.3,0) --(3.3,.5);
\draw (3.6,0) --(3.6,.5);
\draw (3.8,0) --(3.8,.5);
\node at (3.25,-0.15) {$f$};
\draw (2.5,0) rectangle (4,-0.3);
%
%\draw[wei] (2.2,-.4) -- (2.2,0);
\draw (2.6,-0.4) --(2.6,-0.3);
\draw (2.8,-0.4) --(2.8,-0.3);
\draw (3.1,-0.4) --(3.1,-0.3);
\draw (3.3,-0.4) --(3.3,-0.3);
\draw (3.6,-0.4) --(3.6,-0.3);
\draw (3.8,-0.4) --(3.8,-0.3);
\node at (3.25,-0.55) {$n'_{\overline\lambda}$};
\draw (2.5,-0.4) rectangle (4,-0.7);
\draw[wei] (2.2,-0.8) -- (2.2,0.9) node[below,at start]{$Q_1$};
\draw (2.6,-0.9) --(2.6,-0.7);
\draw (2.8,-0.9) --(2.8,-0.7);
\draw (3.1,-0.9) --(3.1,-0.7);
\draw (3.3,-0.9) --(3.3,-0.7);
\draw (3.6,-0.9) --(3.6,-0.7);
\draw (3.8,-0.9) --(3.8,-0.7);
}
\end{equation*}
\caption{Well-definedness of the polynomial representation.}
\label{diag-ex_Phi}
\end{figure}
\end{ex}

\begin{lem}
\label{lem-polrep-lS-inside-split-merge}
 Let $\lambda, \mu\in \calC^\ell_d$ and assume that $\mu$ is a split of $\lambda$. 
\begin{enumerate}
\item The split map in $\Hom_{\lHeck}(m_\lambda\lHeck,m_\mu\lHeck)$ applied to the image of $\Phi_\lambda$ is contained in the image of $\Phi_\mu$. 
\item The merge in $\Hom_{\lHeck}(m_\mu\lHeck,m_\lambda\lHeck)$ applied to the image of $\Phi_\mu$ is contained in the image of $\Phi_\lambda$. 
\end{enumerate}
\end{lem}

\begin{proof}
The proof is totally analogous to the proof of Lemma~\ref{lem-polrep-S-inside}.
\end{proof}

\begin{lem}
\label{lem-polrep-lS-inside-crossings} Let  $\lambda\in \calC^\ell_d$.
Assume that $\mu$ is a left crossing of $\lambda$.
\begin{enumerate}
\item The left crossing in $\Hom_{\lHeck}(m_\lambda\lHeck,m_\mu\lHeck)$ applied to the image of $\Phi_\lambda$ is contained in the image of $\Phi_\mu$. 
\item
  The right crossing in $\Hom_{\lHeck}(m_\mu\lHeck,m_\lambda\lHeck)$ applied to the image of $\Phi_\mu$ is contained in the image of $\Phi_\lambda$. 
  \end{enumerate}
\end{lem}

\begin{proof}
Fix $f\in \bfk[x_1^{\pm 1},\ldots,x_d^{\pm 1}]^{\frakS_{\lambda}}$. It is clear from the definitions that the left crossing map $\lambda\to\mu$ acts by sending $\Phi_\lambda(fe(\lambda))$ to $\Phi_\mu(fe(\mu))$. 
%Let $D$ be the diagram \eqref{eq_def-lsh} representing the right crossing $\mu\to\lambda$. 
Let $t$ be the index such that $\lambda^{(t)}=(\mu^{(t-1)}_{l(\mu^{(t-1)})})\cup\mu^{(t)}$. Set $a=\lambda^{(t)}_1$ and $b=\sum_{i=1}^{t-1}|\lambda^{(i)}|$. Relation \eqref{l-Hecke-diag-1} implies that $\mu\to\lambda$ sends $\Phi_\mu(fe(\mu))$ to $\Phi_\lambda(gfe(\lambda))$, where $g=\prod_{i=b+1}^{b+a}(x_i-Q_t)$. 
\end{proof}

Lemmas~\ref{lem-polrep-lS-inside-split-merge} and~\ref{lem-polrep-lS-inside-crossings} will be used to deduce the following result which as a special case establishes also a proof of Proposition~\ref{prop-polrep_S}.
\begin{prop}
\label{prop-polrep-lS}
There is a unique action of the algebra $\lS$ on  $\lSPol$ satisfying the following properties using the abbreviation $P=\bfk[x^{\pm 1}_1,\hdots,x^{\pm 1}_d]$. 
\begin{itemize}
\item The idempotent $e(\lambda)$,  $\lambda\in \calC^\ell_d$,  acts on $\lSPol$ as projection to $Pe(\lambda)$. 
\item For each $g\in P^{\frakS_\lambda}$, the element $ge(\lambda)$ sends $fe(\lambda)$ to $gfe(\lambda)$.
\item Splits and merges act in the same way as in Proposition~\ref{prop-polrep_S}.
\item Left crossing maps $\lambda\to \mu$ act by sending $fe(\lambda)$, $f\in P^{\frakS_\lambda}$ to $fe(\mu)$.
\item Right crossing maps $\mu\to\lambda$ act by sending $fe(\mu)$, $f\in P^{\frakS_\mu}$ to $gfe(\lambda)$ where $g=\prod_{i\in [\lambda^{(t)}_1]}^{b+a}(Q_t-x_i)$. 
%in case $\lambda=(a+b)$ and $\mu=(a,b)$ with $a+b=d$.  In the general case they acts by the same formula but in the variables attached to the strands involved in the crossing. 
\end{itemize}
Moreover, the obtained representation of $\lS$ in $\lSPol$ is faithful.

\end{prop}
\begin{proof}
The proof of the existence and uniqueness will be spread over the whole next section and then finally follow from Lemma~\ref{lem-generate_lS} and the above lemmas.
The proof of faithfulness is similar to Proposition~\ref{prop-polrep_S}: Assume that there exist $\lambda,\mu\in \calC^\ell_d$ and a nonzero element $\phi\in \Hom_{\lHeck}(m_\lambda\lHeck,m_\mu\lHeck)$ such that $\phi$ acts on $\lSPol$ by zero. Consider the split $\lambda'\to \lambda$ such that for each $r\in\{0,1,\ldots,\ell\}$, we have $\lambda'^{(r)}=(|\lambda'^{(r)}|)$  (i.e., $\lambda'$ is the coarsest possible). Then, after composing $\phi$ with this split, we get a nonzero element of $\psi\in\Hom_{\lHeck}(m_{\lambda'}\lHeck,m_{\mu}\lHeck)$ that acts by zero on $\lSPol$. The fact that $\psi$ acts by zero on $\lSPol$ implies 
$
\psi(m_{\lambda'}\p_{\lambda'}\mathfrak{r}_{\lambda'})=\psi(\Phi_{\lambda'}(1))=0,
$
where
$$
\tikz[thick,xscale=2.5,yscale=1.5]{
\node at (-.9,.25) {$\p_{\lambda'}\;\;=$};
\node at (-0.25,0.25) {$\p_{\lambda'^{(0)}}$};
\draw[wei] (0,0) --(0,.5) node[below,at start]{$Q_1$};
\node at (0.25,0.25) {$\p_{\lambda'^{(1)}}$};
\draw[wei] (.5,0) --(.5,.5) node[below,at start]{$Q_2$};
\node at (0.85,.25) {$\cdots$};
\draw[wei] (1.2,0) --(1.2,.5) node[below,at start]{$Q_{\ell-1}$};
\node at (1.45,0.25) {$\p_{\lambda'^{(\ell-1)}}$};
\draw[wei] (1.7,0) --(1.7,.5) node[below,at start]{$Q_\ell$};
\node at (1.95,0.25) {$\p_{\lambda'^{(\ell)}}$};
}
$$
This implies $\psi(m_{\lambda'})\p_{\lambda'}\mathfrak{r}_{\lambda'}=0$. Moreover, it is clear from \eqref{l-Hecke-diag-1} that the element $\mathfrak{r}_{\lambda'}$ can by multiplied by an element of $\lHeck$ on the right such that the product is of the form $Qe^0(\lambda')$, where $Q\in\bfk[x^{\pm 1}_1,\ldots,x^{\pm 1}_d]$, $Q\ne 0$. We get 
$$
\psi(m_{\lambda'})\p_{\lambda'}Q=\psi(m_{\lambda'})\p_{\lambda'}Qe^0(\lambda')=0.
$$  
Thus, $\psi(m_{\lambda'})=0$, because $\lHeck$ is free as a right $\bfk[x^{\pm 1}_1,\ldots,x^{\pm 1}_d]$-module.
\end{proof}

\subsection{A basis of $\lS$}
The goal of this section is to obtain a basis of $\lS$. For this we first describe the space $\Hom(\lambda,\mu)=\Hom_{\lHeck}(m_\lambda\lHeck,m_\mu\lHeck)$,  for $\lambda,\mu\in\calC_d$ in terms of the finite Hecke algebra $\Hfin_d(q)$, see Remark~\ref{ordinaryHecke}. 

For $\lambda\in\calC_d$, denote by $\Hfin_\lambda(q)\subset\Hfin_d(q)$  the  Hecke algebra corresponding to $\frakS_\lambda$ (see \eqref{Young}) and by $\epsilon_\lambda$ the sign representation of $\Hfin_\lambda(q)$. The following is well-known:
\begin{lem}
\label{lem-ind_sing}
We have an isomorphism of right $\Heck$-modules
\begin{eqnarray}
\label{easy}
m_\lambda\Heck&\simeq& \epsilon_\lambda\otimes_{\Hfin_\lambda(q)}\Heck
\end{eqnarray}
\end{lem}
\begin{proof} Let $v$ be a generator of the (one-dimensional) vector space $\epsilon_\lambda$. We have an obvious morphism of $\Heck$-modules 
$$
\epsilon_\lambda\otimes_{\Hfin_\lambda(q)}\Heck\to m_\lambda\Heck,\quad v\otimes x\mapsto m_\lambda\cdot x.
$$
It is well-defined because we have $m_\lambda T_r=-m_\lambda$ for each $r$ such that $s_r\in\frakS_\lambda$. The bijectivity of this morphism follows from the fact that $\Heck$ is a free left $\Hfin_\lambda(q)$-module. 
%see Lemma \ref{lem:basis-Heck}.
\end{proof}

Now, we would like to extend \eqref{easy} to the higher level affine Hecke algebra $\lHeck$. 
Given  $\lambda\in\calC^\ell_d$, denote again by $\Hfin_\lambda(q)\subset\Hfin_d(q)$  the  Hecke algebra corresponding to $\frakS_\lambda$ (the group $\frakS_\lambda$ is as in Definition~\ref{def-multicomp}). We can identify $\Hfin_\lambda(q)$ with the unitary subalgebra in $e^0(\lambda)\lHeck e^0(\lambda)$ generated by the elements $T_re^0(\lambda)$ where the indices $r$ correspond to simple reflection in $\frakS_\lambda$.  

\begin{lem}
We have an isomorphism of right $\lHeck$-modules 
\begin{eqnarray}
\label{noteasy}
m_\lambda\lHeck&\simeq& \epsilon_\lambda\otimes_{\Hfin_\lambda(q)} e^0(\lambda)\lHeck
\end{eqnarray}
\end{lem}
\begin{proof}
Let $\opH_\lambda(q)$ be the (non-unitary) subalgebra of $\lHeck$ generated by $\Hfin_\lambda(q)$ and $\bfk[x_1^{\pm 1},\ldots,x_d^{\pm 1}]$. (This algebra is clearly isomorphic to a tensor product of the algebras $H_{\lambda_{i}^{(j)}}(q)$.) We have
$$
\begin{array}{rcl}
\epsilon_\lambda\otimes_{\Hfin_\lambda(q)} e^0(\lambda)\lHeck & \simeq &  \epsilon_\lambda\otimes_{\Hfin_\lambda(q)}\opH_\lambda(q)\otimes_{\opH_\lambda(q)} e^0(\lambda)\lHeck\\
& \simeq & m_\lambda \opH_\lambda(q)\otimes_{\opH_\lambda(q)} e^0(\lambda)\lHeck\\
& \simeq & m_\lambda\lHeck.
\end{array}
$$
The first isomorphism is obvious, the second follows from Lemma~\ref{lem-ind_sing} and the third is true because the right $\opH_\lambda(q)$-module $e^0(\lambda)\lHeck$ is free by Proposition~\ref{prop-basis-Hdl}. 
\end{proof}

We thus have that $\Hom(\lambda,\mu)$ is isomorphic to 
\begin{eqnarray}
\label{eq-Hom_via_sing}
\begin{array}{rcl}
%&&\Hom_{\lHeck}(m_\lambda\lHeck,m_\mu\lHeck)\\
&&\Hom_{\lHeck}\left(\epsilon_\lambda\otimes_{\Hfin_\lambda(q)}e^0(\lambda)\lHeck,\epsilon_\mu\otimes_{\Hfin_\mu(q)}e^0(\mu)\lHeck\right)\\
&\simeq& \Hom_{\Hfin_\lambda(q)}\left(\epsilon_\lambda,\epsilon_\mu\otimes_{\Hfin_\mu(q)}e^0(\mu)\lHeck e^0(\lambda)\right). 
\end{array}
\end{eqnarray}

Above, we used the adjunction 
$$
\Hom_{A_1}(M\otimes_{A_2}B,N)=\Hom_{A_2}(M,\Hom_{A_1}(B,N)),
$$
where $A_1$ and $A_2$ are rings, $N$ is a right $A_1$-module,  $M$ is a right $A_2$-module and $B$ is an $(A_2,A_1)$-bimodule. This adjunction is applied to $A_1=\lHeck$, $A_2=\Hfin_\lambda(q)$, $N=\epsilon_\mu\otimes_{\Hfin_\mu(q)}e^0(\mu)\lHeck$, $M=\epsilon_\lambda$, $B=e^0(\lambda)\lHeck$. 

Now, we see that to get a basis of $\lS$, we should understand the structure of the $(\Hfin_\mu(q),\Hfin_\lambda(q))$-bimodule $e^0(\mu)\lHeck e^0(\lambda)$ for $\lambda,\mu\in\calC^\ell_d$.

\begin{df}
Let $\lambda,\mu, \nu\in\calC_d$. Denote by $\lambda\cap\mu$ the composition in $\calC_d$ such that $\frak{S}_{\lambda\cap\mu}=\frak{S}_\lambda\cap \frak{S}_\mu$.  Recall from Section~\ref{subs-affSchur} that we denote by $\Drep_{\lambda,\mu}$ the set of minimal length representatives of the double cosets $\frak{S}_\lambda\backslash \frak{S}_d/\frak{S}_\mu$. If $\frak{S}_\lambda, \frak{S}_\mu$ are subgroups of $\frak{S}_\nu$ we denote $\Drep^\nu_{\lambda,\mu}=\frak{S}_\nu\cap \Drep_{\lambda,\mu}$. 
\end{df}

Let $\Xint$ be the set of Laurent monomials $x_1^{a_1}x_2^{a_2}\ldots x_d^{a_d}$ with $a_r\in\bbZ$. Denote by $\Xdomla$ the subset of $\Xint$ that contains only monomials such that $(a_1,a_2,\ldots,a_d)$ is non-decreasing inside of each component of $\lambda$, i.e., we have
\begin{eqnarray}
\label{Xdomla}
\Xdomla=\left\{ x^{a_1}\cdots x^{a_d}\in \Xint\mid a_r\leqslant a_{r+1}, \mbox{ unless } r=\lambda_1+\ldots+\lambda_t \mbox{ for some }t \right\}.
\end{eqnarray}
%such that we have $a_r\leqslant a_{r+1}$ for each $r$ that is not of the form $r=\lambda_1+\ldots+\lambda_t$ for some $t$. 
For $p=x_1^{a_1}\cdots x_d^{a_d}\in\Xdomla$, denote by $\lambda\cap p$ the unique composition that is finer than $\lambda$ and such that its components correspond precisely to the segments where $(a_1,a_2,\ldots,a_d)$ is constant. In other words, the indices $r,r+1\in \{1,2\ldots,d\}$ are in the same component of the composition $\lambda\cap p$ if and only if they are in the same component of the composition $\lambda$ and $a_r=a_{r+1}$.

\begin{ex}
If for instance $\lambda=(2,3)$, then  $p=x_1^3x_2^3x_3^2x_4^6x_5^6\in\Xdomla$ because $3\leqslant 3$ and $2\leqslant 6\leqslant 6$, and $\lambda\cap p=(2,1,2)$.
\end{ex}

Assume $\lambda,\mu\in \calC_d$, $w\in \Drep_{\lambda,\mu}$. Denote by $\lambda\cap w(\mu)$ the unique partition in $\calC_d$ such that $\frakS_{\lambda\cap w(\mu)}=\frakS_\lambda\cap w\frakS_\mu w^{-1}$. (But $w(\mu)$ itself has no sense as a partition. Note also that $\lambda\cap w(\mu)$ has no sense for an arbitrary permutation $w$ that is not an element of $\Drep_{\lambda,\mu}$.)

Recall that for each $(\ell+1)$-composition $\lambda\in \calC^\ell$ we denote by $\overline \lambda$ the associated composition (i.e., the concatenation of the components of $\lambda$). If $\lambda$, $\mu$ and $\nu$ are $(\ell+1)$-compositions in $\calC_d^\ell$, we can also use notation $\Drep_{\lambda,\mu}$, $\Drep^{\nu}_{\mu,\lambda}$, $\lambda\cap\mu$, $\lambda\cap w(\mu)$, $\Xdomla$ etc. instead of $\Drep_{\overline\lambda,\overline\mu}$, $\Drep^{\nu}_{\overline\mu,\overline\lambda}$, $\overline\lambda\cap\overline\mu$, $\overline\lambda\cap w(\overline\mu)$, $\calX^+_{\overline\lambda}$ etc.  (in this situations we just consider each $(\ell+1)$-composition as an associated composition).

For $p\in\Xdom_{(d)}$, denote by $\WW_p$ the stabilizer of $p$ in $\WW_d$. Then the notation $\Drep_{p,\emptyset}$ also makes sense.  
\begin{lem}
\label{lem:basis-Heck}
The set
$$
B=\{T_wpT_z\mid w\in \WW_d,p\in \Xdom_{(d)}, z\in \Drep_{p,\emptyset} \}
$$
is a basis of $\Heck$.
\end{lem}
\begin{proof}
First we show that $B$ spans, that means we prove that the set
$$
B'=\{pT_z\mid p\in \Xdom_{(d)}, z\in \Drep_{p,\emptyset} \}
$$
generates the left $\Hfin_d(q)$-module $\Heck$. 
To do this, it is enough to show that each monomial $p\in \Xint$ can be written as an $\Hfin_d(q)$-linear combination of elements of $B'$. This can be proved by induction using the equality 
$$
p=q^{-1}T_rs_r(p)T_r+(q^{-1}-1)T_r\partial_r(X_rp).
$$

For the linear independence it is enough to check that the elements of $B$ act on the polynomial representation $\bfk[x_1^{\pm 1},\ldots,x_d^{\pm 1}]$ by linearly independent operators. This can be done similarly to the proof of Proposition~\ref{prop-basis-Hdl}. 
\end{proof}

\begin{coro}
\label{coro-two_set_invert}
Fix $\lambda\in \calC_d$. Consider the left $\Hfin_\lambda(q)$-module $\Heck$. Consider two sets in this module:
$$
\Xint\subset \Heck,\qquad\text{and}\qquad\{pT_z\mid p\in \Xdom_\lambda,z\in \Drep^{\lambda}_{\lambda\cap p,\emptyset}\}\subset \Heck.
$$
The elements of the two sets above can be expressed in terms of each other with an invertible change of basis matrix. 
\end{coro}
\begin{proof}
The statement follows from the fact that both of the sets above form bases in the left $\Hfin_\lambda(q)$-module 
$
\opH_{\lambda_1}(q)\otimes\ldots\otimes \opH_{\lambda_{l(\lambda)}}(q).
$
\end{proof}

\begin{rk}
Let $\lambda,\mu\in\calC^\ell_d$ and pick $w\in \Drep_{\mu,\lambda}$ and $z\in  \WW_{\lambda\cap w^{-1}(\mu)}$.  Setting  $z'=wzw^{-1}$ we obtain  the equality $wz=z'w$ and also $T_wT_z=T_{z'}T_w$ in the Hecke algebra $\Heck$. Now, let $\ub,\uc\in J^{\ell,d}$ be such that we have $e^0(\mu)=e(\ub)$ and $e^0(\lambda)=e(\uc)$. We also would like to have the following version of this equality in $\lHeck$ (see see Section~\ref{subs-basis_lHeck} for the notation)
\begin{equation}
\label{eq-lT_{wz}}
T^{\bfb,\bfc}_w T_z=T_{z'}T^{\bfb,\bfc}_w
\end{equation} 
This is slightly delicate, because the element $T^{\bfb,\bfc}_w$ depends on some choices. We can however make these choices in a way such that indeed \eqref{eq-lT_{wz}} holds. To do this, we first choose for each $w\in \Drep_{\mu,\lambda}$ some $T^{\bfb,\bfc}_w$ arbitrarily and then define  $T^{\bfb,\bfc}_y$ for any other $y\in\WW_\mu w \WW_\lambda$ (dependent on these choices) inductively, by induction on the length. 
Assuming we have constructed $T^{\bfb,\bfc}_{y}$ for some $y$ such that $y(\bfc)=\bfb$, then for each simple reflection $s\in \WW_{\lambda}$ such that ${l}(ws)={l}(w){l}(s)$ (resp. for each simple reflection $s'\in \frakS_{\mu}$ such that ${l}(s'w)={l}(s'){l}(w)$) we set $T^{\bfb,s(\bfc)}_{ws}=T^{\bfb,\bfc}_w T_s$ (resp. $T^{s'(\bfb),\bfc}_{s'w}=T_{s'}T^{\bfb,\bfc}_w$).
\end{rk}

\begin{lem}
\label{lem-basis_bimodule}
The set 
\begin{eqnarray*}
\mathcal{B}&=&\left\{T_xT^{\ub,\uc}_wpT_y\mid w\in \Drep_{\mu,\lambda}, x\in \WW_\mu, p\in \Xdom_{\lambda\cap w^{-1}(\mu)}, y\in \Drep^{\lambda}_{\lambda\cap w^{-1}(\mu)\cap p,\emptyset}\right\}
\end{eqnarray*}
is a basis of $e^0(\mu)\lHeck e^0(\lambda)$.
\end{lem}
\begin{proof}

It is a standard fact that each $y\in\frakS_d$ has a unique presentation of the form $y=xwz$, where $w\in \Drep_{\mu,\lambda}$, $x\in \WW_\mu$, $z\in \Drep^{\lambda}_{\lambda\cap w^{-1}(\mu),\emptyset}$ and ${l}(y)={l}(x)+{l}(w)+{l}(z)$. Together with Proposition~\ref{prop-basis-Hdl} this shows that the left $\Hfin_\mu(q)$-module $e^0(\mu)\lHeck e^0(\lambda)$ is free with a basis
\begin{eqnarray*}
\mathcal{B}_1&=&\left\{T^{\ub,\uc}_w T_y p\mid w\in \Drep_{\mu,\lambda}, y\in \Drep^{\lambda}_{\lambda\cap w^{-1}(\mu), \emptyset}, p\in\Xint \right\}
\end{eqnarray*}
or alternatively with a basis
\begin{eqnarray*}
\mathcal{B}_2&=&
\{T^{\ub,\uc}_w pT_y\mid 
w\in \Drep_{\mu,\lambda}, y\in \Drep^{\lambda}_{\lambda\cap w^{-1}(\mu), \emptyset}, p\in\Xint
\}.
\end{eqnarray*}
Indeed, we can find a bijection between $\mathcal{B}_1$ and $\mathcal{B}_2$ such that the base change matrix in an appropriate order on the bases is triangular with invertible elements on the diagonal. For $w\in \Drep_{\mu,\lambda}, y\in \Drep^{\lambda}_{\lambda\cap w^{-1}(\mu), \emptyset}$, we define
\begin{eqnarray*}
\calB_2^{w,y}=\{T^{\ub,\uc}_w pT_y\mid p\in\Xint\},&& \calB_3^{w,y}=\{T^{\ub,\uc}_w pT_{zy}\mid p\in\Xdom_{\lambda\cap w^{-1}(\mu)},z\in \Drep^{\lambda\cap w^{-1}(\mu)}_{{\lambda\cap w^{-1}(\mu)}\cap p,\emptyset}\}.
\end{eqnarray*}
By Corollary~\ref{coro-two_set_invert}, the elements of the sets $\calB_2$ and $\calB_3$ can be written as $\Hfin_{w(\lambda)\cap \mu}(q)$-linear combinations of each other with an invertible change of basis matrix.

Since $\calB_2=\coprod_{w,y}\calB_2^{w,y}$ is a basis of the left $\Hfin_\mu(q)$-module $e^0(\mu)\lHeck e^0(\lambda)$ so is $\calB_3=\coprod_{w,y}\calB_3{w,y}$. The set $\calB_3$ can be written in a slightly different way as
\begin{eqnarray*}
\calB_3&=&\{T^{\ub,\uc}_wpT_y\mid w\in \Drep_{\mu,\lambda}, p\in \Xdom_{\lambda\cap w^{-1}(\mu)}, y\in \Drep^{\lambda}_{\lambda\cap w^{-1}(\mu)\cap p,\emptyset}\}.
\end{eqnarray*}
This implies that $\calB$ is a basis of the vector space $e^0(\mu)\lHeck e^0(\lambda)$.
\end{proof}
For each $w\in \Drep_{\mu,\lambda}$ and $p\in \Xdom_{\lambda\cap w^{-1}(\mu)}$ 
consider the element 
\begin{eqnarray*}
b^{w,p}\in\Hom(m_\lambda\lHeck,m_\mu\lHeck)&&m_\lambda h\mapsto m_\mu T^{\ub,\uc}_{w}p(\sum_{y} (-q)^{r-{l}(y)}T_y) h,
\end{eqnarray*}
where $y$ runs through $\Drep^{\lambda}_{\lambda\cap w^{-1}(\mu)\cap p,\emptyset}$ and $r$ denotes the length of the longest element therein.

\begin{coro}
\label{Corbasis}
The following is a basis of $\Hom_{\lHeck}(m_\lambda\lHeck,m_\mu\lHeck)$
$$
\{b^{w,p}\mid w\in \Drep_{\mu,\lambda},p\in \Xdom_{\lambda\cap w^{-1}(\mu)}\}.
$$
\end{coro}
\begin{proof}
We have seen in \eqref{eq-Hom_via_sing} that $\Hom_{\lHeck}(m_\lambda\lHeck,m_\mu\lHeck)$ is in bijection with the vector subspace of elements of $\epsilon_\mu\otimes_{\Hfin_\mu(q)}e^0(\mu)\lHeck e^0(\lambda)$ on which $\Hfin_\lambda(q)$ acts from the right by the sign representation. By Lemma~\ref{lem-basis_bimodule}, the right $\Hfin_\lambda(q)$-module $\epsilon_\mu\otimes_{\Hfin_\mu(q)}e^0(\mu)\lHeck e^0(\lambda)$ is a direct sum of submodules $M_{w,p}$,  for $w\in \Drep_{\mu,\lambda}$ and $p\in \Xdom_{\lambda\cap w^{-1}(\mu)}$, with vector space basis 
$$
\{\epsilon_\mu \otimes T^{\bfb,\bfc}_w p T_y\mid y\in \Drep^{\lambda}_{\lambda\cap w^{-1}(\mu)\cap p}\}.
$$
We claim that the vector subspace of vectors of $M_{w,p}$ that transform as a sign representation of $\Hfin_\lambda(q)$ is one-dimensional . Indeed, the right $\Hfin_\lambda(q)$-module $M_{w,p}$ is isomorphic to $\epsilon_\xi\otimes_{\Hfin_\xi(q)} \Hfin_\lambda(q)$, where $\xi=\lambda\cap w^{-1}(\mu)\cap p$. An element of $\epsilon_\xi\otimes_{\Hfin_\xi(q)} \Hfin_\lambda(q)$ can be written uniquely in the form $\sum_{y\in \Drep^\lambda_{\xi,\emptyset}}a_y(\epsilon_\xi\otimes T_y)$, where $a_y\in \bfk$. This element transforms as a sign representation of $\Hfin_\lambda(q)$ if an only if for each $i$ we have 
$$
\left(\sum_{y\in \Drep^\lambda_{\xi,\emptyset}}a_y(\epsilon_\xi\otimes T_y)\right)T_i=-\left(\sum_{y\in \Drep^\lambda_{\xi,\emptyset}}a_y(\epsilon_\xi\otimes T_y)\right).
$$
Standard computation shows that this is equivalent to the condition
$-a_y=qa_{ys_i}$ whenever $y,ys_i\in \Drep^\lambda_{\xi,\emptyset}$ with ${l}(ys_i)>{l}(y)$.
But this condition is simply equivalent to the fact that the element is proportional to $\sum_{y\in \Drep^\lambda_{\xi,\emptyset}}\epsilon_\xi\otimes (-q)^{r-{l}(y)}T_y$,
where $r$ is the length of the longest element of $\Drep^\lambda_{\xi,\emptyset}$. Under the isomorphism \eqref{eq-Hom_via_sing} this corresponds to the basis element $b^{w,p}$.
\end{proof}

We can write the morphism $b^{w,p}$ as a composition as follows:
\begin{eqnarray*}
m_\lambda\lHeck\;\stackrel{b^{1,1}}{\longrightarrow}\;m_{\lambda\cap w^{-1}(\mu)}\lHeck&\stackrel{b^{1,p}}{\longrightarrow}&m_{\lambda\cap w^{-1}(\mu)}\lHeck\\ &\stackrel{b^{w,1}}{\longrightarrow}&m_{w(\lambda)\cap \mu}\lHeck\;\stackrel{b^{1,1}}{\longrightarrow}\;m_{\mu}\lHeck.
\end{eqnarray*}

Note that the first and the last morphisms in this decompositions are obviously a split and a merge, the morphism $b^{1,p}$ is a multiplication by a polynomial, whereas, $b^{w,1}$ is a composition of left, right and black crossings. The discussion above together with Lemma~\ref{lem-black_cross} proves the following lemma.

\begin{lem}
\label{lem-generate_lS}
The algebra $\lS$ is generated by the idempotents $e(\lambda)$, for $\lambda\in\calC_d^\ell$, the splits, the merges, the left/right crossings and the polynomials.
\end{lem}

\subsection{Completion}
This section is very similar to \cite[Sec.~5]{MS}.  As in Section~\ref{subs-compl-Hecke}, we fix $\bfa\in (\bfk^*)^\ell$. The affine Schur algebra considered in \cite{MS} corresponds to the case $\ell=0$ (no red lines). But the completion procedure only does something with black lines.

We consider $m_\lambda\clHeck$ as a right $\clHeck$-module.
\begin{df}
We set $\clS=\End_{\clHeck}\left(\bigoplus_{\lambda\in \calC^\ell_d}m_\lambda \clHeck\right)$.
\end{df}

As for Hecke algebras, the affine Schur algebra gets more idempotents after completion. They can be constructed in the following way. For each $\ui\in \frakS_d\bfa$, we have an idempotent $e(\lambda,\ui)=\sum_{\uj\in\frakS_\lambda\ui}e(\uj)\in\lHeck$. It is clear that $e(\lambda,\ui)$ depends only on the $\frakS_\lambda$-orbit of $\ui$. Similarly to \cite[Lemma~5.3]{MS}, the idempotent $e(\lambda,\ui)$ commutes with $m_\lambda$. Then we obtain
\begin{eqnarray*}
\clS&=&\End_{\clHeck}(\bigoplus_{\lambda\in \calC^\ell_d,\ui\in \frakS_\lambda\backslash\frakS_d\bfa}e(\lambda,\ui)m_\lambda \clHeck).
\end{eqnarray*}
In particular, $\clS$ has idempotents $e(\lambda,\ui)$ projecting to $e(\lambda,\ui)m_\lambda \clHeck$.  

\begin{rk} It is possible to give an equivalent definition of $\clS$ as a completion of $\lS$ with respect to some sequence of ideals (see \cite[Sec.~5.1]{MS}, where this is done for $\ell=0$). In particular, this realizes $\lS$ is a subalgebra of $\clS$. The idempotent $e(\lambda)\in \lS$ is decomposed in $\clS$ as $e(\lambda)=\sum_{\ui\in \frakS_\lambda\backslash\frakS_d\bfa} e(\lambda,\ui)$.
\end{rk}

\subsection{Generators of $\clS$}
Let $\lambda,\mu\in\calC^\ell_d$ be such that $\mu$ is a split of $\lambda$. Fix $\ui\in\frakS_d\bfa$. Then we can define the following elements of $\clS$: 
\begin{equation*}
\begin{array}{lllcl}
\text{the {\it split} element}:&&(\lambda,\ui)\to(\mu,\ui)&=&e(\mu,\ui)(\lambda\to\mu)e(\lambda,\ui),\\
\text{the {\it merge} element}:&&(\mu,\ui)\to(\lambda,\ui)&=&e(\lambda,\ui)(\mu\to\lambda)e(\mu,\ui),
\end{array}
\end{equation*}
where $\lambda\to\mu$ and $\mu\to\lambda$ are the images of the usual split and merge with respect to the inclusion $\lS\subset \clS$.

If now $\mu$ is obtained from $\lambda$ by a left crossing, then we define the left 
$(\lambda,\ui)\to(\mu,\ui)$ respectively right crossing $(\mu,\ui)\to(\lambda,\ui)$ in the same way as for split and merges.

\begin{prop}
The algebra $\clS$ acts faithfully on 
\begin{eqnarray*}
\clSPol=\bigoplus_{\lambda\in \calC^\ell_d,\ui\in \frakS_\lambda\backslash\frakS_d\bfa} \bfk[[x_1-i_1,\ldots,x_d-i_d]]^{\frakS_{\lambda,\ui}}e(\lambda,\ui), 
\end{eqnarray*}
where $\frakS_{\lambda,\ui}$ is the stabilizer of $\ui$ in $\frakS_\lambda$. 
\end{prop}
\begin{proof}
This can be proved as \cite[Prop.~5.18]{MS}.
\end{proof}

\subsection{Modified representation of $\lS$}
We now construct a modification of the representation of $\lS$ in $\lSPol$ which will be relevant later, see Remark~\ref{rk:modified}. 

Assume $\lambda\in \calC^\ell_d$. Let $\q'_\lambda$ be the polynomial such that $\q_\lambda \q'_\lambda=\q_d$. (In other words, we have $\q_\lambda=\prod_{1\leqslant i<j\leqslant d}(x_j-qx_i)$, where the product is taken only over $i$ and $j$ that are in different components of $\lambda$.) Note that this notation is a generalization of $\q'_{a,b}$ used above.

\begin{df}
Let $\lSPol'$ be equal to $\lSPol$ as a vector space, but  equipped with a different action of $\lS$. In this new action the element $x\in \Hom(\lambda, \mu)\subset\lS$ acts on $\lSPol'$ as 
$
(\q'_\mu)^{-1} x \q'_\lambda
$
on $\lSPol$.
\end{df}

A priori, the action of $\lS$ defined above is only well-defined on some localization of $\lSPol'$ (not on $\lSPol'$ itself). But it can be checked on generators (idempotents, polynomials, splits, merges, left and right crossings) that this action is also well-defined on $\lSPol'$. The following lemma describes this action.
\begin{lem}
\begin{enumerate}
\item The idemponents $e(\lambda)$, the  ($\frakS_\lambda$-symmetric) Laurent polynomials, and the left and right crossings in $\lS$ act on $\lSPol'$ in the same way as on $\lSPol$.

\item Let  $\mu$ be a split of $\la$. Then in case $\lambda=(a+b)$ and $\mu=(a,b)$, the split map $\la\rightarrow\mu$ acts by sending $fe(\lambda)\in P^{\frakS_\lambda}e(\lambda)$ to $fe(\mu)$, whereas the merge map acts by sending $fe(\mu)\in P^{\frakS_\mu}$ to $D_{a,b}(\q'_{a,b}f)e(\lambda)$  with $a+b=d$. In the general case split and merge act by the same formulae but in the variables from the two blocks of $\mu$ that form one block of $\lambda$. 
\end{enumerate}
\end{lem}
\begin{proof}
The statement follows directly from Proposition~\ref{prop-polrep-lS}.
\end{proof}

The faithfulness of the representation $\lSPol$ implies the faithfulness of the representation $\lSPol'$. A modification $\clSPol'$ of the faithful representation $\clSPol$ of $\clS$ can be defined similarly.

\section{(Higher level) Quiver Schur algebras $\lA$}
\label{sec-QSchur}
\subsection{Quiver Schur algebras}

In this section we restrict the form of the quiver $\Gamma=(I,A)$. We assume that the quiver $\Gamma$ has no loops and each vertex of the quiver has exactly one incoming arrow and exactly one outgoing arrow. (This assumption means that each connected component of the quiver is either an oriented cycle of length $\geqslant 2$ or an infinite oriented chain.) Note that the quiver $\Gamma_\calF$ in Section~\ref{subs-isom_lHeck-tens-compl} always satisfies this assumption. We make this assumption here, because the quiver Schur algebra is defined in \cite{SW} only for type $A$, although the definition from \cite{SW} could easily be generalized, but this is not our focus here.

As above, we fix $\nu\in I^d$ and $\bfQ\in I^\ell$. We first recall the definition of the quiver Schur algebra $\lA$, introduced by the second author and Webster in \cite{SW}. 

For each $\lambda\in \calC^\ell_d$ and $\ui\in I^\nu$, let $\frakS_{\lambda,\ui}$ be the stabilizer of $\ui$ in $\frakS_\lambda$, and let $\calC^\ell_{\nu}$ the set of pairs $(\lambda,\ui)$ such that $\lambda\in \calC^\ell_d,\ui\in \frakS_\lambda\backslash I^\nu$.
Consider the following vector space
\begin{eqnarray}
\label{sPol}
\PolA&=&\bigoplus_{(\lambda,\ui)\in \calC^\ell_\nu} \bfk[y_1,\ldots,y_d]^{\frakS_{\lambda,\ui}}e(\lambda,\ui).
\end{eqnarray}

\begin{rk}
Note that if $\ui,\uj\in I^\nu$ are in the same $\frakS_\lambda$-orbit, and $w$ is an element of $\frakS_\lambda$ such that $w(\ui)=\uj$, then we have a canonical isomorphism 
$$
\bfk[y_1,\ldots,y_d]^{\frakS_{\lambda,\ui}}\simeq \bfk[y_1,\ldots,y_d]^{\frakS_{\lambda,\uj}}, \quad P(y_1,\ldots,y_d)\mapsto P(y_{w(1)},\ldots,y_{w(d)}).
$$ 
This shows that $\PolA$ is well-defined.
\end{rk}

The following was introduced in \cite{SW}.
\begin{df}
\label{def-QS}
The {\it quiver Schur algebra} $\lA$ 
is the subalgebra of $\End(\PolA)$ generated by the following endomorphisms.
\begin{itemize}
\item The {\it idempotents}: $e(\lambda,\ui)$ for $(\lambda,\ui)\in \calC^\ell_\nu$, \\
defined as the projection onto the summand $\bfk[y_1,\ldots,y_d]^{\frakS_{\lambda,\ui}}e(\lambda,\ui)$.
\item The {\it polynomials}: $Pe(\lambda,\ui)$ for any $(\lambda,\ui)\in \calC^\ell_\nu$ and $P\in \bfk[y_1,\ldots,y_d]^{\frakS_{\lambda,\ui}}$, \\
defined as multiplication by $P$ on the summand $\bfk[y_1,\ldots,y_d]^{\frakS_{\lambda,\ui}}e(\lambda,\ui)$ (and by zero on other summands).
\item The {\it split}: $(\lambda,\ui)\to (\mu,\ui)$ for any $(\lambda,\ui), (\mu,\ui)\in\calC^\ell_\nu$ (the $d$-tuple $\ui\in I^\nu$ is the same for both pairs) such that $\mu$ is a split of $\lambda$ in the component $\lambda^{(r)}$ at position $j$. It acts non-trivially only on the component $\bfk[y_1,\ldots,y_d]^{\frakS_{\lambda,\ui}}e(\lambda,\ui)$ and we have there (in the notation from Definition~\ref{def-multicomp})
$$
fe(\lambda,\ui)\mapsto fe(\mu,\ui).
$$
\item The {\it merge}: $(\mu,\ui)\to (\lambda,\ui)$ for any $(\lambda,\ui)$ and $(\mu,\ui)$, as above. It acts non-trivially only on the component $\bfk[y_1,\ldots,y_d]^{\frakS_{\lambda,\ui}}e(\mu,\ui)$. There it acts by 
$$
fe(\mu,\ui)\mapsto(\prod_{i\in I}D_{a_i,b_i})\left (\prod_{n\in [\mu^{(r)}_j],m\in[\mu^{(r)}_{j+1}]}(y_n-y_m)\right)fe(\lambda,\ui),
$$
where the Demazure operator $D_{a_i,b_i}$ is defined as in Section~\ref{subs_Demazure} with respect to the $a_i+b_i$ polynomial variables $y_r$ with indices $r\in[\mu^{(r)}_j]\cup[\mu^{(r)}_{j+1}]$  such that $i_r=i$ and the product is taken only by the indices $n,m$ such that we have $i_n\to i_m$. Hereby $a_i$ (resp. $b_i$) denotes the number of occurrence of $i$ in $\ui$ in the indices in $[\mu^{(r)}_j]$ (resp. $[\mu^{(r)}_{j+1}]$). 
\item The {\it left crossing}: $(\lambda,\ui)\to (\mu,\ui)$ for any $(\lambda,\ui), (\mu,\ui)\in\calC^\ell_\nu$ such that $\mu$ is a left crossing of $\lambda$, 
defined as $fe(\lambda,\ui)\mapsto fe(\mu,\ui)$.
\item The {\it right crossing}: $(\mu,\ui)\to (\lambda,\ui)$ for any $(\lambda,\ui), (\mu,\ui)\in\calC^\ell_\nu$ such that $\lambda$ is a right crossing of $\mu$, moving the last component of $\mu^{(r)}$ to the first of $\mu^{(r+1)}$, is defined as
$
fe(\mu,\ui)\mapsto (\prod_{n\in [\lambda^{(r+1)}_1],i_n=Q_{r+1}}y_n)fe(\lambda,\ui).
$
\end{itemize}
\end{df}

\begin{rk}
\label{rk:modified}
The definition of $\lA$ differs slightly from the original definition in \cite{SW}. The difference is that the multiplication by the Euler class is moved from the split to the merge and the Euler class is also reversed. The two  algebras are however isomorphic, as proved (with an explicit isomorphism) in \cite[Sec.~9.2-9.3]{MS} for $\ell=0$. The arguments directly generalize to arbitrary $\ell$. Passing to this modified quiver Schur algebra is necessary to identify the completion of the algebra $\lA$ with the completion of the algebra $\lS$ via identification of the polynomial representations. This approach does not work if we use the polynomial representation of $\lA$ considered in \cite{SW}. The modification $\lSPol'$ of $\lSPol$ was defined for the same reason. For a geometric interpretation of $\lA$ we refer to \cite{Tomasz}.
\end{rk}

It is possible to introduce a diagrammatic calculus for $\lA$ similarly to the diagrammatic calculus for $\lS$ (see \cite{SW} for more details). The only difference is that black strands in the diagrams for $\lA$ have labels in $\bbZ_{\geqslant 0} I$ instead of $\bbZ_{> 0}$ (here $\bbZ_{\geqslant 0} I$ is the set of formal $\bbZ_{\geqslant 0}$-linear combinations of elements of $I$). 

We draw the idempotent $e(\lambda,\ui)\in \lA$ by the same diagram as the idempotent $e(\lambda)\in \lS$, except that we replace each integer label $\lambda_r^{(t)}$ on a black strand by the label $\sum_{j\in [\lambda_r^{(t)}]}i_j\in \bbZ_{\geqslant 0} I$. We draw polynomials, splits, merges, left and right crossings in $\lA$ in the same way as for $\lS$.

Let $\clA$ be the completion of $\lA$ with respect to the ideal generated by the homogeneous polynomials of positive degrees. The definitions give rise to the following completed version of the faithful representation \eqref{sPol} of $\lA$.

\begin{lem}
The algebra $\clA$ has a faithful representation in 
$$
\cPolA=\bigoplus_{(\lambda,\ui)\in \calC^\ell_\nu}\bfk[[y_1,\ldots,y_d]]^{\frakS_{\lambda,\ui}}e(\lambda,\ui).
$$
\end{lem}

\subsection{The isomorphisms $\clS\simeq \clA$} 

Fix $q\in\bfk$ such that $q\not\in\{0,1\}$. %If $q$ is not the root of unity set $e=0$. Else, let $e$ be the order of $q$. 
Fix an $\ell$-tuple $\bfQ=(Q_1,\ldots,Q_\ell)\subset (\bfk^*)^\ell$.
As in Section~\ref{subs-isom_lHeck-tens-compl}, we consider the quiver $\Gamma_\calF$ with the vertex set
$$
\calF=\{q^nQ_r\mid n\in\bbZ,r\in[1;\ell]\}\subset\bfk^*.
$$
and consider the algebra $\lA$ defined with respect to this quiver. We take $\nu=\bfa$.

We know, that $\clA$ acts  faithfully on $\cPolA$ and $\clS$ acts faithfully on $\clSPol'$. On the other hand, there is an obvious isomorphism of algebras
\begin{eqnarray*}
\cPolA\simeq \clSPol',&& P(-i_1y_1,\ldots,-i_dy_d)e(\lambda,\ui)\mapsto P(x_1-i_1,\ldots,x_d-i_d)e(\lambda,\ui).
\end{eqnarray*}

To prove that the algebras $\clA$ and $\clS$ are isomorphic, it is enough to identify their actions on $\cPolA\simeq \clSPol'$. As a result obtain such an isomorphism:
\begin{thm}
\label{thm-isom-qS-QS-comp}
There is an isomorphism of algebras $\clA\simeq \clS$.
\end{thm}
\begin{proof}
It is clear that the idempotents $e(\lambda,\ui)$ act on the faithful representation in the same way. Obviously, the power series in $\clA$ yield the same operators on the faithful representation as the power series in $\clS$. It remains to match splits, merges and left/right crossings.

Since splits and merges only use black strands, it is enough to treat the case $\ell=0$. This is already done in \cite[Sec.~9]{MS}. It is also easy to see that the left crossings in $\clA$ and $\clS$ act in the same way on the polynomial representations. Indeed, both of them just change the idempotent without changing the power series. 

Let now $\lambda$ be a right crossing of $\mu$, moving the last component of $\mu^{(t)}$ to the first component of $\mu^{(t+1)}$, and fix $\ui$. We compare the actions of the right crossings $(\mu,\ui)\to (\lambda,\ui)$ in $\clA$ and $\clS$.
The right crossing in $\clA$ acts by
$$
P(y_1,\ldots,y_d)e(\mu,\ui)\mapsto\left (\prod_{n\in[\lambda^{(t+1)}_1],i_n=Q_{r+1}}y_n\right)P(y_1,\ldots,y_d)e(\lambda,\ui).
$$
The right crossing in $\clS$ acts by
$$
P(x_1,\ldots,x_d)e(\mu,\ui)\mapsto \left(\prod_{n\in[\lambda^{(t+1)}_1]}(x_n-Q_{t+1})\right)P(x_1,\ldots,x_d)e(\lambda,\ui).
$$
Then it is clear that these operators can be expressed in terms of each other because we can divide by $(x_n-Q_{t+1})$ if $i_n\ne Q_{t+1}$. This proves the theorem.
\end{proof}

\section{Cyclotomic quotients and the isomorphism theorem}
\label{sec-cycl-quot}
We finish by establishing a higher level version of the (cyclotomic) Brundan-Kleshchev-Rouquier isomorphism. 
As above we fix $\bfQ=(Q_1,\ldots,Q_\ell)\in(\bfk^*)^\ell$ and $q\in \bfk^*$, $q\ne 1$ and consider the quiver $\Gamma_\calF$ as in section~\ref{subs-isom_lHeck-tens-compl}. We assume that all KLR algebras and tensor product algebras in this section are defined with respect to the quiver $\Gamma_\calF$. We take $\nu=\bfa$.

\subsection{Cyclotomic $\ell$-Hecke algebras and tensor product algebras}
\label{subs-cycl-lHeck}

\begin{df}
\label{defcylHeck}
The \emph{cyclotomic $\ell$-Hecke algebra} $\cylHeck$ is the quotient of the algebra $\lHeck$ by the ideal generated by the idempotents $e(\bfc)$ such that $\bfc$ is of the form $\bfc=(0,\ldots)$. In other words, we kill all diagrams that have a piece of a black strand on the left of all red strands. 
\end{df}

\begin{lem}
\label{lem-eigenv_H2}
Let $X_1$, $X_2$ and $T$ be three endomorphisms of a vector space $V$, satisfying the relations of $H_2(q)$, i.e.,
$$
\begin{array}{lclcrcl}
X_1T&=&TX_2-(q-1)X_2,&&
(T-q)(T+1)&=&0,\\
X_2T&=&TX_1+(q-1)X_2,&&
X_1X_2&=&X_2X_1.
\end{array}
$$
(We do not assume that $X_1$ and $X_2$ are invertible.)
Let $\lambda_1,\lambda_2\in \bfk^*$ be such that $\lambda_1\ne q^{\pm 1}\lambda_2$. 
Then if $V$ has a simultaneous eigenvector for $X_1$, $X_2$ with eigenvalues $\lambda_1$, $\lambda_2$, then $V$ has also  a simultaneous eigenvector with eigenvalues $\lambda_2$, $\lambda_1$ respectively.
\end{lem}
\begin{proof}
Let $v\in V$, $v\not=0$ such that $X_1(v)=\lambda_1v$ and $X_2(v)=\lambda_2v$. Consider the vector
$
w=(q-1)\lambda_2v+(\lambda_1-\lambda_2)T(v).
$
It follows directly from the relations that  $X_1(w)=\lambda_2w$ and $X_1(w)=\lambda_2w$. Note that $w=0$ implies that $T(v)$ is proportional to $v$. In this case we have either $T(v)=-v$ or $T(v)=qv$ and then $\lambda_2$ must equal $q\lambda_1$ or $q^{-1}\lambda_1$. But this is impossible by the assumptions on $\lambda_1$ and $\lambda_2$.
\end{proof}

\begin{coro}
\label{coro-eigenv_lHeck}
Let $V$ be a finite dimensional representation of $\cylHeck$. Then for each $r\in\{1,2,\ldots,d\}$, all eigenvalues of the action of $x_r$ on $V$ are in $\calF$. 
\end{coro}
\begin{proof}

Assume that some $x_r$ has an eigenvalue $\lambda\notin \calF$. Since $x_r$ is invertible, we have $\lambda\ne 0$. Then there exists an idempotent $e(\uc)\in J^{\ell,d}$ such that $\lambda$ is an eigenvalue of $x_re(\uc)$. (This simply means that $e(\uc)$ does not annihilate the $\lambda$-eigenspace of $x_r$.) Let $t$ be such that $X_te(\uc)=x_re(\uc)$ in $\lHeck$ and set $k=\sum_{i=1}^tc_i$ (i.e., $k$ is the number of red strands to the left of the dot in the diagram of $x_re(\uc)$). We assume that the index $t$ as above is as minimal as possible (for all possible $r$ and $\bfc$). We clearly have $t>1$, because $X_1=0$ in $\cylHeck$ and $\lambda\ne 0$.  

Assume $c_{t-1}=1$. Let $v$ be an eigenvector of $x_re(\uc)$ with eigenvalue $\lambda$, (in particular $e(\uc)(v)=v$). Then $T_{t-1}(v)\ne 0$. Indeed, we have $T^2_{t-1}e(\uc)=(X_t-Q_k)e(\ui)$. This implies
$
T_{t-1}^2(v)=T_{t-1}^2e(\uc)(v)=(X_t-Q_k)e(\uc)(v)=(\lambda-Q_k)v\ne 0.
$
Moreover, the vector $T_t(v)$ is clearly an eigenvector of $x_re(s_{t-1}(\uc))=X_{t-1}e(s_{t-1}(\uc))$ corresponding to the eigenvalue $\lambda$. This contradicts the minimality of $t$.

Assume $c_{t-1}=0$. Then we can find a vector $v\in V$ such that $v$ is a common eigenvector for $x_{r-1}$ and $x_r$ with $e(\ui)(v)=v$ and $x_r(v)=\lambda v$. Let $\mu$ be such that $x_{r-1}(v)=\mu v$. We have $\mu\ne 0$ because $x_{r-1}$ is invertible. Moreover, the eigenvalue $\mu$ must be in $\calF$ (else, this contradicts the minimality of $t$). Then we can apply Lemma~\ref{lem-eigenv_H2} to $x_{r-1}e(\uc)$, $x_re(\uc)$ and $T_{t-1}e(\uc)$. This shows that $\lambda$ is an eigenvalue of $x_{r-1}e(\uc)=X_{t-1}e(\uc)$. This contradicts the minimality of $t$.
\end{proof}

In $\cylHeck$, we have the idempotents $e(\ui)$ such that $1=\sum_{\ui\in\calF^d}e(\ui)$ and for each index $r$, the element $(x_r-i_r)e(\ui)$ is nilpotent (see Corollary~\ref{coro-eigenv_lHeck}). Moreover, for each $\bfa\in\calF^d$ we have a central idempotent $1_\bfa=\sum_{\bfa\in \frakS_d \bfa}e(\ui)$. Set $\cylHecka=1_\bfa\cylHeck$. Then there is the following direct sum decomposition of algebras $\cylHeck=\bigoplus_{\bfa\in\calF^d}\cylHecka$.

\begin{df}
The \emph{cyclotomic tensor product algebra} $\cylR$ is the quotient of the algebra $\lR$ by the ideal generated by the idempotents $e(\ui)$ such that $\ui\in \Icolnu$ is such that $c(i_1)=0$. In other words, we kill all diagrams that have a piece of a black strand on the left of all red strands. 
\end{df}

It is clear from the definitions that the algebra $\cylHecka$ is a quotient of $\clHeck$ and the algebra $\cylR$ is a quotient of $\clR$. We obtain 

\begin{thm}
\label{prop-isom-lHeck-tens-cycl}
There is an isomorphism of algebras $\cylHecka\simeq \cylR$.
\end{thm}
\begin{proof}
This follows immediately from Theorem~\ref{thm-isom-lHeck-tens-comp}. 
\end{proof}

\subsection{Classical Brundan-Kleshchev-Rouquier isomorphism} 

In this section we show how to deduce from Theorem~\ref{prop-isom-lHeck-tens-cycl} the usual Brundan-Kleshchev-Rouquier isomorphism for cyclotomic KLR and Hecke algebras.

\begin{df}
The \emph{cyclotomic Hecke algebra} $\cyHeck$ is the quotient of the algebra $\Heck$ by the ideal generated by the polynomial $(X_1-Q_1)\ldots(X_1-Q_\ell)$.
\end{df}

For each $\ui=(i_1,\ldots,i_\ell)\in \calF^d$ we have an idempotent $e(\ui)\in\cyHeck$ such that $1=\sum_{\ui\in\calF^d}e(\ui)$ and for each index $r$, the element $(X_r-i_r)e(\ui)$ is nilpotent. Moreover, with $\bfa\in\calF^d$ comes a central idempotent $1_\bfa=\sum_{\bfa\in \frakS_d \bfa}e(\ui)$. Set $\cyHecka=1_\bfa\cyHeck$. There is a direct sum decomposition of algebras $\cyHeck=\bigoplus_{\bfa\in\calF^d}\cyHecka$.

\begin{df}
The \emph{cyclotomic KLR algebra} $\cyR$ is the quotient of the algebra $\R$ by the ideal generated by $y_1^{\Lambda_{i_1}}e(\ui)$. Here,  $\Lambda_i$ the multiplicity of $i\in \calF$ in $\bfQ$. 
\end{df}

Recall the idempotent $e(\omega)\in\lHeck$ such that $e(\omega)\lHeck e(\omega)\simeq \Heck$, see   Section~\ref{subs-centre-lHecke}. We have a similar idempotent $e(\omega)\in\lR$ with $e(\omega)\lR e(\omega)\simeq \R$.

The following is proved in \cite[Thm.~4.18]{Webster}.
\begin{lem}
There is an isomorphism of algebras $e(\omega)\cylR e(\omega)\simeq \cyR$.
\end{lem}

We can prove the following analogue of this statement.
\begin{lem}
There is an isomorphism of algebras $e(\omega)\cylHeck e(\omega)\simeq \cyHeck$.
\end{lem}
\begin{proof}
We will identify $\Heck$ with $e(\omega)\lHeck e(\omega)$ as in Lemma~\ref{lem-Hd_in_Hdl}.

Denote by $K_1$ the kernel of $\Heck\to \cyHeck$. Denote by $K_2$ the kernel of $e(\omega)\lHeck e(\omega)\to e(\omega)\cylHeck e(\omega)$. We have to prove that $K_1=K_2$.

First of all, it is clear that  $K_1\subset K_2$, because we have 

$$
    
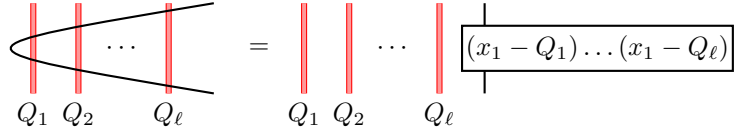
\begin{tikzpicture}[thick, scale=0.6]

      \draw[wei] (6.5,0)  +(-2,-1) -- +(-2,1) node[at start,below]{$Q_1$};
      \draw[wei] (6.5,0) +(-1,-1) -- +(-1,1) node [at start,below]{$Q_2$};

\node at (6.5,0) {$\cdots$}  ;

      \draw[wei] (6.5,0)  +(1,-1) -- +(1,1) node[at start,below]{$Q_\ell$};
   
     \draw  (8.5,-1) .. controls (2.5,0) ..  (8.5,1);
     
     \node at (9.5,0) {$=$};
     
     \draw[wei] (12.5,0)  +(-2,-1) -- +(-2,1) node[at start,below]{$Q_1$};
     \draw[wei] (12.5,0) +(-1,-1) -- +(-1,1) node [at start,below]{$Q_2$};

\node at (12.5,0) {$\cdots$}  ;

      \draw[wei] (12.5,0)  +(1,-1) -- +(1,1) node[at start,below]{$Q_\ell$};
      \draw (14.5,-1) -- (14.5,-0.5);
      \draw (14.5,0.5) -- (14.5,1);
      \draw (14,-0.5) rectangle (20,0.5);
      \node at (17,0) {$(x_1-Q_1)\ldots (x_1-Q_\ell)$};
 \end{tikzpicture}
$$

Let us show $K_2\subset K_1$. We need to show for each $\uc\in J^{\ell,d}$ such that $c_1=0$, it holds $e(\omega)\lHeck e(\uc)\lHeck e(\omega)\subset K_1$.

Denote by $(\uc\to \omega)$ the unique element of $e(\omega)\lHeck e(\uc)$ that is presented by a diagram that contains right crossings only. Similarly, denote by $(\omega\to \uc)$ the unique element of $e(\uc)\lHeck e(\omega)$ that is presented by a diagram that contains left crossings only. For example, for $\uc=(0,1,0,0,0,1)$, we have
$$
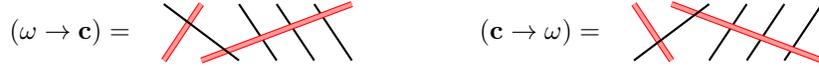
\tikz[thick,xscale=2.5,yscale=1.5]{
\node at (-1.5,.25) {$(\omega\to\uc)=$};
\draw (-0.4,0) --(-0.6,.5);
\draw (-0.2,0) --(-0.4,.5);
\draw (0,0) --(-0.2,.5);
\draw[wei] (-0.8,0) --(0,.5);
\draw[wei] (-1.0,0) --(-0.8,.5);
\draw (-0.6,0) --(-1.0,.5);
\node at (1,.25) {$(\uc\to\omega)=$};
\draw (1.9,0) -- (2.1,0.5);
\draw (2.1,0) -- (2.3,0.5);
\draw (2.3,0) -- (2.5,0.5);
\draw[wei] (2.5,0) -- (1.7,0.5);
\draw[wei] (1.7,0) -- (1.5,0.5);
\draw (1.5,0) -- (1.9,0.5);
}
$$

By Proposition~\ref{prop-basis-Hdl}, each element of $e(\omega)\lHeck e(\uc)$ can be written as $a\cdot(\uc\to\omega)$ with $a\in e(\omega) \lHeck e(\omega)$. Similarly, each element of $e(\uc) \lHeck e(\omega)$ can be written as $(\omega\to\uc)\cdot b$ with $b\in e(\omega) \lHeck e(\omega)$. Then each element of $e(\omega)\lHeck e(\uc)\lHeck e(\omega)$ can be written as $a\cdot(\uc\to\omega)\cdot (\omega\to\uc)\cdot b$. Since $c_1=0$, the element $(\uc\to\omega)\cdot (\omega\to\uc)$ can be written as $e(\omega)P$, where $P\in \bfk[x_1,\ldots,x_\ell]$ is a polynomial divisible by $(x_1-Q_1)\ldots (x_1-Q_\ell)$. This implies $K_2\subset K_1$.
\end{proof}
Consequently, we get the  Brundan-Kleshchev-Rouquier isomorphism,  \cite[Thm.~1.1]{BKKL}, \cite[Cor.~3.20]{Rou2KM}:

\begin{coro}
\label{prop-isom-Heck-KLR-cycl}
There is an isomorphism of algebras $\cyHecka\simeq \cyR$.
\end{coro}

\subsection{The DJM $q$-Schur algebra}
We establish now a connection with the cyclotomic $q$-Schur algebra $\DJM$ defined in \cite{DJM}. Denote by $\calC^{0,\ell}_d$ the subset of $\calC^\ell_d$ that contains all $\lambda$ such that $\lambda^{(0)} = 0$ (here $0$ is the unique (empty) composition of $0$). 

For each $\lambda\in{\calC^{\ell}_d}$, set $u_\lambda=\prod(X_r-Q_t)\in \cyHeck$, where the product is taken over all indices $r$ and $t$ such that $r\leqslant |\lambda^{(0)}|+\ldots+|\lambda^{(t-1)}|$. 

\begin{ex}
For example, for $\ell=3$ and $\lambda=(0,(1,1),(2),(1,2))$, we have
$$
|\lambda^{(0)}|=0,\quad |\lambda^{(1)}|=2,\quad|\lambda^{(2)}|=2,\quad |\lambda^{(3)}|=3
$$
and
$
u_\lambda=(X_1-Q_2)(X_2-Q_2)(X_1-Q_3)(X_2-Q_3)(X_3-Q_3)(X_4-Q_3).
$
\end{ex}

We consider $u_\lambda n_\lambda \cyHeck$ as a right $\cyHeck$-module.

\begin{df}
The Dipper-James-Mathas cyclotomic $q$-Schur algebra $\DJM$ is the algebra
\begin{eqnarray*}
\DJM&=&\End_{\cyHeck}(\bigoplus_{\lambda\in\calC^{\ell}_d}u_\lambda n_\lambda \cyHeck).
\end{eqnarray*}
\end{df}
\begin{rk}
The algebra $\DJM$ is defined in \cite{DJM} with respect to the set $\calC^{0,\ell}_d$ instead of $\calC^{\ell}_d$. But there is no difference because,  $u_\lambda=0$ in $\cyHeck$ if $\lambda\in \calC^\ell_d\backslash \calC^{0,\ell}_d$. Indeed, note that if $\lambda\in \calC^\ell_d\backslash \calC^{0,\ell}_d$, then $(X_1-Q_1)\ldots (X_1-Q_\ell)$ divides $u_\lambda$. This means that $u_\lambda=0$ in $\cyHeck$. 
\end{rk}
\begin{lem}
\label{lem:qSchur-from-lHeck}
There is an isomorphism of algebras
\begin{eqnarray*}
\DJM\simeq \End_{\cylHeck}\left(\bigoplus_{\lambda\in \calC^{\ell}_d}n_\lambda \cylHeck\right).
\end{eqnarray*}
\end{lem}
\begin{proof}
A similar description of the $q$-Schur algebra is given in \cite[(5.8)]{SW}. To get the statement we only need to identify $\cylHeck$ with $\cylRd$, where $\cylRd=\bigoplus_{\nu\in \frakS_d\backslash\calF^d}\cylR$.
\end{proof}

\subsection{The Schur version}
In this section we give the most general version of the isomorphism above: the (higher level) Schur version. 

We consider $m_\lambda \cylHeck$ as a right $\cylHeck$-module.
\begin{df}
The \emph{cyclotomic $q$-Schur algebra} $\cylS$ is the algebra 
\begin{eqnarray*}
\cylS &=& \End_{\cylHeck}(\bigoplus_{\lambda\in \calC^{\ell}_d}m_\lambda \cylHeck).
\end{eqnarray*}
\end{df}
It is clear from the definition that the algebra $\cylS$ is a quotient of $\lS$.
\begin{rk}
\label{rk-cyS-DJM}
By Lemmas~\ref{lem-isom-Sbar-Sop} and~\ref{lem:qSchur-from-lHeck}  we have $\cylS\simeq \DJMop$ as algebras.
\end{rk}

Similarly to the set $\calC_\nu^\ell$ defined above, we denote by $\calC_\ba^\ell$ the set of pairs $(\lambda,\ui)$, where $\lambda\in\calC_d^\ell$ and $\ui\in \frakS_\lambda\backslash \frakS_d\ba$. The algebra $\cylS$ contains idempotents $e(\lambda,\ui)\in\cylS$ such that $1=\sum_{(\lambda,\ui)\in\calC^\ell_\ba}e(\lambda,\ui)$ and such that for each Laurent polynomial $P(x_1,\ldots,x_d)\in \bfk[x^{\pm 1}_1,\ldots,x^{\pm 1}_d]^{\frakS_\lambda}$, the element $(P(x_1,\ldots,x_d)-P(i_1,\ldots,i_d))e(\lambda,\ui)$ is nilpotent. Moreover, for each $\bfa\in\calF^d$ we have a central idempotent $1_\bfa=\sum_{(\lambda,\ui)\in \calC^\ell_\bfa}e(\lambda,\ui)$. Set $\cylSa=1_\bfa\cylS$. We have the following direct sum decomposition of algebras $\cylS=\bigoplus_{\bfa\in\calF^d}\cylSa$.

\begin{df}
The \emph{cyclotomic quiver Schur algebra} $\cylA$ is the quotient of the algebra $\lA$ by the ideal generated by the idempotents of the form $e(\lambda,\ui)$ such that $l(\lambda^{(0)})\ne 0$. In other words, we kill all diagrams that have a piece of a black strand on the left of all red strands. 
\end{df}

It is clear from the definitions that the algebra $\cylSa$ is a quotient of $\clS$ and the algebra $\cylA$ is a quotient of $\clA$. Theorem~\ref{thm-isom-qS-QS-comp} implies the following:

\begin{prop}
\label{prop-isom-qS-QS-cycl}
There is an isomorphism of algebras $\cylSa\simeq \cylA$.
\end{prop}
\begin{proof}
It is clear from the definitions that for each $\lambda\in \calC^\ell_d$ such that $l(\lambda^{(0)})\ne 0$, the idemponent $e(\lambda)$ is in the kernel of $\lS\to \cylS$. This implies that the isomorphism $\clS\simeq \clA$ in Theorem~\ref{thm-isom-qS-QS-comp} yields a surjective homomorphism $\cylA\to \cylSa$. To prove that this is an isomorphism, it is enough to show that these algebras have the same dimensions. 
We have
$$
\dim(\cylSa)=\dim(\op{S}^{DJM}_{\bfa,\bfQ^{-1}}(q))=\dim(A^{\bfQ^{-1}}_{\nu,\bfQ^{-1}})=\dim(\cylA).
$$
The first equality holds by Remark~\ref{rk-cyS-DJM},  the second by \cite[Thm.~6.2]{SW}, and the third since the quivers $\Gamma_\calF$ defined with respect to $\bfQ$ and $\bfQ^{-1}$ are isomorphic.
\end{proof}

\bibliographystyle{abbrv}
\bibliography{references}
\end{document}